\documentclass[11pt]{amsart}
\usepackage{amsfonts,amssymb}
\usepackage{amsmath}
\usepackage{amsthm,amscd,latexsym}
\usepackage{mathrsfs}

\newtheorem{theorem}{Theorem}[section]
\newtheorem{corollary}[theorem]{Corollary}
\newtheorem{lemma}[theorem]{Lemma}
\newtheorem{proposition}[theorem]{Proposition}

\newtheorem*{claim}{Claim}

\theoremstyle{definition}
\newtheorem{definition}{Definition}[section]
\newtheorem{remark}{Remark}[section]
\newtheorem{example}{Example}[section]

\DeclareMathOperator{\argmin}{\mathrm{arg\, min}}

\DeclareMathOperator{\End}{\mathrm{End}}

\begin{document}

\title[]{Semigroups of operator means and generalized Karcher equations}
\author[Mikl\'os P\'alfia] {Mikl\'os P\'alfia}
\address{Department of Mathematics, College of Science, Sungkyunkwan University, Suwon 440-746, Republic of Korea}
\email{palfia.miklos@aut.bme.hu}

\subjclass[2000]{Primary 30C45,47A64,53A15,53B05,53C35}
\keywords{operator monotone function, starlike function, operator mean, matrix mean, affine connection}

\dedicatory{Dedicated to prof. D\'enes Petz and prof. Istv\'an Vajk}

\date{\today}

\begin{abstract}
In this article we consider means of positive bounded linear operators on a Hilbert space. We extend the theory of matrix power means to arbitrary operator means in the sense of Kubo-Ando. The basis of the extension is relying on ideas coming from differential geometry. We consider generalized Karcher equations for positive operators and show that such equations admit unique positive solutions that can be obtained as a limit of one-parameter families of operator means called induced operator means. These means are themselves unique fixed points of one parameter families of strict contractions induced, through Kubo-Ando theory of operator means, by semigroups of holomorphic functions mapping the upper half-plane into itself. These semigroups of holomorphic functions are considered with Koenigs function corresponding to Schroeder's functional equation. Koenigs function in this setting provides us with a logarithm map corresponding to every 2-variable operator mean. The semigroups of 2-variable means behave as geodesics and we exactly classify the cases when they are indeed geodesics of affinely connected manifolds, thereby providing the cases when these generalized Karcher equations are exactly Karcher equations in the geometric sense. This is achieved by studying the arising holonomy groups. The unique solutions of these generalized Karcher equations are called lambda extensions and have numerous desirable properties which are inherited from the induced operator means themselves.

\end{abstract}

\maketitle

\section{Introduction}
Let $E$ be Hilbert space and $S(E)$ denote the Banach space of bounded linear self-adjoint operators. Let $\mathbb{P}\subseteq S(E)$ denote the cone of positive definite operators on $E$. In this article we are concerned with means of members of $\mathbb{P}$ that enjoy certain attractive properties that recently became important from the point of view of averaging in the finite dimensional case, see for example \cite{AF,fillard,karcher,Ba08,BBC}. Usually the main difficulties here arise from the required property of operator monotonicity, i.e. our means must be monotone with respect to the positive definite order on $\mathbb{P}$. The 2-variable theory of such functions is relatively well understood, each such function is represented by an operator monotone function according to the theory of Kubo-Ando \cite{kubo}, however in the several variable case we have no such characterization of operator monotone functions.

When $E$ is finite dimensional, then there are additional geometrical structures that are used to define certain n-variable mappings on $\mathbb{P}$ that are regarded as several variable operator means \cite{ando,moakher,lawsonlim}. In this setting $\mathbb{P}$ is just the cone of positive n-by-n Hermitian matrices denoted by $\textit{P}(n,\mathbb{C})$. It is a smooth manifold as an open subset of the vector space of n-by-n Hermitian matrices $\textit{H}(n,\mathbb{C})$ (which is just $S(E)$ in this case) and has a Riemannian symmetric space structure $\textit{P}(n,\mathbb{C})\cong\textit{GL}(n,\mathbb{C})/\textit{U}(n,\mathbb{C})$, where $\textit{U}(n,\mathbb{C})$ is the unitary group \cite{bridsonhaefliger}. This symmetric space is nonpositively curved, hence a unique minimizing geodesic between any two points exists. The midpoint operation on this space, which is defined as taking the middle point of the geodesic connecting two points, is the geometric mean of two positive definite matrices \cite{bhatia2}. The multivariable geometric mean or Karcher mean of the k-tuple ${\Bbb A}:=(A_1,\ldots,A_k)\in\textit{P}(n,\mathbb{C})^k$ is defined as the center of mass
\begin{equation}\label{riemmean}
\Lambda(w_1,\dots,w_k;A_{1},\dots,A_{k})=\underset{X\in \textit{P}(n,\mathbb{C})}{\argmin}\sum_{i=1}^k w_i\delta^2(X,A_i).
\end{equation}
on the Riemannian manifold $\textit{P}(n,\mathbb{C})$ endowed with the trace metric
\begin{equation*}
d(A,B)=\sqrt{Tr\log(A^{-1}B)}
\end{equation*}
with respect to the positive probability vector $\omega:=(w_1,\ldots,w_k)$. The Karcher mean $\Lambda(\omega;{\Bbb A})$ is also the unique positive definite solution of the Karcher equation
\begin{equation}\label{eq:intr1}
\sum_{i=1}^k w_i\log(X^{-1}A_i)=0
\end{equation}
corresponding to the gradient of the function in the minimization problem \eqref{riemmean}.

Recently Lim and P\'alfia \cite{limpalfia} found a one parameter family of multivariable matrix means called the matrix power means which are defined as the unique positive definite solution of the matrix equation
\begin{equation}\label{eq:intr2}
X=\sum_{i=1}^{n}w_iG_t(X,A_{i})
\end{equation}
where $G_t(A,B)=A^{1/2}\left(A^{-1/2}BA^{-1/2}\right)^tA^{1/2}$, $t\in[0,1]$ is the weighted geometric mean of $A,B\in\textit{P}(n,\mathbb{C})$. An attractive property of this family is that as the defining parameter $t\to 0$, the matrix power means converge to the Karcher mean. Moreover this limiting behavior still holds if $E$ is infinite dimensional \cite{lawsonlim1}, i.e. the operator equation \eqref{eq:intr1} still has a unique positive solution and the unique positive definite solutions of the operator equations \eqref{eq:intr2} for $t\in[0,1]$ still converge to the unique solution of the Karcher equation \eqref{eq:intr1}, although there is no Riemannian metric in the infinite dimensional case. This limiting behavior is used to prove certain nice properties of the Karcher mean, for example operator monotonicity, since the matrix power means have nice properties and these properties are preserved in the limit $t\to 0$. The source of such results can be traced back to the existence of affinely connected geometric structures, like the Riemannian symmetric space structure of $\textit{P}(n,\mathbb{C})$ \cite{ando,moakher,lawsonlim}. For example the well known arithmetic and harmonic means also occur as center of mass operations on, this time, Euclidean spaces, hence also unique solutions of gradient (in other words Karcher) equations.

In this paper we are concerned of extending the above ideas to the case of all possible operator means. One of the main results proved in Section 8 is that an extension of \eqref{eq:intr2} for all possible 2-variable operator means $M(A,B)$ (in the sense of Kubo-Ando)
\begin{equation}\label{eq:intr3}
X=\sum_{i=1}^{n}w_iM(X,A_{i})
\end{equation}
admits a unique positive definite solution in $\mathbb{P}$ defining a mean extension to multiple variables called the induced operator means in Section 9. We achieve this by showing that the map $f(X)=M(A,X)$ is a strict contraction with respect to Thompson's part metric \cite{thompson} on arbitrary bounded subsets of the cone $\mathbb{P}$, hence the map $g(X)=\sum_{i=1}^{n}w_iM(X,A_{i})$ is also a strict contraction, so it has a unique fixed point in $\mathbb{P}$. This result itself, as a byproduct, also proves the open problem that the extension of Ando-Li-Mathias \cite{ando} also works in the infinite dimensional setting $\mathbb{P}$ for all 2-variable operator means using the constructions in \cite{lawsonlim0}. In our setting this contraction property leads to a new multivariable theory of operator means relying on these induced operator means. We show several crucial properties of these induced means, for example operator monotonicity and later our goal is to consider one parameter families of these induced means in Section 10, similarly to the case of matrix power means \eqref{eq:intr2}.

In order to consider those one parameter families we have to construct such families for 2-variable operator means in the sense of Kubo-Ando. In the mentioned cases above of the geometric, arithmetic and harmonic means these one parameter families are naturally derived from the corresponding geometric structures as geodesic lines. In Section 12 we find all possible 2-variable operator means which occur such a way, in other words we classify affine operator means (the problem was raised in \cite{palfia3} and \cite{fujii}). It turns out that these means are exactly the matrix power means. We prove also that the corresponding affine connections are 
\begin{equation*}
\nabla_{X_p}Y_p=DY[p][X_p]-\frac{\kappa}{2}\left(X_pp^{-1}Y_p+Y_pp^{-1}X_p\right)\text{,}
\end{equation*}
where $0\leq\kappa\leq 2$ and the tangent space is $\textit{H}(n,\mathbb{C})$ at every point $p\in \textit{P}(n,\mathbb{C})$. These connections appear earlier when we construct them as prototypes of invariant affine connections in Section 6. In Section 13 among other results we show that these affine connections are non-metric in general, i.e. there exist no other Riemannian (or other metric) structures as in the case of the Karcher mean \eqref{riemmean} in the finite dimensional setting. In order to achieve this we investigate the holonomy groups and other properties of these affine connections.

By knowing that in general there are no geometric structures available for us, we consider instead one parameter semigroups of 2-variable operator means. The idea relies on a geometric construction in Section 3 which can be used over a general affinely connected space to reconstruct the logarithm (hence also the exponential) map of the corresponding affine connection from the midpoint operation $m(p,q)=\exp_p(1/2\log_p(q))$ on the manifold as
\begin{equation*}
\log_p(q)=\lim_{n\to \infty}\frac{m(p,q)^{\circ n}-p}{\frac{1}{2^{n}}}
\end{equation*}
where we use the notation $m(p,q)^{\circ n}\equiv m\left(p,m(p,q)^{\circ (n-1)}\right)$ and $\log_p(q)$ is the logarithm map. We apply an analogue of such a process to 2-variable operator means in Section 4 and we obtain a corresponding "logarithm map"
\begin{equation*}
\log_A(B)=A^{1/2}\log_I\left(A^{-1/2}BA^{-1/2}\right)A^{1/2}
\end{equation*}
of the operator mean, where $\log_I(x)$ is an operator monotone function. We show that $\log_A(B)$, hence $\log_I(x)$ directly induce a one parameter family of operator monotone functions which represent operator means. Since all operator monotone functions are Pick functions, i.e. holomorphic function mapping the upper complex half-plane into itself \cite{bhatia}, this construction is closely related to the classical topic in iteration theory \cite{cowen,kuczma,szekeres} of holomorphic functions coming from Koenigs classical paper \cite{koenigs} written on the problem of solving Schroeder's functional equation \cite{schroeder}
\begin{equation*}
\sigma \circ f_t(z)=t\sigma(z)
\end{equation*}
for a given $f_t(z)$ holomorphic function with isolated attractive fixed point. Then in Section 5 we show that taking directly such operator monotone functions $\log_I(x)$ that can be prototype of logarithm maps, we obtain similar one parameter semigroups of operator means. In the univalent case these families are Loewner semigroups of Pick functions which itself has a classical and rich theory \cite{fitzgerald,fitzgerald2,horn,horn2,horn3,horn4}. We also show that the further extendability to greater parameter values of the one parameter family depends on the distribution of the ramification points of the corresponding logarithm map. We relate this extendability property to functional equations over the upper complex half-plane of the form
\begin{equation}
\log_I(f_t(z))=t\log_I(z),
\end{equation}
where $f_t(z)$ is the representing operator monotone function of the matrix mean and $\log_I(z)$ is the corresponding unique logarithm map. We show that if $\log_I(z)$ has no ramification points in the upper half-plane, then the functional equation, hence the one parameter family $f_t(z)$ is a Pick function, i.e. an operator monotone function for all $t\in[0,1]$.

With these one parameter families of operator means in hand we consider the limit of the corresponding parameter $t\to 0+$ in Section 10. We prove that the one parameter semigroup of induced operator means occurring as unique solutions of operator equations
\begin{equation*}
X=\sum_{i=1}^{n}w_iM_t(X,A_{i})
\end{equation*}
converge as $t\to 0+$ in the strong operator topology and the limit point satisfies, what we call, a generalized Karcher equation
\begin{equation}\label{eq:intr4}
\sum_{i=1}^n w_i\log_X(A_i)=0,
\end{equation}
where $\log_X(A)$ is the "logarithm map" corresponding to the one parameter semigroup of 2-variable operator means $M_t(A,B)$. This limit points are referred to as the lambda extensions of an operator mean. Moreover we prove that these lambda extensions provide the unique positive solutions of \eqref{eq:intr4} and the numerous properties fulfilled by the induced operator means, for example as operator monotonicity, are preserved in the limit, hence also fulfilled by the lambda extensions. These considerations provide our other main result.

In Section 11 we consider the consequences of this theory of induced operator means and lambda extensions in the case of 2-variable operator means. We characterize the subset of lambda extensions 
in the set of all 2-variable operator means. We further prove that even in the case of induced operator means there are generalized Karcher equations \eqref{eq:intr4} such that induced operator means provide their unique positive solutions. Then we formulate some further open problems related to these extensions.

\section{Matrix means and some constructions}
Let us recall the family of matrix (or operator) means \cite{kubo}:
\begin{definition}\label{symmean}
A two-variable function \textit{M}:
$ \mathbb{P}\times \mathbb{P}\mapsto \mathbb{P}$ is called
a matrix or operator mean if
\begin{enumerate}
\renewcommand{\labelenumi}{(\roman{enumi})}
\item $M(I,I)=I$ where $I$ denotes the identity,
\item if $A\leq A'$ and $B\leq B'$, then $M(A,B)\leq M(A',B')$,
\item $CM(A,B)C\leq M(CAC,CBC)$ for all Hermitian $C$,
\item if $A_n\downarrow A$ and $B_n\downarrow B$ then $M(A_n,B_n)\downarrow M(A,B)$,
\end{enumerate}
where $\downarrow$ denotes the convergence in the strong operator topology of a monotone decreasing net.
\end{definition}
In property (ii), (iii), (iv) the partial order being used is the positive definite order, i.e. $A\leq B$ if and only if $B-A$ is positive semidefinite. An important consequence of these properties is \cite{kubo} that every matrix mean can be uniquely represented by a normalized, operator monotone function $f(t)$ in the following form
\begin{equation}\label{mean}
M(A,B)=A^{1/2}f\left(A^{-1/2}BA^{-1/2}\right)A^{1/2}\text{.}
\end{equation}
This unique $f(t)$ is said to be the representing function of the matrix mean $M(A,B)$. So actually matrix means are in one to one correspondence with normalized operator monotone functions, the above characterization provides an order-isomorphism between them. Normalization means that $f(1)=1$. For symmetric means, i.e. for means $M(A,B)=M(B,A)$, we have $f(t)=tf(1/t)$ which implies that $f'(1)=1/2$. Operator monotone functions have strong continuity properties, namely all of them are analytic functions and can be analytically continued to the upper complex half-plane. This is the consequence of the integral characterization of an operator monotone function $f(t)$, which is given over the interval $(0,\infty)$:
\begin{equation}\label{integralchar}
f(t)=\alpha+\beta t+\int_{0}^{\infty}\left(\frac{\lambda}{\lambda^2+1}-\frac{1}{\lambda+t}\right)d\mu(\lambda)\text{,}
\end{equation}
where $\alpha$ is a real number, $\beta\geq 0$ and $\mu$ is a unique positive measure on $(0,\infty)$ such that
\begin{equation}
\int_{0}^{\infty}\frac{1}{\lambda^2+1}d\mu(\lambda)<\infty\text{.}
\end{equation}
Actually the interval $(0,\infty)$ may be changed to an arbitrary $(a,b)$, in this case the integral is transformed to this interval accordingly. These are the consequences of the theory of Loewner, an introduction to the theory can be found in Chapter V \cite{bhatia}. We will use such integral characterization at several points in the article. The set of all matrix means is denoted by $\mathfrak{M}$, i.e.
\begin{equation*}
\begin{split}
\mathfrak{M}=\{&M(\cdot,\cdot):M(A,B)=A^{1/2}f(A^{-1/2}BA^{-1/2})A^{1/2}\text{, }f\text{ operator monotone on }\\
&(0,\infty), f(1)=1\}.
\end{split}
\end{equation*}
Similarly $\mathfrak{m}=\{f(x):f$ is a representing function of an $M\in\mathfrak{M}\}$.

One of our objectives is to find all possible symmetric matrix means which are also geodesic midpoint operations on smooth manifolds. Or more generally those matrix means that are arbitrary dividing points of geodesics. We will call such a matrix mean affine \cite{palfia3}:
\begin{definition}[Affine matrix mean]\label{affinemean}
An affine matrix mean $M:W^2\mapsto W$ is a matrix mean which is also a point of an arc-length parametrized geodesic on a smooth manifold $W\supseteq \textit{P}(n,\mathbb{C})$ equipped with an affine connection $\nabla$. I.e. $M(A,B)=\exp_A(t\log_A(B))$ for a fixed $t\in(0,1)$ and for all $A,B\in \textit{P}(n,\mathbb{C})$, $B$ is assumed to be in the injectivity radius of the exponential map $\exp_A(x)$ of the connection $\nabla$ given at the point $A$. The mapping $\log_A(x)$ is just the inverse of the exponential map at the point $A\in W$.
\end{definition}

We can make some basic observations about affine matrix means. First of all note, that by \eqref{mean} we have that $f(X)=M(I,X)$, so if $M(A,B)$ is an affine matrix mean, then $f(X)$ is some point of a geodesic connecting $X$ and $I$. Also on a smooth manifold with an affine connection if we differentiate the exponential map $\exp_{p}(X)$ at $p$, then we get $d\exp_p=I_p$, where $I_p$ is the identity transformation of the tangent space at $p$ \cite{helgason}. Therefore if we differentiate its inverse, the logarithm map $\log_{p}(q)$ we also get $d\log_p=I_p$ at $p$. So if we combine this with the chain rule we get that the differential of the mapping $M(p,q)=\exp_p(t\log_p(q))$ is $dM(p,\cdot)=tI_p$.

Now if we apply the above argument to an affine matrix mean $M(A,B)$ we get the following result.
\begin{proposition}
Let $M(A,B):=\exp_A(t\log_A(B))$ be an affine matrix mean. Then $f'(1)=t$.
\end{proposition}
\begin{proof}
Since $\textit{P}(n,\mathbb{C})$ is diffeomorphically embedded in $\textit{H}(n,\mathbb{C})$, therefore we can differentiate the map $M(A,B)=\exp_A(t\log_A(B))$ using the vector space structure of $\textit{H}(n,\mathbb{C})$, i.e. calculate the Fr\'echet differential which we denote for an arbitrary differentiable function $g$ by
\begin{equation}
Dg[X][Y]=\lim_{s\to 0}\frac{g(X+Ys)-g(X)}{s}
\end{equation}
at the matrix $X$ in the direction of the matrix $Y$. So by \eqref{mean} for all $H\in \textit{H}(n,\mathbb{C})$ we have
\begin{equation*}
\begin{split}
&\lim_{s\to 0}\frac{M(A,A+Hs)-M(A,A)}{s}\\
&=\lim_{s\to 0}A^{1/2}\frac{M(I,I+A^{-1/2}HA^{-1/2}s)-M(I,I)}{s}A^{1/2}=\\
&=A^{1/2}\lim_{s\to 0}\frac{f(I+A^{-1/2}HA^{-1/2}s)-f(I)}{s}A^{1/2}=A^{1/2}Df[I][A^{-1/2}HA^{-1/2}]A^{1/2}.
\end{split}
\end{equation*}
Since $f$ is an operator monotone function on $(0,\infty)$, it admits an integral characterization \eqref{integralchar}, so it can be analytically continued to the upper half-plane through the interval $(0,\infty)$. Therefore we may differentiate a power series representation of $f$, that uniformly converges on an open interval which contains $1$, so then we get that $Df[I][K]=Df[I][I]K=f'(1)K$ for all $K\in\textit{H}(n,\mathbb{C})$. Combining this with the above we get that 
\begin{equation*}
\lim_{s\to 0}\frac{M(A,A+Hs)-M(A,A)}{s}=A^{1/2}\left(Df[I][I]A^{-1/2}HA^{-1/2}\right)A^{1/2}=f'(1)H.
\end{equation*}
Since $H$ was arbitrary this yields that $t=f'(1)$, because $dM(p,\cdot)=tI_p$ and also the tangent space of $\textit{P}(n,\mathbb{C})$ at every point can be indentified by $\textit{H}(n,\mathbb{C})$.
$\square$
\end{proof}

By the preceding proposition we shall focus on matrix means represented by operator monotone functions $f$ on $(0,\infty)$ such that $f'(1)\in(0,1)$. We will use the notation $\mathfrak{P}(t)$ to denote the set of all operator monotone functions $f$ on $(0,\infty)$ such that $f(x)>0$ for all $x\in(0,\infty)$ and $f(1)=1, f'(1)=t$. We can find the minimal and maximal elements of $\mathfrak{P}(t)$ for all $t\in(0,1)$ easily.
\begin{lemma}\label{maxmininp}
For all $f(x)\in\mathfrak{P}(t)$ we have
\begin{equation}
\left((1-t)+tx^{-1}\right)^{-1}\leq f(x)\leq (1-t)+tx.
\end{equation}
\end{lemma}
\begin{proof}
Since every operator monotone function is operator concave, see Chapter V \cite{bhatia}, therefore we must have $f(x)\leq (1-t)+tx$ by concavity and the normalization conditions on elements of $\mathfrak{P}(t)$. Since the map $x^{-1}$ is order reversing on positive matrices, we have that if $f(x)\in\mathfrak{P}(t)$ then also $f(x^{-1})^{-1}\in\mathfrak{P}(t)$. So again by concavity
\begin{eqnarray*}
f(x^{-1})^{-1}&\leq& (1-t)+tx\\
f(x^{-1})&\geq& \left((1-t)+tx\right)^{-1}\\
f(x)&\geq& \left((1-t)+tx^{-1}\right)^{-1}.
\end{eqnarray*}

\end{proof}

Since $\left((1-t)+tx^{-1}\right)^{-1}$ and $(1-t)+tx$ are operator monotone we see that the they are the minimal and maximal elements of $\mathfrak{P}(t)$ respectively, and also they are the representing functions of the weighted harmonic and arithmetic means. This already gives us that the minimal and maximal affine matrix means are the weighted harmonic and arithmetic means respectively, so if $M(A,B)$ is an affine matrix mean, then
\begin{equation}
\left[(1-t)A^{-1}+tB^{-1}\right]^{-1}\leq M(A,B)\leq (1-t)A+tB.
\end{equation}
In general by the previous Lemma~\ref{maxmininp} the above inequality is true for all $M(A,B)$ matrix means with representing operator monotone function $f$ for which we have $f'(1)=t$. In this sense $\mathfrak{P}(t)$ characterizes weighted matrix means. If we take this as the definition of weighted matrix means, one can compare it with the definition of weighted matrix means given in \cite{palfia3}.

Consider a real differentiable function $f$ on some real open interval $I$ and $a\in I$. The function $f$ has a fixed point at $a$ if $f(a)=a$ and this is an attractive fixed point if $|f'(a)|<1$ or in other words the iterates $f^{\circ n}(x)$ converge to $a$ in a neighborhood of $a$, where $f^{\circ n}(x)=f(f^{\circ (n-1)}(x))$, see \cite{kuczma}.

\begin{lemma}\label{fixedpoint}
All $f(x)\in\mathfrak{P}(t)$ for $t\in(0,1)$ has only one fixed point in $(0,\infty)$ which is $1$ and $1$ is an attractive fixed point on $(0,\infty)$.
\end{lemma}
\begin{proof}
By the definition of $\mathfrak{P}(t)$ for all members $f(x)$ of this set $f(1)=1$, so $1$ is indeed a fixed point. By the preceding Lemma~\ref{maxmininp} we have
\begin{equation*}
\left((1-t)+tx^{-1}\right)^{-1}\leq f(x)\leq (1-t)+tx.
\end{equation*}
Therefore for all $x>1$ we have $f(x)<x$, i.e. $f(x)$ has no fixed point in $(1,\infty)$. Similarly for all $x\in(0,1)$ we have
\begin{equation*}
x<\left((1-t)+tx^{-1}\right)^{-1}\leq f(x),
\end{equation*}
therefore $f(x)$ cannot have a fixed point in $(0,1)$ as well.

Now the attractivity of the fixed point follows from the fact that $f(x)$ is monotonically increasing positive and concave on $(0,\infty)$ by operator monotonicity. Concavity implies that $f''(x)\leq 0$ for all $x\in(0,\infty)$. Also its derivative $f'(1)=t\in(0,1)$ at the fixed point $1$, so by Banach's fixed point theorem this fixed point is attractive on $(\epsilon,\infty)$, where $\epsilon<1$ is such that the derivative $f'(\epsilon)=1$. On $(0,\epsilon)$ the function $(1-t)+tx\geq f(x)>x$ so its subsequent iterates form an increasing sequence of functions. I.e. if we start an iteration with
$x_0\in(0,\epsilon)$, then after finitely many iterations by $f(x)$, $x_n=f(x_{n-1})$ will be in the interval $(\epsilon,1)$. From there convergence to $1$ follows again from Banach's fixed point theorem.

\end{proof}

In the paper \cite{kubo} Kubo and Ando used the following variant of the integral characterization \eqref{integralchar} for all $f(x)$ positive operator monotone functions on $(0,\infty)$:
\begin{equation}\label{kuboint}
f(x)=a+bx+\int_{0}^{\infty}\frac{(1+\lambda)x}{x+\lambda}dm(\lambda)
\end{equation}
where $m$ is a positive Borel measure on $(0,\infty)$, see Chapter 6 in \cite{schilling}.

\begin{proposition}\label{harmonicreprprop}
Let $f(x)$ be a representing function of a matrix mean in $\mathfrak{M}$. Then
\begin{equation}\label{harmonicrepr}
f(x)=\int_{[0,1]}[(1-s)+sx^{-1}]^{-1}d\nu(s)
\end{equation}
where $\nu$ is a probability measure over the closed interval $[0,1]$.
\end{proposition}
\begin{proof}
Let us begin with the integral characterization \eqref{kuboint} and let $\lambda=\frac{s}{1-s}$. Then $\lambda\in[0,\infty]$ if and only if $s\in[0,1]$ and the mapping is a bijection. With $m(\{0\})=a$ and $m(\{\infty\})=b$ we have
\begin{equation*}
\begin{split}
f(x)&=\int_{[0,\infty]}\frac{(1+\lambda)x}{x+\lambda}dm(\lambda)=\int_{[0,1]}\frac{1}{s}\left(\frac{1-s}{s}+x^{-1}\right)^{-1}dm\left(\frac{s}{1-s}\right)\\
&=\int_{[0,1]}[(1-s)+sx^{-1}]^{-1}d\nu(s)
\end{split}
\end{equation*}
where $d\nu(s)=dm\left(\frac{s}{1-s}\right)$ is a positive Borel measure on $[0,1]$. Now since $f(1)=1$ we have that
\begin{equation*}
1=f(1)=\int_{[0,1]}d\nu(s).
\end{equation*}

\end{proof}

\begin{remark}
The above results gives us that all matrix means are uniquely represented as convex combinations of weighted harmonic means, since the normalized operator monotone function
\begin{equation*}
f_t(x)=[(1-t)+tx^{-1}]^{-1}
\end{equation*}
is the representing function of the weighted harmonic mean, see Lemma~\ref{maxmininp}. Also the extreme points of this set are these weighted harmonic means.
\end{remark}

There are two degenerate cases of matrix means induced by a $\nu$ which are supported only over the single points ${0}$ or ${1}$. One of them is the left trivial mean
\begin{equation*}
l(x)=1
\end{equation*}
with represented matrix mean $M(A,B)=A$ and the right trivial mean
\begin{equation*}
r(x)=x
\end{equation*}
with represented matrix mean $M(A,B)=B$.

\begin{proposition}\label{inPt}
Let $M\in\mathfrak{M}$ with representing function $f(x)$. Then $0\leq f'(1)\leq 1$. Moreover if $M$ is not the left or right trivial mean ($f(x)\neq 1$ or $x$), then $f\in\mathfrak{P}(t)$.
\end{proposition}
\begin{proof}
Using Proposition~\ref{harmonicreprprop} we have that
\begin{equation*}
M(A,B)=\int_{[0,1]}[(1-s)A^{-1}+sB^{-1}]^{-1}d\nu(s)
\end{equation*}
where $\nu$ is a probability measure on $[0,1]$. This means that
\begin{equation*}
f(x)=\int_{[0,1]}(1-s+sx^{-1})^{-1}d\nu(s)
\end{equation*}
and
\begin{equation*}
f'(1)=\lim_{h\to 0}\left.\int_{[0,1]}\frac{(1-s+s(x+h)^{-1})^{-1}-(1-s+sx^{-1})^{-1}}{h}d\nu(s)\right|_{x=1}.
\end{equation*}
By Lebesgue's dominated convergence theorem we have
\begin{equation*}
\begin{split}
f'(1)&=\int_{[0,1]}\lim_{h\to 0}\left.\frac{(1-s+s(x+h)^{-1})^{-1}-(1-s+sx^{-1})^{-1}}{h}\right|_{x=1}d\nu(s)\\
&=\int_{[0,1]}sd\nu(s).
\end{split}
\end{equation*}
Since $\nu$ is a probability measure on $[0,1]$, this means that for its expectation $f'(1)$ (the value of the above integral) we have the bound $0\leq f'(1)\leq 1$. If $\nu$ is supported over a larger set then the single point sets $\{0\}$ or $\{1\}$, then clearly its expectation $f'(1)\in(0,1)$.

\end{proof}

\begin{proposition}\label{positivemapthm}
Let $\Phi$ be a positive unital linear map and $M\in\mathfrak{M}$. Then
\begin{equation*}
\Phi(M(A,B))\leq M(\Phi(A),\Phi(B))
\end{equation*}
for $A,B>0$.
\end{proposition}
\begin{proof}
Using Proposition~\ref{harmonicreprprop} we have that
\begin{equation*}
M(A,B)=\int_{[0,1]}[(1-s)A^{-1}+sB^{-1}]^{-1}d\nu(s)
\end{equation*}
where $\nu$ is a probability measure on $[0,1]$. By Theorem 4.1.5 in \cite{bhatia2} we have that
\begin{equation*}
\Phi([(1-s)A^{-1}+sB^{-1}]^{-1})\leq [(1-s)\Phi(A)^{-1}+s\Phi(B)^{-1}]^{-1}.
\end{equation*}
Using the fact that $\nu$ can be approximated by finitely supported measures and the linearity of $\Phi$, we get from the above that
\begin{equation*}
\begin{split}
\Phi(\int_{[0,1]}[(1-s)A^{-1}+sB^{-1}]^{-1}d\nu(s))=\int_{[0,1]}\Phi([(1-s)A^{-1}+sB^{-1}]^{-1})d\nu(s)\\
\leq \int_{[0,1]}[(1-s)\Phi(A)^{-1}+s\Phi(B)^{-1}]^{-1}d\nu(s).
\end{split}
\end{equation*}

\end{proof}

In \cite{kubo} Kubo and Ando defined the transpose of a matrix mean $M(A,B)$ as
\begin{equation}\label{transpose}
M'(A,B)=M(B,A).
\end{equation}
By Proposition~\ref{harmonicreprprop} it is clear that for an
\begin{equation*}
M(A,B)=\int_{[0,1]}[(1-s)A^{-1}+sB^{-1}]^{-1}d\nu(s)
\end{equation*}
we have that
\begin{equation*}
M'(A,B)=M(B,A)=\int_{[0,1]}[(1-s)B^{-1}+sA^{-1}]^{-1}d\nu(s).
\end{equation*}
So if $M'(A,B)$ has corresponding measure $\nu'$, then $d\nu'(s)=d\nu(1-s)$. Similarly for the representing functions we have $f'(x)=xf(1/x)$. Also symmetric means $M(A,B)=M'(A,B)$ have corresponding probability measures $\nu$ such that $d\nu(s)=d\nu(1-s)$ and vice versa.
\begin{corollary}\label{chracterizesymmetric}
The property $d\nu(s)=d\nu(1-s)$ characterizes symmetric means.
\end{corollary}

In order to advance further in the understanding of affine matrix means, we should be able to grasp more geometrical structure related to the affinely connected manifolds corresponding to affine matrix means. In the next section we will study the general situation of affinely connected manifolds given with a geodesic dividing point operation. We will see that in this case we can reconstruct the exponential map and its inverse, the logarithm map from the geodesic dividing point operation.

\section{The reconstruction of the exponential map}
In this section we reconstruct the exponential map of an arbitrary affinely connected differentiable manifold based first on its midpoint map. Without loss of generality we fix a base point $p$ as the starting point of the geodesics. The basics of the exponential map of a manifold can be found for example in Chapter I. paragraph 6 \cite{helgason}.

\begin{theorem}\label{symspc}
Let $M$ be an affinely connected smooth manifold diffeomorphically embedded into a vector space $V$. Suppose that the midpoint map $m(p,q)=\exp_p(1/2\log_p(q))$ is known in every normal neighborhood where the exponential map $\exp_p(X)$ is a diffeomorphism. Then in these normal neighborhoods the inverse of the exponential map $\log_p(q)$ can be fully reconstructed from the midpoint map in the form
\begin{equation}
\log_p(q)=\lim_{n\to \infty}\frac{m(p,q)^{\circ n}-p}{\frac{1}{2^{n}}}\text{,}
\end{equation}
where we use the notation $m(p,q)^{\circ n}\equiv m\left(p,m(p,q)^{\circ (n-1)}\right)$.
\end{theorem}
\begin{proof}
We will use some basic properties of the differential of the exponential map to construct the inverse of it, the logarithm map. Since in small enough normal neighborhoods the exponential map is a diffeomorphism, it can be given as the inverse of the logarithm map $\log_p(q)$.

By the basic properties of the exponential map we have
\begin{equation*}
\left.\frac{\partial\exp_p(Xt)}{\partial t}\right|_{t=0}=\lim_{t\to 0}\frac{\exp_p(Xt)-p}{t}=X\text{,}
\end{equation*}
where $X\in T_pM$. Here we used the fact that we have an embedding into a vector space. Suppose $\exp_p(X)=q$ is in the normal neighborhood. We are going to provide the limit on the right hand side of the above equation. The limit clearly exists in the normal neighborhood so
\begin{equation*}
\lim_{t\to 0}\frac{\exp_p(Xt)-p}{t}=\lim_{n\to \infty}\frac{\exp_p\left(X\frac{1}{2^{n}}\right)-p}{\frac{1}{2^{n}}}=\lim_{n\to \infty}\frac{m(p,q)^{\circ n}-p}{\frac{1}{2^{n}}}\text{.}
\end{equation*}
Here we use the notation $m(p,q)^{\circ n}\equiv m\left(p,m(p,q)^{\circ (n-1)}\right)$. We are in a normal neighborhood so the exponential map has an inverse, the logarithm map, so the limit can be written as
\begin{equation*}
X=\lim_{t\to 0}\frac{\exp_p(Xt)-p}{t}=\lim_{n\to \infty}\frac{m(p,q)^{\circ n}-p}{\frac{1}{2^{n}}}=\log_p(q)\text{.}
\end{equation*}

\end{proof}

In the above assertion we used the midpoint map to reconstruct the exponential map, but we can use arbitrary dividing point operation that yields a point, other then the ending points on the geodesic connecting two points in the normal neighborhood. This is summarized in the following proposition.

\begin{proposition}\label{constructexp}
Let $M$ be an affinely connected smooth manifold diffeomorphically embedded into a vector space $V$. In every normal neighborhood $N$ let $\gamma_{a,b}(t)$ denote the geodesic connecting $a,b\in N$ with parametrization $\gamma_{a,b}(0)=a$ and $\gamma_{a,b}(1)=b$. Suppose that the map $m(a,b)_{t_0}=\gamma_{a,b}(t_0)=\exp_p(t_0\log_p(q))$ is known for a $t_0\in (0,1)$ in every normal neighborhood $N$ where the exponential map is a diffeomorphism and $a,b\in N$. Then in these normal neighborhoods the logarithm map can be fully reconstructed as
\begin{equation*}
\log_p(q)=\lim_{n\to \infty}\frac{m(p,q)_{t_0}^{\circ n}-p}{t_0^{n}}\text{,}
\end{equation*}
with the notation $m(p,q)_{t_0}^{\circ n}\equiv m\left(p,m(p,q)_{t_0}^{\circ (n-1)}\right)_{t_0}$. We also obtain the exponential map by inverting $\log_p(q)$.
\end{proposition}

We are going to use this construction in the next sections to characterize affine matrix means.

\section{The exponential map of affine matrix means}
Based on the idea of reconstruction given by Proposition~\ref{constructexp} we are going to formally take the limits for matrix means in $\mathfrak{P}(t)$. The following result will show that if a matrix mean is affine then the exponential map of the corresponding smooth manifold has a special structure. The construction will be based on Proposition~\ref{constructexp} and is closely related to Schr\"oder's functional equation and its solution by Koenigs function as we will see later. We will use similarly the notation $M(A,B)^{\circ n}=M\left(A,M(A,B)^{\circ (n-1)}\right)$ as before in the previous section. The following result is similar to Theorem 11.6.1 in \cite{kuczma}, actually some parts of it can be derived from that theorem, however we give a slightly more general proof here for the sake completeness and further usage.

\begin{theorem}\label{uniformconv}
Let $M(A,B)$ be a matrix mean with representing function $f\in\mathfrak{P}(t)$. Then
\begin{equation}
\lim_{n\to \infty}\frac{M(A,B)^{\circ n}-A}{f'(1)^{n}}=A^{1/2}\log_I\left(A^{-1/2}BA^{-1/2}\right)A^{1/2}
\end{equation}
where the limit exists and is uniform for all $A,B\in \textit{P}(n,\mathbb{C})$ and $\log_I(x)$ is an operator monotone function which fulfills the functional equation
\begin{equation}\label{funcequ}
\log_I(f(x))=f'(1)\log_I(x)
\end{equation}
on the interval $(0,\infty)$.
\end{theorem}
\begin{proof}
We will prove the convergence to a continuous function $\log_I(t)$ in a more general setting. The operator monotonicity in the matrix mean case will be a particularization.

First of all note that by the repeated usage of \eqref{mean} we can reduce the above problem to the right hand side of the following formula:
\begin{equation*}
\frac{M(A,B)^{\circ n}-A}{f'(1)^{n}}=A^{1/2}\frac{f\left(A^{-1/2}BA^{-1/2}\right)^{\circ n}-I}{f'(1)^{n}}A^{1/2}\text{.}
\end{equation*}
From now on we will explicitly use the notation $g(x)^{\circ n}=g\left(g(x)^{\circ (n-1)}\right)$ for arbitrary function $g(t)$ where this notation is straightforward.

Due to the above formula it is enough to prove the assertion for a single operator monotone function $f(x)$. By operator monotonicity of $f(x)$ this is just the special case of the problem considered for arbitrary concave, analytic functions $f(x)$ given in the following form
\begin{equation}\label{construct}
\lim_{n\to \infty}\frac{f(X)^{\circ n}-I}{f'(1)^{n}}\text{,}
\end{equation}
for $X\in \textit{P}(n,\mathbb{C})$. As every operator monotone function which maps $(0,\infty)$ to $(0,\infty)$, is analytic on $(0,\infty)$ and has an analytic continuation to the complex upper half-plane across $(0,\infty)$, we can consider the functional calculus for Hermitian matrices in the above equations. Therefore we can further reduce the problem to the set of the positive reals by diagonalizing $X$ and considering the convergence for every distinct diagonal element separately. For an extensive study on operator monotone functions one may refer to Chapter V in \cite{bhatia}. 

Without loss of generality we may shift the function $f(x)$ by $1$ so it is enough to show the assertion for
\begin{equation*}
\lim_{n\to \infty}\frac{g(x)^{\circ n}}{g'(0)^{n}}\text{,}
\end{equation*}
where $g(x)=f(x+1)-1$ and so $g(x)^{\circ n}=f(x+1)^{\circ n}-1$. From now on we will be considering the shifted problem. At this point we must emphasize the fact that the function $g$ must have $0$ as an attractive and only fixed point on the interval of interest $(-1,\infty)$. In the unshifted case this is equivalent to $f$ having $1$ as the only attractive fixed point on the interval $(0,\infty)$, which is the case by Lemma~\ref{fixedpoint}. So we can also assume that $0<g'(0)<1$. The rest of the argument will be based on the claim that the above limit of analytic functions of the form $g(x)^{\circ n}/g'(0)^{n}$ is uniform Cauchy therefore the limit function exists and is continuous.

First of all we have $0$ as the attractive and only fixed point of $g$, so for arbitrary $x\in (-1,\infty)$ the sequence $x_n=g(x)^{\circ n}$ converges to $0$. We have $g(0)=0$ and by the mean value theorem we have
\begin{equation*}
x_n=g(x)^{\circ n}=g'(t_n)g(x)^{\circ (n-1)}=\prod_{i=1}^{n}g'(t_i)x\text{,}
\end{equation*}
where $t_i\in \left[0,g(x)^{\circ (i-1)}\right]$ if $x\geq 0$ or $t_i\in \left[g(x)^{\circ (i-1)},0\right]$ if $x<0$, since $g$ is a concave function on $(-1,\infty)$. As $x_n\to 0$ for arbitrary $x$ we have $g'(t_i)\to g'(0)$. Now we have to obtain a suitable upper bound on
\begin{equation}\label{difference}
\left|\frac{g(x)^{\circ n}}{g'(0)^n}-\frac{g(x)^{\circ m}}{g'(0)^m}\right|\text{.}
\end{equation}
We argue as follows
\begin{equation*}
\begin{split}
&\left|\frac{g(x)^{\circ n}}{g'(0)^n}-\frac{g(x)^{\circ m}}{g'(0)^m}\right|=\frac{\left|g(x)^{\circ n}-g'(0)^{n-m}g(x)^{\circ m}\right|}{g'(0)^n}\leq\\
&\leq\frac{\left|\prod_{i=m+1}^{n}g'(t_i)-g'(0)^{n-m}\right|\left|\prod_{i=1}^{m}g'(t_i)\right|}{g'(0)^n}|x|=\\
&=\left|\prod_{i=m+1}^{n}\frac{g'(t_i)}{g'(0)}-1\right|\left|\prod_{i=1}^{m}\frac{g'(t_i)}{g'(0)}\right||x|\text{.}
\end{split}
\end{equation*}
Now uniform convergence follows if $\left|\prod_{i=1}^{\infty}g'(t_i)/g'(0)\right|<\infty$ because then the tail $\prod_{i=m+1}^{\infty}g'(t_i)/g'(0)\to 1$ so \eqref{difference} can be arbitrarily small on any closed interval in $(-1,\infty)$ by choosing a uniform $m$. By the continuity of $g'(t)$ and $x_n\to 0$ we have $g'(t_i)\to g'(0)$ and by assumption $0<g'(0)<1$, therefore there exists $N$ and $q$ such that for all $i>N$ we have $0<g'(t_i)\leq q<1$. What follows here is that $\exists K_1,K_2<\infty$ such that $|t_N|\leq K_1$ and $|g''(t_i)|\leq K_2$ for all $i>N$. This yields the bound $|t_i|\leq K_1q^{i-N}$ for all $i>N$. Considering the Taylor expansion of $g'$ around $0$ we get
\begin{equation*}
\frac{g'(t_i)}{g'(0)}=\frac{g'(0)+g''(t'_i)t_i}{g'(0)}
\end{equation*}
for $0<t'_i<t_i$. What follows from this is that
\begin{equation*}
\left|\prod_{i=N}^{\infty}\frac{g'(t_i)}{g'(0)}\right|\leq \prod_{i=N}^{\infty}\left(1+\frac{K_1K_2}{g'(0)}q^{i-N}\right)\text{.}
\end{equation*}
The infinite product on the right hand side converges because $\sum_{j=0}^{\infty}\frac{K_1K_2}{g'(0)}q^j$ converges hence $\left|\prod_{i=1}^{\infty}g'(t_i)/g'(0)\right|<\infty$ for all $x$ in the closed interval.

At this point we can easily establish the convergence for normalized operator monotone functions because they are concave functions by Theorem V.2.5 in \cite{bhatia}, so $f''(t)\leq 0$ and they have only one fixed point which is $1$. The fact that the limit is operator monotone function in this case follows from the operator monotonicity of the generating $f(t)$.

The functional equation \eqref{funcequ} is the consequence of the following:
\begin{equation*}
\begin{split}
\log_I(f(x))&=\lim_{n\to \infty}\frac{f(f(x))^{\circ n}-1}{f'(1)^{n}}=\lim_{n\to \infty}\frac{f(x)^{\circ (n+1)}-1}{f'(1)^{n}}\\
&=\lim_{n\to \infty}f'(1)\frac{f(x)^{\circ (n+1)}-1}{f'(1)^{n+1}}=f'(1)\log_I(x).
\end{split}
\end{equation*}

\end{proof}

\begin{remark}
The above result is formulated for $A,B\in\mathbb{P}(n,\mathbb{C})$, but invoking the properties of the functional calculus for bounded self-adjoint operators on a Hilbert space, it holds more generally for $A,B\in\mathbb{P}$ as well.
\end{remark}

Actually the proof of Theorem~\ref{uniformconv} works for a larger class of functions then the family of normalized operator monotone functions. The limit in \eqref{construct} exists and it is a continuous function if the twicely differentiable function $f(x)$ has $1$ as the only attractive fixed point and the derivative $0<f'(x)<1$. This is not a coincidence: 

\begin{remark}
The functional equation \eqref{funcequ} was first studied by Schr\"oder for holomorphic functions on the unit disk in \cite{schroeder} long ago. Later Koenigs suggested in \cite{koenigs} the iterative construction given above in Theorem~\ref{uniformconv} to provide a solution to the functional equation on the unit disk. Usually in this setting the function $\log_I$ is said to be a Koenigs eigenfunction for function composition as an operator acting on a certain Hardy space of holomorphic functions on the complex unit disk. He proved also that the rate of convergence of the iteration to $\log_I$ is geometric, moreover that Koenigs function is the unique solution of the functional equation in the class of analytic functions. See also \cite{cowen,shapiro,szekeres} for other results in this setting.
\end{remark}

The next example shows how to calculate the limit function explicitly.
\begin{example}\label{ex1}
Consider the one parameter family of functions
\begin{equation*}
f_q(x)=\left[(1-t)+tx^{q}\right]^{1/q}
\end{equation*}
for $t\in(0,1)$. These are in $\mathfrak{P}(t)$ if and only if $q\in [-1,1]$, because for other values of $q$ the function is not operator monotone, see exercise 4.5.11 in \cite{bhatia2}. It is easy to see that
\begin{equation*}
f_q(x)^{\circ n}=\left[t^{n}x^{q}+\sum_{k=0}^{n-1}t^{k}(1-t)\right]^{1/q}=\left[t^{n}x^{q}-t^{n}+1\right]^{1/q}.
\end{equation*}
In this case we can easily calculate the limit function $\log_{I,f_q}(x)$ by turning the limit into a derivative:
\begin{equation*}
\begin{split}
\log_{I,f_q}(x)=&\lim_{n\to\infty}\frac{(t^{n}x^{q}-t^{n}+1)^{1/q}-1}{t^{n}}=\lim_{s\to 0}\frac{(sx^{q}-s+1)^{1/q}-1}{s}\\
=&\left.\frac{\partial}{\partial s}(sx^q-s+1)^{1/q}\right|_{s=0}=\frac{x^q-1}{q}.
\end{split}
\end{equation*}
The limit functions indeed are operator monotone again if and only if $q\in [-1,1]$. This family has a singularity at $q=0$ but it is easy to verify that it is a removable singularity, so in fact we have
\begin{equation*}
\begin{split}
f_0(x)&=x^t\\
\log_{I,f_0}(x)&=\log(x),
\end{split}
\end{equation*}
where $\log(x)$ and $x^t$ are also well known to be operator monotone. Particularly $x^t$ as a representing function corresponds to the weighted geometric mean.
\end{example}

\begin{proposition}\label{injective}
The limit function $\log_I(x)$ in Theorem~\ref{uniformconv} satisfies the following:
\begin{itemize}
\item[(i)] $\log_I(x)$ maps $\textit{P}(n,\mathbb{C})$ to $\textit{H}(n,\mathbb{C})$ injectively,
\item[(ii)] $1-x^{-1}\leq\log_I(x)\leq x-1$ for all $x>0$,
\item[(iii)] If $\log_{I,f}(x)$ and $\log_{I,g}(x)$ are the corresponding limit functions for $f,g\in\mathfrak{P}(t)$ such that $f(x)\leq g(x)$ for all $x>0$, then $\log_{I,f}(x)\leq \log_{I,g}(x)$ for all $x>0$,
\item[(iv)] $\log_I(1)=0$ and $\log_I'(1)=1$.
\end{itemize}
\end{proposition}
\begin{proof}
(iii): Since $f(x)\leq g(x)$ by monotonicity we have $f(x)^{\circ n}\leq g(x)^{\circ n}$. From this it follows that
\begin{equation*}
\frac{f(x)^{\circ n}-1}{f'(1)^{n}}\leq \frac{g(x)^{\circ n}-1}{g'(1)^{n}},
\end{equation*}
and the inequality is also preserved in the limit.

(ii): By Lemma~\ref{maxmininp} we have
\begin{equation*}
\left((1-t)+tx^{-1}\right)^{-1}\leq f(x)\leq (1-t)+tx
\end{equation*}
where on the left hand side we have the function $f_{-1}(x)$ and on the right hand side we have $f_{1}(x)$ from Example~\ref{ex1}. In Example~\ref{ex1} we calculated the corresponding limit functions, so these combined with the previous property (iii) proves property (ii).

(i): By property (ii) it follows that $\log_I(x)$ is nonconstant on $(0,\infty)$. Also $\log_I(x)$ is operator monotone there, so it is strictly concave, therefore injective and real valued. This combined with the functional calculus for matrix functions proves the property.

(iv): $\log_I(1)=0$ follows from (ii). Using this and (ii) again we have
\begin{equation*}
\frac{1-(1+h)^{-1}}{h}\leq \frac{\log_I(1+h)-\log_I(1)}{h}\leq \frac{(1+h)-1}{h}.
\end{equation*}
Taking the limit $h\to 0$ we get derivatives on the left and right hand sides are $1$, so also $\log_I'(1)=1$.

\end{proof}

Since $\log_I(x)$ is operator monotone on $(0,\infty)$, it is also analytic there, so it has an analytic inverse $\exp_I(x)$ by Lagrange's inversion theorem, since its derivative is nonzero due to Proposition~\ref{injective}. It is also easy to see that $\exp'_I(0)=1$ and $\exp_I(0)=1$. By these considerations we have just arrived at the following

\begin{proposition}\label{meanswithexp}
Let $f\in\mathfrak{P}(t)$. Then
\begin{equation}
f(x)=\exp_I\left(f'(1)\log_I(x)\right),
\end{equation}
where $\log_I\in\mathfrak{L}$ is the unique solution of the functional equation \eqref{funcequ} in the wider class of continuously differentiable and invertible functions on $(0,\infty)$ which vanish at $1$ and have derivative $1$ at $1$.
\end{proposition}
\begin{proof}
The first part of the assertion follows from the invertibility of $\log_I(x)$ on $(0,\infty)$ by Theorem~\ref{uniformconv}. For the 
second uniqueness part note that if $\log_{I,2}(x)$ is an invertible continuously differentiable solution of the functional equation \eqref{funcequ} and also $\log_{I,2}(1)=0$ and $\log_{I,2}'(1)=1$, then its inverse $\exp_{I,2}(x)$ exists, $\exp_{I,2}(0)=1$ and $\exp_{I,2}'(0)=1$. Moreover by Theorem~\ref{uniformconv}
\begin{eqnarray*}
\log_I(x)&=&\lim_{n\to\infty}\frac{f(x)^{\circ n}-1}{f'(1)^{n}}=\lim_{n\to\infty}\frac{\exp_{I,2}(f'(1)^n\log_{I,2}(x))-\exp_{I,2}(0)}{f'(1)^{n}}\\
&=&\lim_{s\to 0}\frac{\exp_{I,2}(s\log_{I,2}(x))-\exp_I(0)}{s}=\left.\frac{\partial}{\partial s}\exp_{I,2}(s\log_{I,2}(x))\right|_{s=0}\\
&=&\log_{I,2}(x),
\end{eqnarray*}
so the functions $\log_I$ and $\log_{I,2}$ are identical.

\end{proof}

The above propositions put some restrictions on the possible functions $\log_I(x)$ that can occur as limits in Theorem~\ref{uniformconv}. Therefore we will use the notation $\mathfrak{L}$ to denote the set of operator monotone functions $g(x)$ on $(0,\infty)$ such that $g(1)=0$ and $g'(1)=1$. By Proposition~\ref{meanswithexp} it is clear, that for each $f\in\mathfrak{P}(t)$ we have a unique corresponding $\log_I(x)$ in $\mathfrak{L}$. 
\begin{definition}[Exponential and logarithm maps]
We say that for an $f\in\mathfrak{P}(t)$ the corresponding unique solution $\log_I(x)$ in $\mathfrak{L}$ of the functional equation \eqref{funcequ} is the logarithm map corresponding to $f(x)$, while its inverse $\exp_I(x)$ is the exponential map corresponding to $f(x)$.
\end{definition}


In the following section we will go the other way around and see whether the function
\begin{equation*}
f(x)=\exp_I\left(t\log_I(x)\right)
\end{equation*}
is in $\mathfrak{P}(t)$ for all $\log_I\in\mathfrak{L}$ and $t\in(0,1)$.

\section{Semigroups of representing functions induced by logarithm maps}
In the previous section we established that for every $f\in\mathfrak{P}(t)$ there exists a unique function $\log_I\in\mathfrak{L}$ such that it fulfills the functional equation \eqref{funcequ}. In order to see whether an element $\log_I\in\mathfrak{L}$ also induces a representing function $f\in\mathfrak{P}(t)$ with the generalized functional equation
\begin{equation*}
\log_I(f(x))=t\log_I(x)
\end{equation*}
for all $t\in(0,1)$, we must extend our investigations into the upper complex half-plane $\mathbb{H}^{+}=\left\{z\in\mathbb{C}:\Im z>0\right\}$.

First of all let us recall Nevanlinna's representation \cite{bhatia} of holomorphic functions $f:\mathbb{H}^{+}\to\mathbb{H}^{+}$. By Nevanlinna's theorem each such $f$ can uniquely be written as
\begin{equation}\label{nevanlinna}
f(z)=\alpha+\beta z+\int_{-\infty}^{\infty}\frac{\lambda z+1}{\lambda-z}d\nu(\lambda),
\end{equation}
where $\alpha\in\mathbb{R}, \beta\geq 0$ and $\nu$ is a positive measure with support in $(-\infty,\infty)$. It is well known that $f$ can be extended to the lower half-plane $\mathbb{H}^{-}=\left\{z\in\mathbb{C}:\Im z<0\right\}$ as well by Schwarz reflection $\overline{f(\overline{z})}$ for all $z\in\mathbb{H}^{-}$. Therefore also if this extension is by analytic continuation over an interval $(a,b)$, then $\nu$ vanishes on the interval \cite{bhatia}. Similarly if $\nu$ vanishes on a real interval, then $f$ is holomorphic on the interval as well and can be analyticly continued to the lower half-plane.

Representation \eqref{nevanlinna} will be useful for studying functions in $\mathfrak{L}$. For example Nevanlinna's representation yields that all $f\in\mathfrak{L}$ can be represented as 
\begin{equation}\label{nevanlinna2}
f(z)=\alpha+\beta z+\int_{-\infty}^{0}\frac{\lambda z+1}{\lambda-z}d\nu(\lambda),
\end{equation}
where $\alpha\in\mathbb{R}, \beta\geq 0$ and $\nu$ is a positive measure with support in $(-\infty,0)$. This is due to the required holomorphicity of $f$ on $(0,\infty)$. Next let us find the maximal and minimal elements in $\mathfrak{L}$.

\begin{lemma}\label{maxmininl}
For all $\log_I\in\mathfrak{L}$ we have
\begin{equation}
1-x^{-1}\leq \log_I(x)\leq x-1.
\end{equation}
\end{lemma}
\begin{proof}
Since every operator monotone function is operator concave, therefore we must have $\log_I(x)\leq x-1$ by concavity and the normalization conditions on elements of $\mathfrak{L}$. Since the map $x^{-1}$ and $-x$ is order reversing on hermitian matrices, we have that if $\log_I(x)\in\mathfrak{L}$ then also $-\log_I(x^{-1})\in\mathfrak{L}$. So again by concavity
\begin{eqnarray*}
-\log_I(x^{-1})&\leq& x-1\\
\log_I(x)&\geq& 1-x^{-1}.
\end{eqnarray*}
Clearly $x-1$ and $1-x^{-1}$ are also in $\mathfrak{L}$.

\end{proof}

\begin{proposition}\label{logrepr}
Let $\log_I\in\mathfrak{L}$. Then
\begin{equation}\label{logreprharm}
\log_I(x)=\int_{[0,1]}\left[1-2s+\frac{sx-(1-s)}{(1-s)x+s}\right]\frac{d\nu(s)}{(1-s)^2+s^2}
\end{equation}
where $\nu$ is a probability measure over the closed interval $[0,1]$.
\end{proposition}
\begin{proof}
By \eqref{nevanlinna2} and the transformation of the integral we have that
\begin{equation}\label{nevanlinna3}
\log_I(z)=\alpha+\beta z+\int_{0}^{\infty}\frac{\lambda z-1}{\lambda+z}d\mu(\lambda),
\end{equation}
where $\alpha\in\mathbb{R}, \beta\geq 0$ and $\mu$ is a positive measure with support in $(0,\infty)$. The mapping $\lambda=\frac{s}{1-s}$ is a bijection from $[0,1]$ to $[0,\infty]$. Then by change of variables we have
\begin{equation*}
\log_I(z)=\alpha+\beta z+\int_{0}^{1}\frac{sz-(1-s)}{(1-s)z+s}d\mu\left(\frac{s}{1-s}\right).
\end{equation*}
Letting $\mu(\{\infty\})=\beta$, we have
\begin{equation*}
\log_I(z)=\alpha+\int_{[0,1]}\frac{sz-(1-s)}{(1-s)z+s}d\mu\left(\frac{s}{1-s}\right).
\end{equation*}
Since $\log_I(1)=0$ we have
\begin{equation*}
\alpha=\int_{[0,1]}1-2sd\mu\left(\frac{s}{1-s}\right).
\end{equation*}
Using Lemma~\ref{maxmininl} and Lebesgue's dominated convergence theorem we get
\begin{equation*}
\log_I'(1)=1=\int_{[0,1]}(1-s)^2+s^2d\mu\left(\frac{s}{1-s}\right)
\end{equation*}
which means that the measure
\begin{equation*}
d\nu(s)=[(1-s)^2+s^2]d\mu\left(\frac{s}{1-s}\right)
\end{equation*}
is a probability measure on $[0,1]$, so \eqref{logreprharm} follows.

\end{proof}

At this point let us refer again to the functional equation \eqref{funcequ} in the previous section. By the above considerations we can generalize \eqref{funcequ} by analytic continuation.
\begin{proposition}\label{proplog}
Let $f\in\mathfrak{P}(t)$. Then the function $\log_I(x)$ given in Theorem~\ref{uniformconv} admits analytic continuation to $\mathbb{H}^{+}$ and also to $\mathbb{H}^{-}$ across $(0,\infty)$ by relfection, moreover it fulfills the functional equation
\begin{equation}\label{funcequ2}
\log_I(f(z))=f'(1)\log_I(z)
\end{equation}
for all $z\in\mathbb{C}\backslash(-\infty,0]$.
\end{proposition}
\begin{proof}
Since analytic continuation of $f(x)$ and $\log_I(x)$ can be performed using the integral characterizations \eqref{integralchar} and \eqref{nevanlinna2} respectively, we end up with holomorphic functions living on $\mathbb{C}\backslash(-\infty,0]$. Since $\log_I(z)$ is a holomorphic function, it has a meromorphic inverse $\exp_I(z)$. So we have
\begin{equation*}
F(z)=\exp_I\left(f'(1)\log_I(z)\right),
\end{equation*}
a meromorphic function that is identical to $f(z)$ everywhere on the domain $(0,\infty)$. Therefore by uniqueness of meromorphic and analytic continuation we must have $F(z)=f(z)$ everywhere on the domain $\mathbb{C}\backslash(-\infty,0]$.

\end{proof}

The above result tells us, that for a given $\log_I\in\mathfrak{L}$ we should consider the generalized functional equation
\begin{equation}\label{funcequ3}
\log_I(f_t(z))=t\log_I(z)
\end{equation}
to define a representing function $f_t(z)$ for all $t\in(0,1)$ corresponding to $\log_I(z)$ which was itself obtained by analyitc continuation using representation \eqref{nevanlinna2}. The obvious question that arizes here is whether every $\log_I(z)$ in $\mathfrak{L}$ has a corresponding $f_t\in\mathfrak{P}(t)$? We need the following:

\begin{definition}[Radial convexity]
Let $S\subseteq \mathbb{C}$ be such that $0\in S$. We will say that $S$ is radially convex if and only if for all $z\in S$ also $tz\in S$ for all $t\in[0,1]$.
\end{definition}

\begin{proposition}\label{starlikeimage}
Let $\log_I\in\mathfrak{L}$. Then $\log_I$ maps $\mathbb{H}^{+}$ to a radially convex set in $\mathbb{H}^{+}$.
\end{proposition}
\begin{proof}
By Proposition~\ref{logrepr} we have that
\begin{equation*}
\log_I(z)=\int_{[0,1]}\left[1-2s+\frac{sx-(1-s)}{(1-s)x+s}\right]\frac{d\nu(s)}{(1-s)^2+s^2}
\end{equation*}
where $\nu$ is a probability measure on $[0,1]$. Since the set of probability measures on $[0,1]$ is weak-$*$ compact, by the Krein-Milman theorem there exists a net $\nu_l$ of finitely supported probability measures converging to $\nu$. Therefore the functions
\begin{equation*}
\log_{I,l}(z)=\int_{[0,1]}\left[1-2s+\frac{sx-(1-s)}{(1-s)x+s}\right]\frac{d\nu_l(s)}{(1-s)^2+s^2}
\end{equation*}
converge pointwisely to $\log_I(z)$. Since $\nu_l$ is finitely supported, $\log_{I,l}(z)$ is a finite convex combinations of functions of the form
\begin{equation*}
h_{s}(z)=\frac{1-2s}{(1-s)^2+s^2}+\frac{1}{(1-s)^2+s^2}\frac{sx-(1-s)}{(1-s)x+s}.
\end{equation*}
If we have $h_{s}(z)=w$ for $z\in\mathbb{H}^{+}$ and $s\neq 1$, then after some calculation we get that
\begin{equation*}
z=\frac{(1-s)^2+s^2}{[(1-s)^4+s^2(1-s)^2]w+2s(1-s)^2-(1-s)}-\frac{s}{1-s}
\end{equation*}
which means that if $h_{s}(z)=w\in\mathbb{H}^{+}$, then for all $a\in(0,1)$ there exists $z_a\in\mathbb{H}^{+}$ such that $h_{\lambda}(z_a)=aw$. Similar argument proves this in the case of $s=1$.

Now if we consider any convex combination of such functions $h_{s}(z)$, the resulting function will still have a radially convex image of $\mathbb{H}^{+}$. The reason for this is that if $x_i\in S_i\subseteq\mathbb{H}^{+}$ where $S_i$ are radially convex sets, then $ax_i\in S_i$ for all $a\in(0,1)$. Therefore if $S_i$ are the images of $\mathbb{H}^{+}$ under the mappings $K_ih_{s_i}(z)$ for some $K_i>0$ and $s_i$, then the image $S$ of $\mathbb{H}^{+}$ under the function that we get as the sum of the functions $K_ih_{s_i}(z)$, is radially convex, since every element of it can be written as a sum of some $x_i\in S_i$. So we also have that the sum of $ax_i$ is in $S$ too by the convexity of each $S_i$. Therefore $S$ must be radially convex. Now sums of $K_ih_{s_i}(z)$ give each $\log_{I,l}(z)$ that converge to $\log_I$. Since sums of $K_ih_{s_i}(z)$ have radially convex image, so does $\log_I$.

\end{proof}

\begin{theorem}\label{nobranchthm}
Let $\log_I\in\mathfrak{L}$. Then $f_t\in\mathfrak{P}(t)$ for all $t\in(0,1)$ if and only if $log_I(z)$ has no ramification point in $\mathbb{H}^{+}$.
\end{theorem}
\begin{proof}
First of all since $\log_I\in\mathfrak{L}$, it follows that $\log_I(x)$ is invertible on $(0,\infty)$ because it is nonconstant monotone increasing there, also it is invertible in a neighborhood of $(0,\infty)$ and its inverse $\exp_I(z)$ is holomorphic in that neighborhood and $f_t(x)\in(0,\infty)$ and meromorphic in $\mathbb{H}^{+}$.

Suppose that $log_I(z)$ has no ramification point in $\mathbb{H}^{+}$. Then by the previous Proposition~\ref{starlikeimage} it maps $\mathbb{H}^{+}$ to a radially convex set. Since $log_I(z)$ has no ramification point, it has a univalent holomorphic inverse $\exp_I(z)$, so
\begin{equation*}
f_t(z)=\exp_I(t\log_I(z))
\end{equation*}
is a well defined holomorphic function on $\mathbb{H}^{+}$. Moreover $f_t(z)$ is real valued over $(0,\infty)$. Since the image $\log_I(\mathbb{H}^{+})$ of $\mathbb{H}^{+}$ under the map $\log_I(z)$ is radially convex, we have that for any $s\in\log_I(\mathbb{H}^{+})$ also $ts\in\log_I(\mathbb{H}^{+})$. Therefore $t\log_I(\mathbb{H}^{+})\subseteq\log_I(\mathbb{H}^{+})$, so also $\exp_I(t\log_I(\mathbb{H}^{+}))\subseteq \mathbb{H}^{+}$.

Now for the only if part suppose on the contrary that $\log_I$ has a ramification point in $\mathbb{H}^{+}$. Then its inverse $\exp_I$ has a branch point at the image of the ramification point under $\log_I$ which means that $f_t(z)$ is not holomorphic there, but this contradicts $f_t\in\mathfrak{P}(t)$.

\end{proof}

What happens if $\log_I$ has a ramification point in $\mathbb{H}^{+}$? What can then be said about $f_t(z)$?

\begin{proposition}\label{ftinPt}
Let $\log_I\in\mathfrak{L}$ be induced by an $f_{t_0}\in\mathfrak{P}(t_0)$ using Proposition~\ref{meanswithexp}. Then $f_t\in\mathfrak{P}(t)$ for all $0<t\leq t_0$.
\end{proposition}
\begin{proof}
By Proposition~\ref{meanswithexp} we have that there is no image of a ramification point of $\log_I$ in the domain $t_0\log_I(\mathbb{H}^{+})\subseteq\mathbb{H}^{+}$, otherwise $f_{t_0}(z)$ would have a singularity in $\mathbb{H}^{+}$. But since for all $0<t\leq t_0$ we have that $t\log_I(\mathbb{H}^{+})\subseteq t_0\log_I(\mathbb{H}^{+})$, therefore $f_t(z)$ is singularity free as well.

\end{proof}

\begin{remark}
In general one can assure that if for a given $\log_I\in\mathfrak{L}$ with ramification points $t\log_I(\mathbb{H}^{+})$ avoids the image of the ramification points (of $\log_I$) under $\log_I$ in $\mathbb{H}^{+}$, then $f_t\in\mathfrak{P}(t)$.
\end{remark}

Considering only real $x>0$ it makes sense to talk about $f_t(x)$ for all $t\in[0,1]$, even if $f_t(z)$ has some singularities, since there are no positive real singularities. Then one can prove a general upper bound on $f_t(x)$ and also a monotonicity result.
\begin{proposition}\label{p:propbound1}
Let $f_{t_0}\in\mathfrak{P}(t_0)$. Then for all $t\in[0,1]$
\begin{equation*}
f_t(x)\leq (1-t)+tx.
\end{equation*}
Moreover $1\geq s\geq t$ implies $f_s(x)\geq f_t(x)$ for all $x>0$.
\end{proposition}
\begin{proof}
By definition $f_t(x)=\exp_I(t\log_I(x))$ and this is well defined for real $x>0$, since $\log_I\in\mathfrak{L}$, so $\log_I$ is strictly monotone and holomorphic. By Lemma~\ref{maxmininl} we have $\log_I(x)\leq x-1$ and since $\log_I\in\mathfrak{L}$, therefore $\log_I$ is also concave, i.e. $\log_I''(x)\leq 0$. By simple calculation
\begin{equation}\label{eq:propbound1}
f_t'(x)=\frac{\partial}{\partial x}f_t(x)=\exp_I'(t\log_I(x))t\log_I'(x)
\end{equation}
also since $x=\exp_I(\log_I(x))$ therefore
\begin{equation}\label{eq:propbound2}
1=\frac{\partial}{\partial x}\exp_I(\log_I(x))=\exp_I'(\log_I(x))\log_I'(x).
\end{equation}
Notice that $\frac{\partial}{\partial x}(1-t)+tx=t$ and also $f_t(1)=1=\left.(1-t)+tx\right|_{x=1}$ and $f_t'(1)=t=\left.\partial/\partial x(1-t)+tx\right|_{x=1}$. So to conclude the assertion it suffices to prove by the Mean value theorem that $f_t'(x)\leq t$ for $x\geq 1$ and $f_t'(x)\geq t$ for $0<x<1$. Now $\log_I$ is concave and $\log_I(1)=0$, $\log_I'(1)=1$, therefore $\exp'_I(x)$ is increasing and $\exp'_I(0)=1$. This means that $\exp_I'(tx)\leq \exp_I'(x)$ for $x\geq 0$ and $\exp_I'(tx)\geq \exp_I'(x)$ for $x<0$. This combined with \eqref{eq:propbound2} and \eqref{eq:propbound1} yields that $f_t'(x)\leq t$ for $x\geq 1$ and $f_t'(x)\geq t$ for $0<x<1$.

The second monotonicity part of the assertion follows from a similar argument leading to $f_t'(x)\leq f_s'(x)$ for $x\geq 1$ and $f_t'(x)\geq f_s'(x)$ for $0<x<1$.
\end{proof}

\begin{remark}
By Proposition~\ref{ftinPt} if $f_{t_0}\in\mathfrak{P}(t_0)$, then for all $0<s,t\leq t_0$ we have
\begin{equation*}
f_{st}=f_s\circ f_t=f_t\circ f_s
\end{equation*}
and $f_{st}\in\mathfrak{P}(st)$. I.e. $f_t$ is a semigroup of holomorphic functions with respect to function composition, see \cite{cowen,kuczma}. Actually this semigroup property is still true if we consider $0\leq s,t\leq 1$, but then we have possible singularities as well.
\end{remark}

\begin{proposition}
Let $f_t\in\mathfrak{P}(t)$ and $\log_I\in\mathfrak{L}$ its corresponding logarithm map such that it fulfills the functional equation \eqref{funcequ}. Then $z_0$ is a ramification point of $f_t$ if and only if it is a ramification point of $\log_I$.
\end{proposition}
\begin{proof}
By Proposition~\ref{proplog} $f_t(z)=\exp_I(t\log_I(z))$. So
\begin{equation*}
f_t'(z)=t\exp_I'(t\log_I(z))\log_I'(z).
\end{equation*}
Since $f_t(z)=\exp_I(t\log_I(z))$ is holomorphic on $\mathbb{H}^{+}$ also $\exp_I'(t\log_I(z))$ is holomorphic, moreover $\exp_I'(t\log_I(z))\neq 0$ since the inverse $\exp_I^{-1}(z)=\log_I(z)$ exists and is holomorphic on the whole $\mathbb{H}^{+}$ and $f_t(\mathbb{H}^{+})\subseteq\mathbb{H}^{+}$. Therefore if $f_t'(z_0)=0$ then also $\log_I'(z_0)=0$ and vice versa.

\end{proof}

According to Theorem~\ref{nobranchthm} we need to find members of $\mathfrak{L}$ without ramification points. In other words we are looking for mappings that are univalent (schlicht) holomorphic functions on $\mathbb{H}^{+}$ mapping $\mathbb{H}^{+}$ into itself. Such mappings are characterized by FitzGerald in the classical article \cite{fitzgerald}.

\begin{theorem}[FitzGerald]
Suppose $f(x)$ is a twice continuously differentiable, real-valued function
with positive first derivative on $(a,b)$. Suppose the origin is in $(a,b)$ and $f(0)=0$.
A necessary and sufficient condition that $f$ can be continued to be a univalent
analytic function of $\mathbb{H}^{+}$ onto a subset of itself that is radially convex with
respect to the origin is that the function
\begin{equation*}
\eta(x)=-\frac{f(x)}{f'(x)}
\end{equation*}
be conditionally positive definite, i.e.
\begin{equation*}
\int_{a}^{b}\int_{a}^{b}\phi(s)\frac{\eta(s)-\eta(t)}{s-t}\phi(t)dsdt\geq 0
\end{equation*}
for all real continuous $\phi$ having compact support in $(a,b)$ and satisfying $\int_{a}^{b}\phi(s)ds=0$, where $\frac{\eta(s)-\eta(s)}{s-s}$ is identified with $\eta'(s)$.
\end{theorem}


To summarize the results from the point of view of affine matrix means, Theorem~\ref{uniformconv} with Proposition~\ref{constructexp} leads us to

\begin{proposition}\label{mappingsprop}
If a matrix mean $M(A,B)$ is affine, then the exponential map and its inverse, the logarithm map of the corresponding manifold $W$ with affine connection are given as
\begin{equation}\label{mappings1}
\begin{split}
\exp_p(X)=p^{1/2}\exp_I\left(p^{-1/2}Xp^{-1/2}\right)p^{1/2}\\
\log_p(X)=p^{1/2}\log_I\left(p^{-1/2}Xp^{-1/2}\right)p^{1/2}
\end{split}
\end{equation}
for $p\in \textit{P}(n,\mathbb{C})$, where $\exp_I(X)$ and $\log_I(X)$ are analytic functions such that $\exp_I:\textit{H}(n,\mathbb{C})\mapsto \textit{P}(n,\mathbb{C})$ and $\log_I(X)$ is its inverse and $\log'_I(I)=I, \exp'_I(0)=I, \log_I(I)=0, \exp_I(0)=I$.
\end{proposition}
Note that by Weierstrass's approximation theorem we also have
\begin{equation}\label{mappings2}
\begin{split}
p^{1/2}\exp_I\left(p^{-1/2}Xp^{-1/2}\right)p^{1/2}=p\exp_I\left(p^{-1}X\right)\\
p^{1/2}\log_I\left(p^{-1/2}Xp^{-1/2}\right)p^{1/2}=p\log_I\left(p^{-1}X\right)\text{.}
\end{split}
\end{equation}

In some cases, to ensure easier reading, similarly as in the above formulas, we will denote matrices with uppercase letters which are elements of some tangent space, while at the same time we will use lowercase letters for denoting matrices which are points of a differentiable manifold.

\section{Construction of an invariant affine connection}
Let us recall the classical symmetric space $\textit{GL}(n,\mathbb{C})/\textit{U}(n,\mathbb{C})$, the cone of positive definite $n\times n$ matrices $\textit{P}(n,\mathbb{C})$ \cite{bridsonhaefliger}. This is a Lie group and the $\textit{K}=\textit{U}(n,\mathbb{C})$ isotropy group invariant inner product at the identity $I$ is $\left\langle U,V\right\rangle=Tr\left\{UV\right\}$. The tangent space, considering the Cartan decomposition of the Lie algebra, is the space of Hermitian matrices $\textit{H}(n,\mathbb{C})$. The action of the isometry group $\textit{GL}(n,\mathbb{C})$ on this manifold is $g(o)=gog^*$ and acting with left translations we can transport the inner product to any point $p$ on this manifold and we get the Riemannian metric $\left\langle U,V\right\rangle_p=Tr\left\{p^{-1}Up^{-1}V\right\}$. The exponential map is just the ordinary matrix exponential at the identity. The left invariant affine connection is
\begin{equation}
\nabla_{X_p}Y_p=DY[p][X_p]-\frac{1}{2}\left(X_pp^{-1}Y_p+Y_pp^{-1}X_p\right)\text{,}
\end{equation}
here $DY[p][X_p]$ denotes the Fr\'echet-differential of $Y$ at the point $p$ in the direction $X_p$. A well known property of this metric is that the midpoint map of the space $m(p,q)=exp_p(1/2log_p(q))$ is just the geometric mean of two positive matrices
\begin{equation}\label{geom}
G(A,B)=A^{1/2}\left(A^{-1/2}BA^{-1/2}\right)^{1/2}A^{1/2}\text{.}
\end{equation}

The question that can be asked at this point is that are there other symmetric matrix means which correspond to symmetric spaces as midpoint maps on $\textit{P}(n,\mathbb{C})$? Two other examples are known, these are the arithmetic mean $(A+B)/2$ and the harmonic mean $2(A^{-1}+B^{-1})^{-1}$. The symmetric spaces corresponding to these two means are Euclidean while the symmetric space corresponding to the geometric mean has nonpositive curvature. It has flat and negatively curved de Rham factors.

At this point we begin with the characterization of means that correspond to affine symmetric spaces in general. What we know at this point is that the two functions, which are of each others inverse, $\log_I(t)$ and $\exp_I(t)$ exist for all matrix means, as it was proved in Theorem~\ref{uniformconv}.

In \cite{guggenheimer} and \cite{helgason} there is an extensive study of affine connections on manifolds. A well known fact is that the affine connection on a manifold can be reconstructed by differentiating the parallel transport:
\begin{equation*}
\nabla_{X_p}Y_p=\lim_{t\to 0}\frac{\Gamma_t^0(\gamma)Y_{\gamma(t)}-Y_{\gamma(0)}}{t}\text{,}
\end{equation*}
where $\gamma(t)$ denotes an arbitrary smooth curve emanating from $p$ in the direction $X_p=\partial\gamma(t)/\partial t|_{t=0}$ and $\Gamma_t^s(\gamma)Y$ denotes the parallel transport of the vector field $Y$ along the curve $\gamma$ from $\gamma(t)$ to $\gamma(s)$, refer to \cite{guggenheimer,helgason}. The above limit does not depend on the curve itself, only on its initial direction vector and it depends on the vector field $Y$ in an open neighborhood of $p$. On affine symmetric spaces the parallel transport from one point to another along the connecting geodesic is given by the differential of the geodesic symmetries with a negative sign. The geodesic symmetry is given as
\begin{equation*}
S_p(q)=\exp_p(-\log_p(q))\text{.}
\end{equation*}
On affine symmetric spaces this map is an affine transformation so one can conclude that
\begin{equation}\label{paralleltransport}
\Gamma_1^0(\gamma)Y=-\left.\frac{\partial S_{\gamma(1/2)}(\exp_q(Yt))}{\partial t}\right|_{t=0}\text{,}
\end{equation}
where $\gamma(t)$ is a geodesic connecting $p=\gamma(0)$ and $q=\gamma(1)$.

We have already proved the following formulas for the exponential and logarithm maps at the end of the preceding section
\begin{equation}\label{mappings}
\begin{split}
\exp_p(X)=p^{1/2}\exp_I\left(p^{-1/2}Xp^{-1/2}\right)p^{1/2}=p\exp_I\left(p^{-1}X\right)\\
\log_p(X)=p^{1/2}\log_I\left(p^{-1/2}Xp^{-1/2}\right)p^{1/2}=p\log_I\left(p^{-1}X\right)\text{.}
\end{split}
\end{equation}
The above identities already specify the geodesic symmetries with the notation $S_I(X)=\exp_I(-\log_I(X))$ as
\begin{equation}\label{geodsymmetry}
S_p(q)=\exp_p(-\log_p(q))=p^{1/2}S_I\left(p^{-1/2}qp^{-1/2}\right)p^{1/2}=pS_I\left(p^{-1}q\right)\text{.}
\end{equation}

Now we are in position to prove the following
\begin{theorem}
Let $\textit{P}(n,\mathbb{C})$ be subset of an affine symmetric space with affine geodesic symmetries given as \eqref{geodsymmetry}. Then the invariant affine connection has the form
\begin{equation}\label{invconnection}
\nabla_{X_p}Y_p=DY[p][X_p]-\frac{\kappa}{2}\left(X_pp^{-1}Y_p+Y_pp^{-1}X_p\right)\text{,}
\end{equation}
where $\kappa=S_I''(1)/2$.
\end{theorem}
\begin{proof}
We are going to use \eqref{paralleltransport} to obtain the connection \eqref{invconnection}. We make the assumption that the geodesic symmetries are of the form \eqref{geodsymmetry}. The functions $\exp_p(X)$ and $\log_p(X)$ are of the form \eqref{mappings}, where $\exp_I(t)$ and $\log_I(t)$ are analytic functions on a disk centered around $0$ and $1$ respectively. We also have that $\log_I(1)=0$, $\exp_I(0)=1$ and furthermore
\begin{equation*}
\left.\frac{\partial\exp_I(t)}{\partial t}\right|_{t=0}=1\text{.}
\end{equation*}

First of all we have to differentiate the map $S_p(q)$ given in \eqref{geodsymmetry} to obtain $\Gamma_1^0(\gamma)Y=T_{q\to p}Y$, where $\gamma(t)$ is a geodesic connecting $p=\gamma(0)$ and $q=\gamma(1)$.
\begin{equation}\label{diffSI}
\begin{split}
&\left.\frac{\partial S_{p}(\exp_q(Yt))}{\partial t}\right|_{t=0}=\left.\frac{\partial pS_{I}(p^{-1}\exp_q(Yt))}{\partial t}\right|_{t=0}=\\
&=pDS_{I}\left[p^{-1}q\right]\left[p^{-1}Y\right]
\end{split}
\end{equation}
We used the fact that $\partial\exp_q(Yt)/\partial t|_{t=0}=Y$ which is a consequence of $\exp'_I(0)=1$.

Now we are going to differentiate the parallel transport as given by \eqref{paralleltransport} to get back the connection. We use the holomorphic functional calculus to express the Fr\'echet-differential in \eqref{diffSI} as
\begin{equation*}
DS_{I}[X][U]=\frac{1}{2\pi i}\int_g S_{I}(z)[zI-X]^{-1}U[zI-X]^{-1}dz\text{.}
\end{equation*}
It also easy to see that $DS_{I}[I][I]=S'_{I}(1)=-1$, so we may express the limit \eqref{paralleltransport} by the following differential
\begin{equation*}
\begin{split}
&\nabla_{\gamma'(0)}Y_{\gamma(0)}=-\left.\frac{\partial \gamma(t/2)DS_{I}\left[\gamma(t/2)^{-1}\gamma(t)\right]\left[\gamma(t/2)^{-1}Y_{\gamma(t)}\right]}{\partial t}\right|_{t=0}=
\end{split}
\end{equation*}
we massage this further by using the holomorphic functional calculus
\begin{equation*}
\begin{split}
&=-\frac{\partial}{\partial t}\gamma(t/2)\frac{1}{2\pi i}\int_g S_{I}(z)[zI-\gamma(t/2)^{-1}\gamma(t)]^{-1}\gamma(t/2)^{-1}Y_{\gamma(t)}\times\\
&\left.[zI-\gamma(t/2)^{-1}\gamma(t)]^{-1}dz\right|_{t=0}=-\frac{1}{2}\gamma'(0)\gamma(0)^{-1}Y_{\gamma(0)}DS_I[I][I]-\\
&-\gamma(0)\frac{1}{2\pi i}\int_g S_{I}(z)\left\{[zI-I]^{-1}\frac{1}{2}\gamma(0)^{-1}\gamma'(0)[zI-I]^{-1}\gamma(0)^{-1}Y_{\gamma(0)}[zI-I]^{-1}\right.+\\
&+[zI-I]^{-1}\gamma(0)^{-1}Y_{\gamma(0)}[zI-I]^{-1}\frac{1}{2}\gamma(0)^{-1}\gamma'(0)[zI-I]^{-1}+\\
&+[zI-I]^{-1}\left[-\gamma(0)^{-1}\frac{1}{2}\gamma'(0)\gamma(0)^{-1}Y_{\gamma(0)}+\gamma(0)^{-1}DY[\gamma(0)][\gamma'(0)]\right]\times\\
&\left.[zI-I]^{-1}\right\}dz=
\end{split}
\end{equation*}
by using the fact that $DS_I[I][I]$ and $[zI-I]^{-1}$ commutes with every matrix we get
\begin{equation*}
\begin{split}
&=-\frac{DS_I[I][I]}{2}\gamma'(0)\gamma(0)^{-1}Y_{\gamma(0)}-\\
&-\gamma(0)\frac{1}{2\pi i}\int_g \frac{S_{I}(z)dz}{(z-1)^3}\frac{1}{2}\gamma(0)^{-1}\gamma'(0)\gamma(0)^{-1}Y_{\gamma(0)}-\\
&-\gamma(0)\frac{1}{2\pi i}\int_g \frac{S_{I}(z)dz}{(z-1)^3}\frac{1}{2}\gamma(0)^{-1}Y_{\gamma(0)}\gamma'(0)\gamma(0)^{-1}-\\
&-\gamma(0)\frac{1}{2\pi i}\int_g \frac{S_{I}(z)dz}{(z-1)^2}\left[-\frac{1}{2}\gamma(0)^{-1}\gamma'(0)\gamma(0)^{-1}Y_{\gamma(0)}+\gamma(0)^{-1}DY[\gamma(0)][\gamma'(0)]\right]
\end{split}
\end{equation*}
at this point we use the integral representation
\begin{equation*}
S_I^{(j)}(1)=\frac{j!}{2\pi i}\int_g \frac{S_{I}(z)}{(z-1)^{j+1}}dz
\end{equation*}
to further simplify the above.
\begin{equation*}
\begin{split}
&\nabla_{\gamma'(0)}Y_{\gamma(0)}=-\frac{S_I''(1)}{4}\left[\gamma'(0)\gamma(0)^{-1}Y_{\gamma(0)}+Y_{\gamma(0)}\gamma(0)^{-1}\gamma'(0)\right]-\\
&-\frac{S_I'(1)}{2}\gamma'(0)\gamma(0)^{-1}Y_{\gamma(0)}-\frac{S_I'(1)}{2}\left[-\gamma'(0)\gamma(0)^{-1}Y_{\gamma(0)}+2DY[\gamma(0)][\gamma'(0)]\right]=\\
&=-S_I'(1)DY[\gamma(0)][\gamma'(0)]-\frac{S_I''(1)}{4}\left[\gamma'(0)\gamma(0)^{-1}Y_{\gamma(0)}+Y_{\gamma(0)}\gamma(0)^{-1}\gamma'(0)\right]\text{.}
\end{split}
\end{equation*}
So we have that $\kappa=S_I''(1)/2$.

\end{proof}

The above clearly tells us that all symmetric spaces occuring in such a way that their geodesic division maps are matrix means, have invariant affine connections in the form \eqref{invconnection}. We are going to study these connections as $\kappa$ being a parameter. We will find out later for which values of $\kappa$ are these spaces symmetric. Also for arbitrary real $\kappa$ \eqref{invconnection} defines an affine connection with corresponding exponential and logarithm map which are of the form \eqref{mappings} as we will see later. We will also determine if these connections are metric or not.

\section{Properties of these affine connections}
In this section we study the connections
\begin{equation}\label{invconnection2}
\nabla_{X_p}Y_p=DY[p][X_p]-\frac{\kappa}{2}\left(X_pp^{-1}Y_p+Y_pp^{-1}X_p\right)
\end{equation}
for $p\in\textit{P}(n,\mathbb{C})$ and vector fields $X_p,Y_p\in\textit{H}(n,\mathbb{C})$ on the smooth manifold $\textit{P}(n,\mathbb{C})$ with tangent bundle $\textit{H}(n,\mathbb{C})$. It is easy to see that indeed these connections are affine and analytic for real $\kappa$.

We can fix a coordinate frame by taking the basis $E_i\in\textit{H}(n,\mathbb{C})$, where $i$ indices over the set of distinct Hermitian matrices which have zero entries, excluding exactly the entry $[E_i]_{kl}=1$ and its transpose $[E_i]_{lk}=1$. If we equip $\textit{H}(n,\mathbb{C})$ with the inner product $\left\langle X,Y\right\rangle=Tr\left\{XY\right\}$, then the $E_i$ form an orthonormal basis of $\textit{H}(n,\mathbb{C})$. The dimension of $\textit{H}(n,\mathbb{C})$ is $n(n+1)/2$ such as the dimension of the smooth manifold $\textit{P}(n,\mathbb{C})$. In this coordinate frame the Christoffel symbols are given as
\begin{equation}\label{christoffel}
\Gamma^{k}_{ij}E_k=-\frac{\kappa}{2}\left(E_ip^{-1}E_j+E_jp^{-1}E_i\right),
\end{equation}
where we used the Einstein summation convention for repeated covariant and contravariant indices. Given an arbitrary connection $\nabla$ the geodesic equations corresponding to it are given as
\begin{equation}
\nabla_{\dot{\gamma}(t)}\dot{\gamma}(t)=0
\end{equation}
with given initial conditions $\gamma(0)$ and $\dot{\gamma}(0)$, for all $t\in[0,T)$. I.e. the curve $\gamma$ must be parallel along itself.

\begin{theorem}\label{geodesics}
The geodesic equations corresponding to the affine connections \eqref{invconnection2} are
\begin{equation}\label{geodesiceq}
\ddot{\gamma}=\kappa\dot{\gamma}\gamma^{-1}\dot{\gamma}\text{.}
\end{equation}
The solutions of these equations with initial conditions $\gamma(0)=p$, $\dot{\gamma}(0)=X$ are the folowing one parameter family of functions
\begin{equation}
\gamma(t)=\exp_p(Xt)=p^{1/2}\exp_I\left(p^{-1/2}Xp^{-1/2}t\right)p^{1/2}
\end{equation}
where
\begin{equation}
\exp_I(X)=
    \begin{cases}
    [(1-\kappa)X+1]^{\frac{1}{1-\kappa}}&\text{if $\kappa\neq1$,}\\
    \exp(X)&\text{else.}
    \end{cases}
\end{equation}
\end{theorem}
\begin{proof}
For the connections \eqref{invconnection2} it is easy to see that the corresponding $\nabla_{\dot{\gamma}(t)}\dot{\gamma}(t)=0$ geodesic equations are \eqref{geodesiceq}.

Let us first consider the case when $\gamma(0)=p=I=\dot{\gamma}(0)=X$. Then
it is enough to solve the equation \eqref{geodesiceq} for real numbers. Therefore the equation takes the form
\begin{equation}\label{geodesiceq2}
\exp_I''(t)=\kappa\exp_I'(t)^2\exp_I(t)^{-1}\text{.}
\end{equation}
If we transform the equation to the inverse function of $\exp_I(t)$ which will be the logarithm map $\log_I(t)$, then we get a separable first order differential equation of the form 
\begin{equation*}
\log_I''(t)=-\kappa\log_I'(t)t^{-1}\text{.}
\end{equation*}
Solving the above we get the logarithm map as
\begin{equation*}
\log_I(X)=
    \begin{cases}
    \frac{X^{1-\kappa}-1}{1-\kappa}&\text{if $\kappa\neq1$,}\\
    \log(X)&\text{else.}
    \end{cases}
\end{equation*}
From this by inverting the above function we get the assertion for real numbers.

Now we check by substitution into \eqref{geodesiceq} that the curve
\begin{equation*}
\gamma(t)=p^{1/2}\exp_I\left(p^{-1/2}Xp^{-1/2}t\right)p^{1/2}
\end{equation*}
is also a solution of the equations \eqref{geodesiceq}, since the function $\exp_I$ is analytic. Indeed
\begin{equation*}
\begin{split}
\dot{\gamma}(t)&=Xp^{-1/2}\exp_I'\left(p^{-1/2}Xp^{-1/2}t\right)p^{1/2}=p^{1/2}\exp_I'\left(p^{-1/2}Xp^{-1/2}t\right)p^{-1/2}X\\
\ddot{\gamma}(t)&=Xp^{-1/2}\exp_I''\left(p^{-1/2}Xp^{-1/2}t\right)p^{-1/2}X
\end{split}
\end{equation*}
and after substitution we get
\begin{equation*}
\begin{split}
Xp^{-1/2}\exp_I''\left(p^{-1/2}Xp^{-1/2}t\right)p^{-1/2}X=\kappa Xp^{-1/2}\exp_I'\left(p^{-1/2}Xp^{-1/2}t\right)\times\\
\exp_I\left(p^{-1/2}Xp^{-1/2}t\right)^{-1}\exp_I'\left(p^{-1/2}Xp^{-1/2}t\right)p^{-1/2}X
\end{split}
\end{equation*}
which is fulfilled since
\begin{equation*}
\exp_I''\left(p^{-1/2}Xp^{-1/2}t\right)=\kappa\exp_I'\left(p^{-1/2}Xp^{-1/2}t\right)^2\exp_I\left(p^{-1/2}Xp^{-1/2}t\right)^{-1}
\end{equation*}
holds by the functional calculus for $\exp_I$ and its derivatives and \eqref{geodesiceq2}.

\end{proof}

\begin{corollary}
The exponential and logarithm map for the affine connections \eqref{invconnection2} are given in the form
\begin{equation}\label{mappingsaff}
\begin{split}
\exp_p(X)=p^{1/2}\exp_I\left(p^{-1/2}Xp^{-1/2}\right)p^{1/2}\\
\log_p(X)=p^{1/2}\log_I\left(p^{-1/2}Xp^{-1/2}\right)p^{1/2},
\end{split}
\end{equation}
where
\begin{equation}
\begin{split}
\exp_I(X)&=
    \begin{cases}
    [(1-\kappa)X+1]^{\frac{1}{1-\kappa}}&\text{if $\kappa\neq1$,}\\
    \exp(X)&\text{else,}
    \end{cases}
\\
\log_I(X)&=
    \begin{cases}
    \frac{X^{1-\kappa}-1}{1-\kappa}&\text{if $\kappa\neq1$,}\\
    \log(X)&\text{else.}
    \end{cases}
\end{split}
\end{equation}
The affine matrix means which induce these affinely connected manifolds are
\begin{equation}\label{means}
\begin{split}
M_t(X,Y)=&\exp_X\left(t\log_X(Y)\right)=\\
    =&\begin{cases}
    X^{1/2}\left[(1-t)I+t\left(X^{-1/2}YX^{-1/2}\right)^{1-\kappa}\right]^{\frac{1}{1-\kappa}}X^{1/2}&\text{if $\kappa\neq1$,}\\
    X^{1/2}\left(X^{-1/2}YX^{-1/2}\right)^tX^{1/2}&\text{else,}
    \end{cases}
\end{split}
\end{equation}
if $\kappa\in [0,2]$, for other values of $\kappa$ the functions \eqref{means} are not matrix means.
\end{corollary}
\begin{proof}
The first part of the assertion is clear, the second part follows from the fact that \eqref{means} are matrix means if and only if $\kappa\in [0,2]$, refer to Example~\ref{ex1}.

\end{proof}

\begin{remark}
The one parameter family of affine matrix means \eqref{means} seems to have a singularity at $\kappa=1$, however it is known that the singularity is removable and indeed as $\kappa\to 1$ we get the matrix geometric mean as the limit. This phenomenon has already been investigated in \cite{limpalfia}. In that paper the same one parameter family of matrix means were considered under the name of matrix power means.

If $\kappa=0$ we get back the arithmetic mean as the midpoint operation, and the weighted arithmetic mean
\begin{equation}\label{arithmetic}
A_t(A,B)=(1-t)A+tB
\end{equation}
is the geodesic line connecting $A$ and $B$ with respect to the metric $\left\langle X,Y\right\rangle_p=Tr\left\{X^{*}Y\right\}$. If $\kappa=2$ we get back the harmonic mean as the midpoint operation, and the weighted harmonic mean
\begin{equation}\label{harmonic}
H_t(A,B)=\left((1-t)A^{-1}+tB^{-1}\right)^{-1}
\end{equation}
is also a geodesic with respect to the metric $\left\langle X,Y\right\rangle_p=Tr\left\{p^{-2}Xp^{-2}Y\right\}$. We have already mentioned that the second metric is isometric to the first one, so it is also Euclidean.

In the case when $\kappa=1$ the midpoint is the geometric mean and the geodesics are given by the weighted geometric mean
\begin{equation}\label{geometric}
G_t(A,B)=A^{1/2}\left(A^{-1/2}BA^{-1/2}\right)^tA^{1/2}\text{.}
\end{equation}
The corresponding Riemannian metric is $\left\langle X,Y\right\rangle_p=Tr\left\{p^{-1}Xp^{-1}Y\right\}$. This manifold, which is the symmetric space $\textit{GL}(n,\mathbb{C})/\textit{U}(n,\mathbb{C})$, is nonpositively curved while the other two has zero curvature.
\end{remark}

In the paper \cite{limpalfia} matrix power means $P_s(w_1,\ldots,w_k;A_1,\ldots,A_k)$ are defined as the unique positive definite solution of the equations
\begin{equation}\label{powermeanequ}
X=\sum_{i=1}^{k}w_{i}G_s(X,A_i)
\end{equation}
where $s\in[-1,1], w_i>0, \sum_{i=1}^kw_i=1$ and $A_i\in\textit{P}(n,\mathbb{C})$. Existence and uniqueness of the solutions follow from the fact that the function
\begin{equation*}
f(X)=\sum_{i=1}^{k}w_{i}G_s(X,A_i)
\end{equation*}
is a strict contraction for $s\in[-1,1],s\neq 0$ with respect to Thompson's part metric \cite{limpalfia}. In the case $k=2$ we get back the affine matrix means \eqref{means} with $s=\kappa-1$ and $t=w_2$.

\begin{corollary}
With the identification $s=\kappa-1$, the two-variable matrix power means $P_s(w_1,w_2;A_1,A_2)$ are geodesic lines, with $w_2$ being the arc-length parameter, of the affinely connected spaces with affine connections \eqref{invconnection2}.
\end{corollary}

The arithmetic \eqref{arithmetic}, harmonic \eqref{harmonic} and geometric \eqref{geometric} means have nice chracterizations and extensions to several variables as the center of mass or Karcher mean of the corresponding manifolds \cite{palfia2,palfia3,bhatiaholbrook,limpalfia}. I.e.
\begin{equation}\label{karchermean}
\Lambda(w_1,\ldots,w_k;A_{1},\dots,A_{k})=\underset{X\in \textit{P}(n,\mathbb{C})}{\argmin}\sum_{i=1}^n w_id^2(X,A_i),
\end{equation}
where $d(\cdot,\cdot)$ is a Riemannian metric given as
\begin{equation}\label{riemmetric}
d^2(X,Y)=\left\langle \log_X(Y),\log_X(Y)\right\rangle_X
\end{equation}
where $\left\langle \cdot,\cdot\right\rangle_X$ is one of the corresponding metrics given above for the arithmetic \eqref{arithmetic}, harmonic \eqref{harmonic} and geometric \eqref{geometric} means, and $\log_X(Y)$ are the corresponding logarithm maps \eqref{mappingsaff} (for $\kappa=0,2,1$ respectively). It is well known that in geodesically convex neighborhoods on a Riemannian manifold \eqref{karchermean} has a unique solution \cite{karcher,palfia2}. The solution can be expressed by taking the gradient of the cost function on the right hand side of \eqref{karchermean} \cite{karcher} and then one arrives at
\begin{equation}\label{karcherequ}
\sum_{i=1}^{n}w_i\log_X(A_{i})=0.
\end{equation}
The unique solution of this equation can be expressed in closed form in the case of the arithmetic and harmonic means, since the corresponding manifolds are Euclidean. The solutions are just the multivariable weighted arithmetic $\sum_{i=1}^kw_iA_i$ and harmonic means $\left(\sum_{i=1}^kw_iA_i^{-1}\right)^{-1}$ \cite{palfia3}. These functions are monotone in their variables with respect to the positive definite order and have some other desirable properties \cite{limpalfia}. These two cases are well known and of less interest, however the same situation is of much more interest in the case of the geometric mean. In this case the unique solution of the minimization problem \eqref{karchermean} cannot be expressed easily in closed form since the corresponding Riemannian manifold is no longer flat. The corresponding equation for the gradient \eqref{karcherequ} is given in the form
\begin{equation}\label{karcherequ2}
\sum_{i=1}^{n}w_i\log(X^{1/2}A_{i}^{-1}X^{1/2})=0
\end{equation}
and usually this equation is called the Karcher equation \cite{limpalfia} and the corresponding unique positive definite solution \eqref{karchermean} the Karcher mean. Several properties of this mean were open problems, for example its monotonicity with respect to the positive definite order, however this and other key properties of the mean were proved by using different techniques \cite{lawsonlim,limpalfia,bhatiakarandikar}. One of the techniques given in \cite{limpalfia} is based on the matrix power means $P_s(w_1,\ldots,w_k;A_1,\ldots,A_k)$. These means are given as the unique positive definite solutions of \eqref{powermeanequ}. The following result provides a geometric characterization of matrix power means.

\begin{proposition}\label{powerkarcher}
The matrix power means $P_s(w_1,\ldots,w_k;A_1,\ldots,A_k)$ for $s\in[-1,1]$ are the unique positive definite solutions of the Karcher equations
\begin{equation}\label{karcherequ3}
\sum_{i=1}^{n}w_i\log_X(A_{i})=0
\end{equation}
where $\log_X(A_{i})$ are the logarithm maps \eqref{mappingsaff} corresponding to the affine family \eqref{invconnection2} with parameter identification $s=\kappa-1$.
\end{proposition}
\begin{proof}
The defining equation \eqref{powermeanequ} of matrix power means with $s\neq 0$ is equivalent to
\begin{eqnarray*}
\sum_{i=1}^{k}w_{i}(G_s(X,A_i)-X)&=&0\\
\sum_{i=1}^{n}w_{i}X^{1/2}\left[\left(X^{-1/2}A_{i}X^{-1/2}\right)^s-I\right]X^{1/2}&=&0\\
\sum_{i=1}^{n}w_{i}X^{1/2}\frac{\left(X^{-1/2}A_{i}X^{-1/2}\right)^s-I}{s}X^{1/2}&=&0,
\end{eqnarray*}
which is by \eqref{mappingsaff} equivalent to
\begin{equation*}
\sum_{i=1}^{n}w_{i}\log_X(A_i)=0.
\end{equation*}
The case $s=0$ is just the case \eqref{karcherequ2}.

\end{proof}

\begin{remark}
By the continuity of fixed points of pointwise continuous families of strict contractions \cite{limpalfia}, the unique solution of \eqref{karcherequ3} varies continuously with respect to the parameter $s$. The singularity at $s=0$ is known to be removable and the limit is the Karcher mean \cite{limpalfia}.
\end{remark}

Now since the Karcher equations \eqref{karcherequ3} admit unique positive definite solutions the obvious question arises whether there are Riemannian metrics corresponding to other values of $\kappa$? Also for what other values of $\kappa$ is the manifold $\textit{P}(n,\mathbb{C})$ with affine connection \eqref{invconnection2} a symmetric space? The full solution of these questions requires the study of the curvature tensors and holonomy groups which is postponed to the last section. At this point we prove some other results which gets us closer to this metrizability problem of the affine connections \eqref{invconnection2}. First of all we compute the parallel transport over a geodesic connecting an arbitrary point and the identity. The parallel transport of a vector $Y_{\gamma(0)}$ given in the tangent space at $\gamma(0)$ with respect to the connection $\nabla$ along the curve $\gamma(t)$ is defined to be the vector field $Y_{\gamma(t)}$ which is the solution of the ODE
\begin{equation*}
\nabla_{\dot{\gamma}(t)}Y=0.
\end{equation*}

\begin{proposition}
Let $c(t)$ be a geodesic with respect to the connection \eqref{invconnection2} and $c(0)=I,c(1)=p$. Then the unique solution of $\nabla_{\dot{c}(t)}Y=0$ with respect to the connection \eqref{invconnection2} and the initial condition $Y_{c(0)}=Y_0$ is the vector field
\begin{equation}
Y(t)=c(t)^{\frac{\kappa}{2}}Y_0c(t)^{\frac{\kappa}{2}}\text{.}
\end{equation}
\end{proposition}
\begin{proof}
We have to integrate the equation $\nabla_{c'(t)}Y_{c(t)}=0$. This is equivalent to
\begin{equation*}
DY[c(t)][c'(t)]-\frac{\kappa}{2}\left(c'(t)c(t)^{-1}Y_{c(t)}+Y_{c(t)}c(t)^{-1}c'(t)\right)=0\text{.}
\end{equation*}
We are looking for the solution $Y_{c(t)}=Y(t)$ in the form
\begin{equation*}
Y(t)=f(c(t))Y_0f(c(t))\text{,}
\end{equation*}
for some analytic function $f(x)$. We have for the Fr\'echet-differential
\begin{equation*}
DY[c(t)][c'(t)]=\frac{dY(t)}{dt}=\frac{df(c(t))}{dt}Y_0f(c(t))+f(c(t))Y_0\frac{df(c(t))}{dt}\text{.}
\end{equation*}
Now substituting into the equation of the parallel transport above, we get
\begin{equation*}
\frac{df(c(t))}{dt}Y_0f(c(t))+f(c(t))Y_0\frac{df(c(t))}{dt}=\frac{\kappa}{2}\left(c'(t)c(t)^{-1}Y_{c(t)}+Y_{c(t)}c(t)^{-1}c'(t)\right)\text{.}
\end{equation*}
Since $c(t)=\exp_I(t\log_I(p))$, it has a power series expansion, as has $f(x)$, so we have by commutativity that
\begin{equation*}
\frac{\kappa}{2}c'(t)c(t)^{-1}f(c(t))=\frac{df(c(t))}{dt}=Df[c(t)][c'(t)]=f'(c)c'(t)\text{.}
\end{equation*}
Since everything on the left and right hand side commutes with one another, we arrive at the following separable differential equation
\begin{equation*}
\frac{\kappa}{2}c^{-1}=f'(c)f(c)^{-1}\text{,}
\end{equation*}
which has its solution in the form $f(c)=c^{\kappa/2}$.

\end{proof}

On a Riemannian manifold the length of vectors are left invariant by parallel transports with respect to the Levi-Civita connection due to the Fundamental Theorem of Riemannian geometry \cite{kobayashi1}. It is easy to check that the connections \eqref{invconnection2} are symmetric and torsion free so any of them can possibly be a Levi-Civita connection of a Riemannian manifold. So by the above proposition we should look for the Riemannian metrics in the form
\begin{equation}\label{ptransportproduct}
\left\langle p^{-\kappa/2}Xp^{-\kappa/2},p^{-\kappa/2}Yp^{-\kappa/2}\right\rangle_{\kappa}
\end{equation}
for some positive definite bilinear forms $\left\langle \cdot,\cdot \right\rangle_{\kappa}$ given on the tangent space at $I$. Later in the last sections we prove that all affine matrix means are actually matrix power means, i.e. we do not have to look for other connections than \eqref{invconnection2}.

\section{Contraction property of operator means}
This section we prove further properties of matrix means using explicitly the integral characterizations given in Proposition~\ref{harmonicreprprop}. We will use the results given in this section to generalize the construction of matrix power means to all possible matrix means in later sections.

Let $E$ be a Hilbert space, $\mathfrak{B}(E)$ denote the Banach space of bounded linear operators, $S(E)$ denote the Banach space of bounded linear self-adjoint operators and $\mathbb{P}\subseteq S(E)$ the cone of positive definite operators. On $\mathbb{P}$ we have the positive definite order similarly to the finite dimensional case which means that $A\leq B$ for $A,B\in\mathbb{P}$ if and only if $B-A\in\mathbb{P}$. It is also easy to see that if for $A,B\in S(E)$ and $0\leq A\leq B$ then also $\left\|A\right\|\leq \left\|B\right\|$ \cite{lawsonlim1}. We will use the notation $[A,B]$ for the order interval generated by $A\leq B$, i.e. $[A,B]=\{X\in\mathbb{P}:A\leq X\leq B\}$. We also have that $\mathbb{P}=\bigcup_{k=1}^{\infty}\left[\frac{1}{k}I,kI\right]$. On $\mathbb{P}$ this partial ordering induces a complete metric space structure \cite{thompson}. The Thompson or part metric is defined as
\begin{equation}\label{thompsonmetric}
d_\infty(A,B)=\max\left\{\log M(A/B),M(B/A)\right\}
\end{equation}
for any $A,B\in\mathbb{P}$, where $M(A/B)=\inf\{\alpha:A\leq\alpha B\}$. The metric space $(\mathbb{P},d_\infty)$ is complete and has some several other nice properties.
\begin{lemma}[Lemma 10.1 in \cite{lawsonlim0}]\label{thompsonproperties}
\begin{enumerate}
	\item $d_\infty(rA,rB)=d_\infty(A,B)$ for any $r>0$,
	\item $d_\infty(A^{-1},B^{-1})=d_\infty(A,B)$,
	\item $d_\infty(MAM^*,MBM^*)=d_\infty(A,B)$ for all $M\in \mathrm{GL}(E)$ where $\mathrm{GL}(E)$ denotes the Banach-Lie group of all invertible bounded linear operators on $E$,
	\item $d_\infty(\sum_{i=1}^kt_iA_i,\sum_{i=1}^kt_iB_i)\leq \max_{1\leq i\leq k}d_\infty(A_i,B_i)$ where $t_i>0$,
	\item $e^{-d_\infty(A,B)}B\leq A\leq e^{d_\infty(A,B)}B$ and $e^{-d_\infty(A,B)}A\leq B\leq e^{d_\infty(A,B)}A$.
\end{enumerate}
\end{lemma}

Property 4. in Lemma~\ref{thompsonproperties} is important for us, but we need a refined, weighted version of it.

\begin{proposition}\label{weightedthompson}
Let $A_i,B_i\in\mathbb{P}$, $1\leq i\leq k$ and suppose that $d_\infty(A_m,B_m)\geq d_\infty(A_i,B_i)$. Then we have
\begin{equation*}
\begin{split}
e^{d_{\infty}\left(\sum_{i=1}^kA_i,\sum_{i=1}^kB_i\right)}\leq&\max\left\{\frac{\sum_{i=1}^ke^{d_\infty(A_i,B_i)}e^{-d_{\infty}(A_m,A_i)}}{\sum_{i=1}^ke^{-d_{\infty}(A_m,A_i)}},\right.\\
&\left.\frac{\sum_{i=1}^ke^{d_\infty(A_i,B_i)}e^{-d_{\infty}(B_m,B_i)}}{\sum_{i=1}^ke^{-d_{\infty}(B_m,B_i)}}\right\}.
\end{split}
\end{equation*}
\end{proposition}
\begin{proof}
Let $\alpha_i=e^{d_\infty(A_i,B_i)}$. We seek the infimum of all $\beta\geq 0$ such that both
\begin{equation*}
\begin{split}
\sum_{i=1}^kA_i\leq \beta\sum_{i=1}^kB_i\\
\sum_{i=1}^kB_i\leq \beta\sum_{i=1}^kA_i
\end{split}
\end{equation*}
are satisfied. We also have that 
\begin{equation*}
\begin{split}
\sum_{i=1}^kA_i\leq \sum_{i=1}^k\alpha_iB_i\\
\sum_{i=1}^kB_i\leq \sum_{i=1}^k\alpha_iA_i.
\end{split}
\end{equation*}
From this it follows that the infimum of the seeked $\beta$ are bounded above by the infimum of all $\beta\geq 0$ satisfying both
\begin{equation}\label{weightedthompson:eq1}
\begin{split}
\sum_{i=1}^kA_i\leq \sum_{i=1}^k\alpha_iB_i\leq \beta\sum_{i=1}^kB_i\\
\sum_{i=1}^kB_i\leq \sum_{i=1}^k\alpha_iA_i\leq \beta\sum_{i=1}^kA_i.
\end{split}
\end{equation}
The first inequality above is equivalent to
\begin{equation*}
0\leq \sum_{i=1}^k(\beta-\alpha_i)B_i.
\end{equation*}
Now by Property 4. in Lemma~\ref{thompsonproperties} we have the natural bound $\beta\leq \max_{1\leq i\leq n}\alpha_i$. So we may try to find a better bound by assuming that $\beta\geq \alpha_i$ for $2\leq i\leq n$, where without loss of generality $\alpha_1$ is assumed to be the maximal of all $\alpha_i$. Using the assumption on $\beta$ we get that we seek the infimum of all $\beta\geq 0$ such that
\begin{equation*}
\sum_{i=2}^k(\beta-\alpha_i)B_i\geq (\alpha_1-\beta)B_1.
\end{equation*}
The infimum here is therefore bounded above again by the infimum of all $\beta\geq 0$ such that
\begin{equation*}
\sum_{i=2}^k(\beta-\alpha_i)e^{-d_{\infty}(B_1,B_i)}\geq (\alpha_1-\beta).
\end{equation*}
This is equivalent to
\begin{equation*}
\sum_{i=1}^k(\beta-\alpha_i)e^{-d_{\infty}(B_1,B_i)}\geq 0,
\end{equation*}
in other words we have that
\begin{equation*}
\beta\geq \frac{\sum_{i=1}^k\alpha_ie^{-d_\infty(B_1,B_i)}}{\sum_{i=1}^ke^{-d_\infty(B_1,B_i)}}=\frac{\sum_{i=1}^ke^{d_\infty(A_i,B_i)}e^{-d_\infty(B_1,B_i)}}{\sum_{i=1}^ke^{-d_\infty(B_1,B_i)}}.
\end{equation*}
Doing the same calculation ($A_i$ in place of $B_i$) by starting with the second inequality in \eqref{weightedthompson:eq1} we get
\begin{equation*}
\beta\geq \frac{\sum_{i=1}^ke^{d_\infty(A_i,B_i)}e^{-d_\infty(A_1,A_i)}}{\sum_{i=1}^ke^{-d_\infty(A_1,A_i)}}.
\end{equation*}
This means that
\begin{equation*}
\begin{split}
e^{d_{\infty}\left(\sum_{i=1}^kA_i,\sum_{i=1}^kB_i\right)}\leq&\max\left\{\frac{\sum_{i=1}^ke^{d_\infty(A_i,B_i)}e^{-d_{\infty}(A_1,A_i)}}{\sum_{i=1}^ke^{-d_{\infty}(A_1,A_i)}},\right.\\
&\left.\frac{\sum_{i=1}^ke^{d_\infty(A_i,B_i)}e^{-d_{\infty}(B_1,B_i)}}{\sum_{i=1}^ke^{-d_{\infty}(B_1,B_i)}}\right\}.
\end{split}
\end{equation*}
which is what we wanted to prove.

\end{proof}

Let $\overline{B}_A(r)=\{X\in\mathbb{P}:d_{\infty}(A,X)\leq r\}$.


\begin{lemma}\label{arithmeticcontraction}
Let $a,b>0$ be real numbers. Then the mappings $h^+_{a,b,A}(B)=aA+bB$ and $h^-_{a,b,A}(B)=(aA^{-1}+bB^{-1})^{-1}$ are strict contractions on every $\overline{B}_A(r)$ for all $r<\infty$, i.e. for all $X,Y\in\overline{B}_A(r)$
\begin{equation*}
d_\infty(h^\pm_{a,b,A}(X),h^\pm_{a,b,A}(Y))\leq \rho d_\infty(X,Y)
\end{equation*}
where
\begin{equation*}
\rho=\frac{\log\frac{e^{-r-|\log a-\log b|}+e^{2r}}{e^{-r-|\log a-\log b|}+1}}{2r}.
\end{equation*}
\end{lemma}
\begin{proof}
It suffices to prove the above for $h^+_{a,b,A}(B)$, since then the same follows for $h^-_{a,b,A}(B)$ by the inversion invariancy of the metric $d_\infty$. Also by property 3 in Lemma~\ref{thompsonproperties} it is enough to prove for the case when $A=I$. Let $X,Y\in\overline{B}_I(r)$. By Proposition~\ref{weightedthompson} we have that
\begin{equation*}
\begin{split}
e^{d_\infty(h^+_{a,b,I}(X),h^+_{a,b,I}(Y))}\leq\max\left\{\frac{e^{-d_\infty(bX,aI)}+e^{d_\infty(X,Y)}}{e^{-d_\infty(bX,aI)}+1},\frac{e^{-d_\infty(bY,aI)}+e^{d_\infty(X,Y)}}{e^{-d_\infty(bY,aI)}+1}\right\}\\
=\max\left\{\frac{e^{-d_\infty((b/a)X,I)}+e^{d_\infty(X,Y)}}{e^{-d_\infty((b/a)X,I)}+1},\frac{e^{-d_\infty((b/a)Y,I)}+e^{d_\infty(X,Y)}}{e^{-d_\infty((b/a)Y,I)}+1}\right\}.
\end{split}
\end{equation*}
Since $X,Y\in\overline{B}_A(r)$ we have that
\begin{equation*}
\begin{split}
e^{-r}I\leq X,&Y\leq e^rI\\
\frac{b}{a}e^{-r}I\leq \frac{b}{a}X,&\frac{b}{a}Y\leq \frac{b}{a}e^rI,
\end{split}
\end{equation*}
which means that
\begin{equation*}
d_\infty((b/a)X,I),d_\infty((b/a)Y,I)\leq r+\left|\log\frac{a}{b}\right|.
\end{equation*}
With $R=r+\left|\log\frac{a}{b}\right|$ this means that
\begin{equation*}
e^{d_\infty(h^+_{a,b,I}(X),h^+_{a,b,I}(Y))}\leq\frac{e^{-R}+e^{d_\infty(X,Y)}}{e^{-R}+1}.
\end{equation*}
So we seek $0<\rho<1$ such that
\begin{equation*}
d_\infty(h^+_{a,b,I}(X),h^+_{a,b,I}(Y))\leq\log\left(\frac{e^{-R}+e^{d_\infty(X,Y)}}{e^{-R}+1}\right)\leq \rho d_\infty(X,Y)
\end{equation*}
for all $X,Y\in\overline{B}_I(r)$, i.e. $d_\infty(X,Y)\leq 2r$. It therefore suffices to find the maximum of the function
\begin{equation*}
f(x)=\frac{\log\left(\frac{e^{-R}+e^{x}}{e^{-R}+1}\right)}{x}
\end{equation*}
on the interval $(0,2r]$. Since $f'(x)>0$ on the interval $(0,2r]$, the maximum is $f(2r)$, so
\begin{equation*}
\rho=\frac{\log\left(\frac{e^{-R}+e^{2r}}{e^{-R}+1}\right)}{2r}
\end{equation*}
suffices.

\end{proof}

\begin{remark}
By Proposition~\ref{weightedthompson} it is clear that the functions $h^\pm_{a,b,A}(X)$ for all $a,b\geq 0$ are nonexpansive on the whole $\mathbb{P}$, i.e.
\begin{equation*}
d_\infty(h^\pm_{a,b,A}(X),h^\pm_{a,b,A}(Y))\leq d_\infty(X,Y).
\end{equation*}
\end{remark}

%

\begin{remark}
By Lemma~\ref{arithmeticcontraction} it follows that the weighted arithmetic $(1-s)A+sB$ and harmonic $((1-s)A^{-1}+sB^{-1})^{-1}$ means are strict contractions on $\overline{B}_A(r)$ for all $r<\infty$ and $s\in(0,1)$. The contraction coefficients $\rho$ are striclty less then $1$ for all $s\in(0,1)$, in general $\rho$ monotonically increases as $s\geq 1/2$ increases and $\rho\to 1$ as $s\to 1-$. Similarly as $s\leq 1/2$ decreases, $\rho$ monotonically increases and $\rho\to 1$ as $s\to 0+$. The cases $s=0,1$ are degenerate, $s=1$ gives the right trivial mean $M(A,B)=B$ that is nonexpansive, i.e.
\begin{equation*}
d_\infty(M(A,X),M(A,Y))\leq d_\infty(X,Y),
\end{equation*}
while $s=0$ is the left trivial mean $M(A,B)=A$ and it has contraction coefficient $0$ on all of $\mathbb{P}$.
\end{remark}

These preliminary results yield the following main result.

\begin{theorem}\label{contract}
Let $M\in\mathfrak{M}$ and $f(X)=M(A,X)$. If $M$ is not the right trivial mean (i.e. $M(A,B)\neq B$) then the mapping $f(X)$ is a strict contraction on $\overline{B}_A(r)$ for all $r<\infty$, i.e. there exists $0<\rho_r<1$ such that
\begin{equation*}
d_\infty(f(X),f(Y))\leq\rho_r d_\infty(X,Y)
\end{equation*}
for all $X,Y\in\overline{B}_A(r)$.

If $M$ is the right trivial mean (i.e. $M(A,B)=B$) then $f(X)$ is nonexpansive on $\mathbb{P}$, that is
\begin{equation*}
d_\infty(f(X),f(Y))\leq d_\infty(X,Y)
\end{equation*}
for all $A,X,Y\in\mathbb{P}$.
\end{theorem}
\begin{proof}
The case of the right trivial mean is just the preceding remark, so assume that $M(A,B)$ is not the right trivial mean. Again by property 3 in Lemma~\ref{thompsonproperties} it is enough to prove for the case when $A=I$. By Proposition~\ref{harmonicreprprop} the mean $M\in\mathfrak{M}$ is represented as
\begin{equation}\label{contractint}
M(I,X)=\int_{[0,1]}[(1-s)I+sX^{-1}]^{-1}d\nu(s)=\int_{[0,1]}h_s(X)d\nu(s)
\end{equation}
So let $X,Y\in\overline{B}_I(r)$. There are other simple cases when the probability measure is supported only over the two points $\{0\},\{1\}$. These cases include the weighted arithmetic mean with $s\in(0,1)$ and the case of the left trivial mean which is covered in the preceding remark and are clearly strict contractions on $\overline{B}_I(r)$.

For the remaining cases we split the integral in \eqref{contractint} to the sum of integrals over the mutually disjoint intervals $I_1=[0,a)$, $I_2=[a,1-a]$, $I_3=(1-a,1]$ for some $a\in(0,1/2)$ such that $\nu$ has nonzero mass on the interval $I_2$. Such an $a$ clearly exists since we have just exlcuded the cases when the measure $\nu$ is supported only on the points $\{0\},\{1\}$. Let $J_i=1$ if $\nu$ is supported on $I_i$ and $J_i=0$ if it is not. By assumption $J_2=1$ always. We have that
\begin{equation*}
f_{i}(X)=\int_{I_i}[(1-s)I+sX^{-1}]^{-1}d\nu(s).
\end{equation*}
Moreover due to the weak-$*$ compactness of the convex cone of probability measures on the compact interval $[0,1]$, the integral in \eqref{contractint} can be approximated by finite convex combinations (this is the Krein-Milman theorem) of the form
\begin{equation*}
\sum_{i}[(1-s_i)I+s_iX^{-1}]^{-1}K_{s_i}
\end{equation*}
where $s_i\in[0,1]$, $K_{s_i}>0$. More precisely there exists a net of finitely supported probability measures $\nu_k$ on $[0,1]$, such that the net
\begin{equation}\label{convexcomb}
\int_{[0,1]}[(1-s)I+sX^{-1}]^{-1}d\nu_k(s)
\end{equation}
converges to $f(X)=M(I,X)$. Now let
\begin{equation*}
f_{i,k}(X)=\int_{I_i}[(1-s)I+sX^{-1}]^{-1}d\nu_k(s)
\end{equation*}
for $i=1,2,3$, i.e. $f(X)=\lim_{k}f_{1,k}(X)+f_{2,k}(X)+f_{3,k}(X)$. The functions $f_{1,k}$ and $f_{3,k}$ are nonexpansive due to property 4 in Lemma~\ref{thompsonproperties} and the preceding remarks, i.e.
\begin{equation*}
d_\infty(f_{i,k}(X),f_{i,k}(Y))\leq \rho_i d_\infty(X,Y)
\end{equation*}
where $\rho_i=1$ for $i=1,3$. Again due to property 4 in Lemma~\ref{thompsonproperties}, the preceding remarks and Lemma~\ref{arithmeticcontraction} we have that
\begin{equation*}
\rho_2=\frac{\log\frac{e^{-r-|\log a-\log (1-a)|}+e^{2r}}{e^{-r-|\log a-\log (1-a)|}+1}}{2r}
\end{equation*}
since it is easy to see that by Lemma~\ref{arithmeticcontraction} the contraction coefficient $\rho$ corresponding to a $h_s(X)$ with $s\in[a,1-a]$ is bounded above by the contraction coefficient $\rho$ corresponding to $h_a(X)$ or equivalently $h_{1-a}(X)$. Moreover it is easy to see that these $\rho_i$ are uniform for all $k$, so taking the limit $k\to\infty$ we get that
\begin{equation*}
d_\infty(f_{i}(X),f_{i}(Y))\leq \rho_i d_\infty(X,Y).
\end{equation*}

Now by Proposition~\ref{weightedthompson} we have
\begin{equation*}
\begin{split}
e^{d_\infty(f(X),f(Y))}&\leq\max\left\{\frac{\sum_{i=1}^{3}e^{d_\infty(f_i(X),f_i(Y))}J_ie^{-d_\infty(f_m(X),f_i(X))}}{\sum_{i=1}^{3}J_ie^{-d_\infty(f_m(X),f_i(X))}},\right.\\
&\left.\frac{\sum_{i=1}^{3}e^{d_\infty(f_i(X),f_i(Y))}J_ie^{-d_\infty(f_m(Y),f_i(Y))}}{\sum_{i=1}^{3}J_ie^{-d_\infty(f_m(Y),f_i(Y))}}\right\}\\
&\leq\max\left\{\frac{\sum_{i=1}^{3}e^{\rho_id_\infty(X,Y)}J_ie^{-d_\infty(f_m(X),f_i(X))}}{\sum_{i=1}^{3}J_ie^{-d_\infty(f_m(X),f_i(X))}},\right.\\
&\left.\frac{\sum_{i=1}^{3}e^{\rho_id_\infty(X,Y)}J_ie^{-d_\infty(f_m(Y),f_i(Y))}}{\sum_{i=1}^{3}J_ie^{-d_\infty(f_m(Y),f_i(Y))}}\right\}
\end{split}
\end{equation*}
where $m$ is such that $d_\infty(f_m(X),f_m(Y))\geq d_\infty(f_i(X),f_i(Y))$, $i=1,2,3$. To obtain the second inequality above we used the monotonicity of the functions $e^x$ and the weighted arithmetic mean $\frac{\sum_{i=1}^3w_ix_i}{\sum_{i=1}^3w_i}$, with weights of the form $w_i=e^{-d_\infty(f_m(Y),f_i(Y))}$. The next step is to see that $d_\infty(f_j(Y),f_i(Y))$ is bounded for $i,j=1,2,3$. Indeed if $f_{i}$ is nonzero then it is a continuous (in fact analytic) function on $\overline{B}_I(r)$ and since $\overline{B}_I(r)$ is bounded, the image of it under $f_{i}$ is also bounded, so 
\begin{equation*}
\max_{i,j=1,2,3}\sup_{X\in\overline{B}_I(r)}d_\infty(f_j(X),f_i(X))
\end{equation*}
is bounded as well. Let this bound be $L$. Then since $\rho_1=\rho_3=1$ we have that
\begin{equation*}
\begin{split}
e^{d_\infty(f(X),f(Y))}&\leq\max\left\{\frac{\sum_{i=1}^{3}e^{\rho_id_\infty(X,Y)}J_ie^{-d_\infty(f_m(X),f_i(X))}}{\sum_{i=1}^{3}J_ie^{-d_\infty(f_m(X),f_i(X))}},\right.\\
&\left.\frac{\sum_{i=1}^{3}e^{\rho_id_\infty(X,Y)}J_ie^{-d_\infty(f_m(Y),f_i(Y))}}{\sum_{i=1}^{3}J_ie^{-d_\infty(f_m(Y),f_i(Y))}}\right\}\\
&\leq\frac{e^{d_\infty(X,Y)}(J_1+J_3)+e^{\rho_2d_\infty(X,Y)}e^{-L}}{e^{-L}+J_1+J_3}.
\end{split}
\end{equation*}
Similarly as in the end of the proof of Lemma~\ref{arithmeticcontraction} we seek some $0<\rho<1$ such that
\begin{equation*}
\log\left(\frac{e^{d_\infty(X,Y)}(J_1+J_3)+e^{\rho_2d_\infty(X,Y)}e^{-L}}{e^{-L}+J_1+J_3}\right)\leq \rho d_\infty(X,Y)
\end{equation*}
for all $X,Y\in\overline{B}_I(r)$, i.e. $d_\infty(X,Y)\leq 2r$. By the same argument as in the end of the proof of Lemma~\ref{arithmeticcontraction} we see that
\begin{equation*}
\rho=\frac{\log\left(\frac{e^{2r}(J_1+J_3)+e^{\rho_22r}e^{-L}}{e^{-L}+J_1+J_3}\right)}{2r}
\end{equation*}
suffices and clearly $\rho<1$.

\end{proof}

\begin{remark}
In \cite{lawsonlim0} Lawson and Lim provided an extension of the geometric, logarithmic and some other iterated means to several variables over $\mathbb{P}$ relying on the Ando-Li-Mathias construction provided in \cite{ando}. They established the above contractive property for these means. Our Theorem~\ref{contract} shows that in fact the construction is applicable to all matrix means, hence providing multivariable extensions which work in the possibly infinite dimensional setting of $\mathbb{P}$. This were only known in the finite dimensional setting so far which case was proved in \cite{palfia3}.
\end{remark}

The further importance of Theorem~\ref{contract} will be apparent in the following sections, when we consider matrix (in fact operator) equations similarly to the case of the matrix power means. We close the section with a general nonexpansive property.
\begin{proposition}\label{nonexpansivemean}
Let $M:\mathbb{P}^k\to \mathbb{P}$ be such that
\begin{enumerate}
	\item if $A_i\leq B_i$ for all $1\leq i\leq k$, then $M(A_1,\ldots,A_k)\leq M(B_1,\ldots,B_k)$,
	\item if $t>0$, then $M(tA_1,\ldots,tA_k)=tM(A_1,\ldots,A_k)$,
\end{enumerate}
then
\begin{equation*}
d_\infty(M(A_1,\ldots,A_k),M(B_1,\ldots,B_k))\leq \max_{1\leq i\leq k}d_\infty(A_i,B_i)
\end{equation*}
for all $A_i,B_i\in\mathbb{P}$.
\end{proposition}
\begin{proof}
Let $t=\max_{1\leq i\leq k}d_\infty(A_i,B_i)$. Then $A_i\leq tB_i$ and $B_i\leq tA_i$ for all $1\leq i\leq k$, so by property 1 and 2
\begin{equation*}
\begin{split}
M(A_1,\ldots,A_k)\leq M(tB_1,\ldots,tB_k)=tM(B_1,\ldots,B_k)\\
M(B_1,\ldots,B_k)\leq M(tA_1,\ldots,tA_k)=tM(A_1,\ldots,A_k),
\end{split}
\end{equation*}
i.e. 
\begin{equation*}
\begin{split}
M(A_1,\ldots,A_k)\leq \max_{1\leq i\leq k}d_\infty(A_i,B_i)M(B_1,\ldots,B_k)\\
M(B_1,\ldots,B_k)\leq \max_{1\leq i\leq k}d_\infty(A_i,B_i)M(A_1,\ldots,A_k).
\end{split}
\end{equation*}

\end{proof}

\section{Extension of operator means via contraction principle}
Let $\Delta_{n}$ denote the convex set of positive probability vectors, i.e. if $\omega=(w_{1},\dots,w_{n})\in\Delta_{n}$, then $w_i>0$ and $\sum_{i=1}^nw_i=1$. We will use the following notations:

For ${\Bbb A}=(A_{1},\dots,A_{k})\in {\Bbb P}^{k}$, $M\in
\mathrm{GL}(E),{\bf a}=(a_{1},\dots,a_{k})\in (0,\infty)^{k},
\omega=(w_{1},\dots,w_{k})\in \Delta_{k}$, $f:(0,\infty)\to (0,\infty)$, and for a
permutation $\sigma$ on $k$-letters let
\begin{eqnarray*}
M{\Bbb A}M^{*}&=&(MA_{1}M^{*},\dots,MA_{k}M^{*}),\ \
{\Bbb A}_{\sigma}=(A_{\sigma(1)},\dots,A_{\sigma(k)}),\\
{\Bbb A}^{(n)}&=&(\underbrace{{\Bbb A},\dots,{\Bbb A}}_{n})\in {\Bbb
P}^{nk},\ \
{\omega}^{(n)}=\frac{1}{n}(\underbrace{\omega,\dots,\omega}_{n})\in
\Delta_{nk},\\
f({\bf a})&=&(f(a_{1}),\dots,f(a_{k})),\ \ {\omega}\odot {\bf
a}=\frac{1}{\sum_{i=1}^{k}w_{i}a_{i}}(w_{1}a_{1},\dots,w_{k}a_{k})\in
\Delta_{k},\\
 {\hat\omega}&=&\frac{1}{1-w_{k}}(w_{1},\dots,w_{k-1})\in
\Delta_{k-1},\ \ \ \ {\bf a}\cdot {\Bbb A}=(a_{1} A_{1},\dots,a_{k}
A_{k}) .
\end{eqnarray*}
In \cite{limpalfia} Lim and P\'alfia defined the one parameter family of matrix power means $P_s(\omega;{\Bbb A})$ as the unique positive definite solution of the equations
\begin{equation}\label{powermeanequ2}
X=\sum_{i=1}^{k}w_{i}G_s(X,A_i)
\end{equation}
where $s\in[-1,1], w_i>0, \sum_{i=1}^kw_i=1$ and $A_i\in\mathbb{P}$ and
\begin{equation*}
G_s(A,B)=A^{1/2}\left(A^{-1/2}BA^{-1/2}\right)^{s}A^{1/2}
\end{equation*}
is again the weighted geometric mean. Existence and uniqueness of the solution of \eqref{powermeanequ2} follow from the fact that the function
\begin{equation*}
f(X)=\sum_{i=1}^{k}w_{i}G_s(X,A_i)
\end{equation*}
is a strict contraction for $s\in[-1,1],s\neq 0$ with respect to Thompson's part metric \cite{limpalfia}. In the case $k=2$ we get back the affine matrix means \eqref{means} with $s=\kappa-1$ and $t=w_2$ as we have seen earlier.

Now we will study the generalized form of \eqref{powermeanequ2}.
\begin{lemma}\label{powercontract}
Let $\omega\in\Delta_k$ and $A_i\in\mathbb{P}$, $1\leq i\leq k$ and $M\in\mathfrak{M}$ which is not the right trivial mean. Then the function
\begin{equation}\label{meanequf}
f_M(X)=\sum_{i=1}^{k}w_{i}M(X,A_i)
\end{equation}
is a strict contraction with respect to the Thompson metric $d_\infty(\cdot,\cdot)$ on every bounded $S\subseteq\mathbb{P}$ such that $A_i\in S$ for all $1\leq i\leq k$.
\end{lemma}
\begin{proof}
Let $A\in\mathbb{P}$ and $r<\infty$ given such that $S\subseteq\overline{B}_A(r)$ and for all $X,Y\in S$ the functions $g_i(X)=M(X,A_i)$ are strict contractions for all $1\leq i\leq k$. Clearly by the boundedness of $S$ and the set $\{A_1,\ldots,A_k\}$ and Theorem~\ref{contract} such $A\in\mathbb{P}$ and $r<\infty$ exists. Suppose the largest contraction coefficient for the functions $g_i(X)=M(X,A_i)$ for $1\leq i\leq k$ is $\rho_m$ on $\overline{B}_A(r)$. Then by property 4 in Lemma~\ref{thompsonproperties} we have that
\begin{equation*}
d_\infty(f_M(X),f_M(Y))\leq \rho_m d_\infty(X,Y).
\end{equation*}

\end{proof}

\begin{proposition}
Let $\omega\in\Delta_k$ and $A_i\in\mathbb{P}$, $1\leq i\leq k$ and $M\in\mathfrak{M}$. Then the equation
\begin{equation}\label{meanequ2}
X=\sum_{i=1}^{k}w_{i}M(X,A_i)
\end{equation}
has a unique positive definite solution in $\mathbb{P}$.
\end{proposition}
\begin{proof}
Suppose that $M$ is not the right trivial mean. Then by Lemma~\ref{powercontract} for every bounded subset $S\subseteq\mathbb{P}$ the function $f_M(X)$ given in \eqref{meanequf} is a strict contraction on $S$, so by Banach's fixed point theorem $f_M(X)$ has a unique fixed point on $S$, so the equation \eqref{meanequ2} has a unique positive definite solution on $S$. Since $S$ was arbitrary bounded subset of $\mathbb{P}$, it follows that the same holds on all of $\mathbb{P}$.

If $M$ is the right trivial mean, then \eqref{meanequ2} is equivalent to
\begin{equation*}
X=\sum_{i=1}^{k}w_{i}A_i,
\end{equation*}
i.e. the unique solution is the weighted arithmetic mean.

\end{proof}

\begin{definition}[Induced Operator Mean]
Let $M(\cdot,\cdot)\in\mathfrak{M}$, ${\Bbb A}=(A_{1},\dots,A_{k})\in {\Bbb P}^{k}$ and $\omega\in \Delta_{k}$.
We denote by $M(\omega;{\Bbb A})$ the unique solution of the equation
\begin{equation}\label{meanequ3}
X=\sum_{i=1}^{k}w_{i}M(X,A_i).
\end{equation}
We call $M(\omega;{\Bbb A})$ the $\omega$-weighted induced operator mean of $M\in\mathfrak{M}$ of $A_{1},\dots,A_{n}$.
\end{definition}

\begin{remark}\label{monotone}
Let $f_M(X)$ be defined by \eqref{meanequf}. Then by the monotonicity of $M$, $f_M$ is monotone: $X\leq Y$ implies that $f_M(X)\leq f_M(Y)$.
\end{remark}

\begin{remark}
The one parameter family of matrix power means $P_s(\omega;{\Bbb A})$ is the unique positive definite solution of the equations
\begin{equation*}
X=\sum_{i=1}^{k}w_{i}G_s(X,A_i)
\end{equation*}
where $s\in[-1,1]$. These means are induced means for $s\in(0,1]$ and the inducing mean is the weighted geometric mean $G_s(A,B)$.
\end{remark}

\begin{proposition}\label{MPro}
Let ${\Bbb A}=(A_{1},\dots,A_{k}), {\Bbb
B}=(B_{1},\dots,B_{k})\in {\Bbb P}^{k},\omega\in \Delta_{k}$ and $M,N\in\mathfrak{M}$ and $M(\omega;{\Bbb A})$, $N(\omega;{\Bbb A})$ the corresponding induced operator means. Then
\begin{itemize}
 \item[(1)] $M(\omega;{\Bbb A})=A$ if $A_i=A$ for all $1\leq i\leq k$;
 \item[(2)] $M(\omega_{\sigma};{\Bbb A}_{\sigma})=M(\omega;{\Bbb A})$ for any
 permutation $\sigma;$
 \item[(3)] $M(\omega;{\Bbb A})\leq M(\omega;{\Bbb B})$ if $A_{i}\leq
 B_{i}$ for all $i=1,2,\dots,k;$
 \item[(4)] if $M(A,B)\leq N(A,B)$ for all $A,B\in{\Bbb P}$ then $M(\omega;{\Bbb A})\leq N(\omega;{\Bbb A})$;
 \item[(5)] $M(\omega;X{\Bbb A}X^{*})=XM(\omega;{\Bbb
 A})X^{*}$ for any $X\in \mathrm{GL}(E);$
 \item[(6)] $(1-u)M(\omega;{\Bbb A})+uM(\omega;{\Bbb B})\leq
 M(\omega;(1-u){\Bbb A}+u{\Bbb B})$ for any $u\in [0,1];$
 \item[(7)]  $d_{\infty}(M(\omega;{\Bbb A}), M(\omega;{\Bbb
 B}))\leq \max_{1\leq i\leq k}\{d_{\infty}(A_{i},B_{i})\};$
 \item[(8)] $M(\omega^{(n)};{\Bbb A}^{(n)})=M(\omega;{\Bbb A})$
 for any $n\in {\Bbb N};$
 \item[(9)]
$M(\omega;A_{1},\dots,A_{k-1},X)=X$ if and only if
$X=M({\hat \omega};A_{1},\dots,A_{k-1}).$ In particular,
$M(A_{1},\dots,A_{k},X)=X$ if and only if
$X=M(A_{1},\dots,A_{k});$
\item[(10)]$\Phi(M(\omega;{\Bbb A}))\leq
M(\omega;\Phi({\Bbb A}))$ for any  positive unital linear map
$\Phi,$ where $\Phi({\Bbb A})=(\Phi(A_{1}),\dots,\Phi(A_{k})).$
\end{itemize}
\end{proposition}
\begin{proof}
(1) By \eqref{meanequ3} we have $X=\sum_{i=1}^kw_iM(X,A)$ and using that $M(A,A)=A$ we see that $X=A$ is a, and by uniqueness, the solution to \eqref{meanequ3}.

\vspace{2mm}

(2) Follows from the defining equation \eqref{meanequ3}.

\vspace{2mm}

(3) Suppose that $A_{i}\leq B_{i}$ for all $i=1,2,\dots,k$. Define $ f_M(X)=\sum_{i=1}^{k}w_{i}M(X,A_{i})$ and $
g_M(X)=\sum_{i=1}^{k}w_{i}M(X,B_{i}). $ Then $M(\omega;{\Bbb
A})=\lim_{l\to \infty}f_M^{l}(X)$ and $M(\omega;{\Bbb
B})=\lim_{l\to \infty}g_M^{l}(X)$ for any $X\in {\Bbb P}$, by the
Banach fixed point theorem. By the monotonicity of $M\in\mathfrak{M}$,
$f_M(X)\leq g_M(X)$ for all $X\in {\Bbb P},$ and $f_M(X)\leq f_M(Y),
g_M(X)\leq g_M(Y)$ whenever $X\leq Y.$ Let $X_{0}>0.$ Then $f_M(X_{0})\leq
g_M(X_{0})$ and $f_M^{2}(X_{0})=f_M(f_M(X_{0}))\leq g_M(f_M(X_{0}))\leq
g_M^{2}(X_{0}).$ Inductively, we have $f_M^{l}(X_{0})\leq g_M^{l}(X_{0})$
for all $l\in {\Bbb N}.$ Therefore, $M(\omega;{\Bbb
A})=\lim_{l\to\infty}f_M^{l}(X_{0})\leq \lim_{l\to\infty}
g_M^{l}(X_{0})=M(\omega;{\Bbb B}).$

\vspace{2mm}

(4) Define $ f(X)=\sum_{i=1}^{k}w_{i}M(X,A_{i})$ and $g(X)=\sum_{i=1}^{k}w_{i}N(X,A_{i}). $ Then $M(\omega;{\Bbb
A})=\lim_{l\to \infty}f^{l}(X)$ and $N(\omega;{\Bbb A})=\lim_{l\to \infty}g^{l}(X)$ for any $X\in {\Bbb P}$, by Banach's fixed point theorem. Since $M\leq N$, $f(X)\leq g(X)$ for all $X\in {\Bbb P},$ and $f(X)\leq f(Y), g(X)\leq g(Y)$ whenever $X\leq Y.$ Let $X_{0}>0.$ Then $f(X_{0})\leq
g(X_{0})$ and $f^{\circ 2}(X_{0})=f(f(X_{0}))\leq g(f(X_{0}))\leq
g^{\circ 2}(X_{0}).$ Inductively, we have $f^{l}(X_{0})\leq g^{l}(X_{0})$
for all $l\in {\Bbb N}.$ Therefore, $M(\omega;{\Bbb
A})=\lim_{l\to\infty}f^{l}(X_{0})\leq \lim_{l\to\infty}
g^{l}(X_{0})=N(\omega;{\Bbb A}).$

 \vspace{2mm}

 (5) It follows from the defining equation of $M(\omega;{\Bbb
A})$ and the uniqueness of the positive definite solution.


\vspace{2mm}

 (6) Let $X=M(\omega;{\Bbb A})$ and $Y=M(\omega;{\Bbb B}).$ For $u\in [0,1],$ we set $Z_{u}=(1-u)X+uY.$ Let $f_M(Z)=\sum_{i=1}^{n}w_{i}M(Z,((1-u)A_{i}+uB_{i})).$ Then by the joint concavity of two-variable operator means (Theorem 3.5 \cite{kubo})
\begin{eqnarray*}
Z_{u}&=&(1-u)X+uY=\sum_{i=1}^{n}w_{i}[(1-u)M(X,A_{i})+uM(Y,B_{i})]\\
&\leq&\sum_{i=1}^{n}w_{i}M(((1-u)X+uY),((1-u)A_{i}+uB_{i}))=f(Z_{u}).
\end{eqnarray*}
Inductively, $Z_{u}\leq f_M^{l}(Z_{u})$ for all $l\in {\Bbb N}.$
Therefore, $ (1-u)M(\omega;{\Bbb A})+uM(\omega;{\Bbb
B})=Z_{u}\leq M(\omega; (1-u){\Bbb A}+u{\Bbb B}).$

\vspace{2mm}

 (7) Follows from Proposition~\ref{nonexpansivemean} using property (4) and (5).

\vspace{2mm}

 (8) Let $X=M(\omega; {\Bbb
 A}).$ Then
 $$X=\sum_{i=1}^{k}w_{i}M(X,A_{i})=\frac{1}{n}(\underbrace{\sum_{i=1}^{k}w_{i}M(X,A_{i})+\cdots+\sum_{i=1}^{k}w_{i}M(X,A_{i})}_{k})$$
 and therefore $X=M(\omega^{(n)};{\Bbb A}^{(n)}).$

 \vspace{2mm}

 (9) We have $ M(\omega;A_{1},\dots,A_{k-1},X)=X$
if and  only if $X=\sum_{i=1}^{k-1}w_{i}M(X,A_{i})+w_{k}X$ if
and only if $X=\frac{1}{1-w_{k}}\sum_{i=1}^{k-1}w_{i}M(X,A_{i})$
if and only if $X=M({\hat \omega};A_{1},\dots,A_{k-1}).$

\vspace{2mm}

 (10) Note that $\Phi(M(A,B))\leq
M(\Phi(A),\Phi(B))$ for any $A,B>0$ by Proposition~\ref{positivemapthm}. Then
\begin{eqnarray}\label{E:L1}
\Phi(M(\omega;{\Bbb A}))=\sum_{i=1}^{k}w_{i}\Phi(M(M(\omega;{\Bbb A}),A_{i}))\leq
\sum_{i=1}^{k}w_{i}M(\Phi(M(\omega;{\Bbb A})),\Phi(A_{i})).
\end{eqnarray}
Define $f_M(X)=\sum_{i=1}^{k}w_{i}M(X,\Phi(A_{i})).$ Then
$\lim_{l\to\infty}f_M^{l}(X)=M(\omega;\Phi({\Bbb A}))$ for any
$X>0.$ By \eqref{E:L1}, $f_M(\Phi(M(\omega;{\Bbb A})))\geq \Phi(M(\omega;{\Bbb A})).$ Since $f$
is monotonic, $f_M^{l}(\Phi(X_{t}))\geq \Phi(M(\omega;{\Bbb A}))$ for all $l\in
{\Bbb N}.$ Thus
\begin{eqnarray*}
M(\omega;\Phi({\Bbb A}))=\lim_{l\to\infty}f_M^{l}(\Phi(M(\omega;{\Bbb A})))\geq
\Phi(M(\omega;{\Bbb A}))=\Phi(M(\omega;{\Bbb A})).
\end{eqnarray*}

\end{proof}

\begin{corollary}\label{cor:opmean}
If $k=2$, $M(w_1,w_2;A,B)\in\mathfrak{M}$ is an operator mean (induced by another mean $M(A,B)\in\mathfrak{M}$).
\end{corollary}
\begin{proof}
By property (5) in Proposition~\ref{MPro} it follows that $$M(w_1,w_2;A,B)=A^{1/2}M(w_1,w_2;I,A^{-1/2}BA^{-1/2})A^{1/2}$$ and property (1) yields that $M(w_1,w_2;I,I)=I$. By Lemma~\ref{powercontract} we have that
\begin{equation*}
\lim_{l\to \infty}f^{\circ l}(X)=M(w_1,w_2;I,A^{-1/2}BA^{-1/2})
\end{equation*}
for all $X\in\mathbb{P}$ with $f(X)=w_1M(X,I)+w_2M(X,A^{-1/2}BA^{-1/2})$. So we can choose $X=I$ and then
\begin{equation*}
\lim_{l\to \infty}f^{\circ l}(I)=M(w_1,w_2;I,C)
\end{equation*}
where $C=A^{-1/2}BA^{-1/2}$. Also by simple calculation we have that $$f(X)=w_1Xg(X^{-1})+w_2Xg(X^{-1}C)$$ where $g$ is the representing function of $M$. Let $h(C):=M(w_1,w_2;I,C)$. Then $h(C)=\lim_{l\to \infty}f^{\circ l}(I)$ and by property (3) in Proposition~\ref{MPro}, $h$ is operator monotone.  Moreover $f^{\circ l}(I)$ is an analytic real map in the single variable $C$ for all $l$, moreover the net $f^{\circ l}(I)$ converges uniformly on bounded subsets of $\mathbb{P}$ due to the strict contraction property of $f(X)$. Hence the pointwise limit $\lim_{l\to\infty}f^{\circ l}(1)$ for positive real (scalar) $C$ is a continuous real map as well and is identical to $h$ by the properties of the functional calculus of self-adjoint operators, since the net $f^{\circ l}(I)$ converges in norm for all $C$ (the topology generated by the metric $d_\infty$ agrees with the relative Banach space topology \cite{thompson}). It is also easy to see that $h$ is positive on $(0,\infty)$ and $h(1)=1$, hence $h$ is an operator monotone function in $\mathfrak{m}$. So by Theorem 3.2 in \cite{kubo} we get that $M(w_1,w_2;A,B)$ is an operator mean in the sense of Definition~\ref{symmean}.

\end{proof}

\begin{proposition}\label{P:2variable}
Let $\omega\in\Delta_2$, $A,B\in\mathbb{P}$ and $M\in\mathfrak{M}$ with representing function $f$. Then
\begin{eqnarray*}
M(w_1,w_2;A,B)=A^{1/2}g\left(A^{-1/2}BA^{-1/2}\right)A^{1/2}
\end{eqnarray*}
where
\begin{eqnarray}\label{ginverse}
g^{-1}(x)=xf^{-1}\left(\frac{1-w_1f(x^{-1})}{w_2}\right).
\end{eqnarray}
\end{proposition}
\begin{proof}
By \eqref{meanequ3} we have
\begin{eqnarray*}
X=w_1X^{1/2}f(X^{-1/2}AX^{-1/2})X^{1/2}+w_2X^{1/2}f(X^{-1/2}BX^{-1/2})X^{1/2}
\end{eqnarray*}
which, with $f'(x)=xf(1/x)$ denoting the transpose of $f$, is equivalent to
\begin{eqnarray*}
&&A^{-1/2}XA^{-1/2}=w_1f'(A^{-1/2}XA^{-1/2})\\
&&+w_2A^{-1/2}X^{1/2}f(X^{-1/2}A^{1/2}A^{-1/2}BA^{-1/2}A^{1/2}X^{-1/2})X^{1/2}A^{-1/2}\\
&&=w_1f'(A^{-1/2}XA^{-1/2})\\
&&+w_2A^{-1/2}XA^{-1/2}A^{1/2}X^{-1/2}f(X^{-1/2}A^{1/2}A^{-1/2}BA^{-1/2}A^{1/2}X^{-1/2})X^{1/2}A^{-1/2}\\
&&=w_1f'(A^{-1/2}XA^{-1/2})+w_2A^{-1/2}XA^{-1/2}f(A^{1/2}X^{-1}A^{1/2}A^{-1/2}BA^{-1/2})\\
&&=w_1f'(U)+w_2Uf(U^{-1}W)=w_1Uf(U^{-1})+w_2Uf(U^{-1}W)=U,
\end{eqnarray*}
where $U=A^{-1/2}XA^{-1/2}$ and $W=A^{-1/2}BA^{-1/2}$. From this we get that $I=w_1f(U^{-1})+w_2f(U^{-1}W)$, i.e. $W=Uf^{-1}\left(\frac{I-w_1f(U^{-1})}{w_2}\right)=g^{-1}(U)$. This means $X=A^{1/2}g\left(A^{-1/2}BA^{-1/2}\right)A^{1/2}$. Note that \eqref{ginverse} may only be well defined in a small neighborhood, but we know by Corollary~\ref{cor:opmean}, that $g$ can be analytically continued to the whole $(0,\infty)$, since $M(w_1,w_2;A,B)$ is an operator mean with representing operator monotone function $g$.

\end{proof}

\begin{remark}
Since the induced two-variable operator mean $M(w_1,w_2;A,B)$ is uniquely determined by the inducing $M\in\mathfrak{M}$, therefore the induced means $M(\omega;\mathbb{A})$ can be regarded as an extension of the operator mean $M(w_1,w_2;A,B)$ to more then two variables. Formula \eqref{ginverse} uniquely determines the representing function of $M(w_1,w_2;A,B)$ in terms of the representing function $f$ of $M\in\mathfrak{M}$, the mapping $ind:\mathfrak{M}\mapsto\mathfrak{M}$ given as $M(\cdot,\cdot)\mapsto M(w_1,w_2;\cdot,\cdot)$ gives a self mapping of the set $\mathfrak{M}$. It is easy to see that the weighted arithmetic and harmonic means are inducing themselves, so they are fixed points of the map $ind$.
\end{remark}

\section{Generalized Karcher equations and one parameter families of operator means}
In this section we generalize the results of \cite{limpalfia,lawsonlim1} which were given for the one parameter family of power means. We will provide solutions of nonlinear operator equations that are given in the following
\begin{definition}[Generalized Karcher equation]\label{generalizedkarcherequ}
Let $\log_I\in\mathfrak{L}$ and $\log_X(A)=X^{1/2}\log_I(X^{-1/2}AX^{-1/2})X^{1/2}$.
The generalized Karcher equation induced by $\log_I$ is the operator equation
\begin{equation*}
\sum_{i=1}^kw_i\log_X(A_i)=0
\end{equation*}
where $X,A_i\in\mathbb{P}$.
\end{definition}

By Proposition~\ref{inPt} we have that if $M\in\mathfrak{M}$ is not the left or right trivial mean, then $f\in\mathfrak{P}(t)$, $t=f'(1)$. The results in section 4 and 5, in particular Theorem~\ref{uniformconv} ensures us, that all $f\in\mathfrak{P}(t_0)$ can be uniquely written as
\begin{equation}\label{oneparamf}
f(x)=\exp_I(t_0\log_I(x))
\end{equation}
where $\log_I\in\mathfrak{L}$ is the unique logarithm map corresponding to $f$ and $\exp_I$ is the inverse of $\log_I$. We also have by Theorem~\ref{nobranchthm} that if $\log_I$ has no ramification points in the upper half-plane $\mathbb{H}^{+}$, then the one parameter family $f_t(x)=\exp_I(t\log_I(x))$ is in $\mathfrak{P}(t)$ for all $t\in(0,1)$. In the general situation of ramification points if a given $\log_I\in\mathfrak{L}$ is induced by an $f\in\mathfrak{P}(t_0)$, then by Proposition~\ref{ftinPt} $f_t\in\mathfrak{P}(t)$ for all $0<t\leq t_0$. This makes it possible to consider one parameter families of induced operator means, similarly to the case of the matrix power means $P_s(\omega;{\Bbb A})$ where $s\in[-1,1]$ \eqref{powermeanequ2}. 

Throughout this section we suppose that $M\in\mathfrak{M}$ with representing function $f(x)$ given as \eqref{oneparamf} and $f(x)$ is not the left or right trivial mean. This means that $f\in\mathfrak{P}(t_0)$ and also then
\begin{equation*}
f_t(x)=\exp_I(t\log_I(x))
\end{equation*}
is well defined for $0<t\leq t_0$, i.e. $f_t\in\mathfrak{P}(t)$ and $M_t(A,B)$ denotes its corresponding mean in $\mathfrak{M}$. Also we assume that $A_i\in\mathbb{P}$ for all $1\leq i\leq k$ and that $\omega\in\Delta_k$.

\begin{proposition}\label{oneparameterfam}
The one parameter family of induced operator means $M_t(\omega;{\Bbb A})$ induced by the $M_t(A,B)\in\mathfrak{M}$ with representing function $f_t(x)$ is continuous for $t\in(0,t_0]$ on any bounded set $S\subseteq\mathbb{P}$.
\end{proposition}
\begin{proof}
The induced operator means $M_t(\omega;{\Bbb A})$ are fixed points of mappings $f_{M_t}(x)$ given in \eqref{meanequf} which are strict contractions on any bounded subset of $S\subseteq\mathbb{P}$ according to Lemma~\ref{powercontract}. Therefore on every bounded set $S\subseteq\mathbb{P}$ $M_t(\omega;{\Bbb A})$ varies continuously with respect to $t$ due to the continuity of fixed points of pointwisely continuous families of strict contractions \cite{neeb}.

\end{proof}

\begin{lemma}\label{maxmininduced}
For $t\in(0,t_0]$ we have
\begin{equation*}
\left(\sum_{i=1}^kw_iA_i^{-1}\right)^{-1}\leq M_t(\omega;{\Bbb A})\leq \sum_{i=1}^kw_iA_i.
\end{equation*}
\end{lemma}
\begin{proof}
By Lemma~\ref{maxmininp} we have that
\begin{equation*}
[(1-t)A^{-1}+tB^{-1}]^{-1}\leq M_t(A,B)\leq (1-t)A+tB.
\end{equation*}
By substituting into the defining equations \eqref{meanequ3} with the two-variable weighted harmonic and arithmetic mean we get back the corresponding multivariable versions $\left(\sum_{i=1}^kw_iA_i^{-1}\right)^{-1}$ and $\sum_{i=1}^kw_iA_i$ respectively. Then by property (4) in Proposition~\ref{MPro} we get the assertion.

\end{proof}

Let us recall the strong topology on $\mathbb{P}$. The strong topology is the topology of pointwise convergence which means that $A_n\to A$ if for all $x\in E$ we have $\left\langle x,A_nx\right\rangle\to \left\langle x,Ax\right\rangle$. The positive definite partial order $\leq$ is strongly closed, so if $A_n\to A$, $B_n\to B$ and $A_n\leq B_n$ then $A\leq B$. Also if $A_n$ is a monotonically decreasing net in $\mathbb{P}$ with respect to $\leq$ and it is bounded from below, then it converges strongly to the infimum of $A_n$. Similarly if $B_n$ is monotonically increases and bounded from above, then $B_n$ converges strongly to its supremum \cite{weidman}.

\begin{theorem}\label{inducedconv}
There exists $X_0\in\mathbb{P}$ such that
\begin{equation*}
\lim_{t\to 0+}M_t(\omega;{\Bbb A})=X_0.
\end{equation*}
Furthermore for $0<t\leq s\leq t_0$ we have
\begin{equation*}
X_0\leq M_t(\omega;{\Bbb A})\leq M_s(\omega;{\Bbb A})\leq M_{t_0}(\omega;{\Bbb A}).
\end{equation*}
\end{theorem}
\begin{proof}
First of all since $f_t(x)\in\mathfrak{P}(t)$, by Proposition~\ref{p:propbound1} we have $f_u(x)\leq (1-u)+ux$ for all $u\in[0,1]$ and $x>0$ real numbers and also $f_u(x)>0$ by the monotonicity property in $u$. Then these with simple considerations yield that
\begin{equation*}
M_u(A,B):=A^{1/2}f_u(A^{-1/2}BA^{-1/2})A^{1/2}\leq (1-u)A+uB.
\end{equation*}
Moreover $M_t(A,B)=M_{\frac{t}{s}}(A,M_s(A,B))$ by the properties of $\exp_I$ and $\log_I$. Let $f_M(X)=\sum_{i=1}^kw_iM_t(X,A_i)$. By Banach's fixed point theorem we have $M_t(\omega;{\Bbb A})=\lim_{l\to\infty}f_M^l(X)$ for all $X\in\mathbb{P}$. From the above we have
\begin{equation*}
\begin{split}
f_M(X)&=\sum_{i=1}^kw_iM_t(X,A_i)=\sum_{i=1}^kw_iM_{\frac{t}{s}}(X,M_s(X,A_i))\\
&\leq \sum_{i=1}^kw_i\left(1-\frac{t}{s}\right)X+\frac{t}{s}M_s(X,A_i)=\left(1-\frac{t}{s}\right)X+\frac{t}{s}\sum_{i=1}^kw_iM_s(X,A_i).
\end{split}
\end{equation*}
Let $X_s=M_s(\omega;{\Bbb A})$. Then by the above we have
\begin{equation*}
f_M(X_s)\leq \left(1-\frac{t}{s}\right)X_s+\frac{t}{s}\sum_{i=1}^kw_iM_s(X_s,A_i)=\left(1-\frac{t}{s}\right)X_s+\frac{t}{s}X_s=X_s.
\end{equation*}
By Remark~\ref{monotone} $f_M(X)$ is monotone, so $f_M^{l+1}(X_s)\leq f_M^{l}(X_s)\leq\cdots\leq f_M(X_s)\leq X_s$ for all $l\geq 0$, so
\begin{equation*}
M_t(\omega;{\Bbb A})=\lim_{l\to\infty}f_M^l(X_s)\leq X_s=M_s(\omega;{\Bbb A}).
\end{equation*}
By Lemma~\ref{maxmininduced} the monotonically decreasing net $M_t(\omega;{\Bbb A})$ is bounded from below (and above) so has a strong limit as $t\to 0+$.

\end{proof}

\begin{definition}
Let $\Lambda_M(\omega;\mathbb{A})=\lim_{t\to 0+}M_t(\omega;{\Bbb A})$ and call it the $\omega$-weighted lambda extension of $M\in\mathfrak{M}$.
\end{definition}

\begin{remark}
If we take the one parameter family of matrix power means $P_t(\omega;{\Bbb A})$, then it is known that $\lim_{t\to 0}P_t(\omega;{\Bbb A})$ is the Karcher mean $\Lambda(\omega;\mathbb{A})$ given as \eqref{karchermean}.
\end{remark}

\begin{theorem}\label{lambdaprop}
Let ${\Bbb A}=(A_{1},\dots,A_{k}), {\Bbb
B}=(B_{1},\dots,B_{k})\in {\Bbb P}^{k},\omega\in \Delta_{k}$ and $M,N\in\mathfrak{M}$ and $\Lambda_M(\omega;{\Bbb A})$, $\Lambda_N(\omega;{\Bbb A})$ the corresponding lambda extensions. Then
\begin{itemize}
 \item[(1)] $\Lambda_M(\omega;{\Bbb A})=A$ if $A_i=A$ for all $1\leq i\leq k$;
 \item[(2)] $\Lambda_M(\omega_{\sigma};{\Bbb A}_{\sigma})=\Lambda_M(\omega;{\Bbb A})$ for any
 permutation $\sigma;$
 \item[(3)] $\Lambda_M(\omega;{\Bbb A})\leq \Lambda_M(\omega;{\Bbb B})$ if $A_{i}\leq
 B_{i}$ for all $i=1,2,\dots,k;$
 \item[(4)] if $M(A,B)\leq N(A,B)$ for all $A,B\in{\Bbb P}$ then $\Lambda_M(\omega;{\Bbb A})\leq \Lambda_N(\omega;{\Bbb A})$;
 \item[(5)] $\Lambda_M(\omega;X{\Bbb A}X^{*})=X\Lambda_M(\omega;{\Bbb
 A})X^{*}$ for any $X\in \mathrm{GL}(E);$
 \item[(6)] $(1-u)\Lambda_M(\omega;{\Bbb A})+u\Lambda_M(\omega;{\Bbb B})\leq
 \Lambda_M(\omega;(1-u){\Bbb A}+u{\Bbb B})$ for any $u\in [0,1];$
 \item[(7)]  $d_{\infty}(\Lambda_M(\omega;{\Bbb A}), \Lambda_M(\omega;{\Bbb
 B}))\leq \max_{1\leq i\leq k}\{d_{\infty}(A_{i},B_{i})\};$
 \item[(8)] $\Lambda_M(\omega^{(n)};{\Bbb A}^{(n)})=\Lambda_M(\omega;{\Bbb A})$
 for any $n\in {\Bbb N};$
\item[(9)]
$\Phi(\Lambda_M(\omega;{\Bbb A}))\leq
\Lambda_M(\omega;\Phi({\Bbb A}))$ for any positive unital linear map
$\Phi,$ where $\Phi({\Bbb A})=(\Phi(A_{1}),\dots,\Phi(A_{k}));$
\item[(10)]
$\left(\sum_{i=1}^kA_i^{-1}\right)^{-1}\leq \Lambda_M(\omega;{\Bbb A})\leq \sum_{i=1}^kA_i.$
\end{itemize}
\end{theorem}
\begin{proof}
Each of the properties easily follows from Proposition~\ref{MPro} and Lemma~\ref{maxmininduced} by taking the limit $t\to 0+$.

\end{proof}

Now we turn to the study of the generalized Karcher equation
\begin{equation}\label{karcherequx}
\sum_{i=1}^kw_i\log_X(A_i)=0,
\end{equation}
where $\log_X(A)=X^{1/2}\log_I\left(X^{-1/2}AX^{-1/2}\right)X^{1/2}$ corresponding to $\log_I\in\mathfrak{L}$. We denote by $Ks(\omega,\mathbb{A})$ the set of all solutions $X$ of \eqref{karcherequx} in $\mathbb{P}$.

\begin{lemma}
Operator multiplication is strongly continuous on any bounded set.
\end{lemma}
\begin{proof}
Let $A_l\to A,B_l\to B$ strongly, and $\left\|A_l\right\|,\left\|B_l\right\|\leq K$. Then
\begin{equation*}
\left\|(A_lB_l-AB)x\right\|\leq \left\|A_l(B_l-B)x\right\|+\left\|(A_l-A)Bx\right\|\leq K\left\|(B_l-B)x\right\|+\left\|(A_l-A)Bx\right\|,
\end{equation*}
so $\left\|(A_lB_l-AB)x\right\|\to 0$ as well.

\end{proof}

\begin{lemma}
Let $Q$ be an open or closed subset of $\mathbb{R}$ and let $f:Q\to \mathbb{R}$ be continuous and bounded. Then $f$ is strong operator continuous on the set $S(E)$ of self adjoint operators with spectrum in $Q$.
\end{lemma}
\begin{proof}
Special case of Theorem 3.6 in \cite{kadison}.

\end{proof}

\begin{lemma}
Let $Q$ be an open or closed subset of $\mathbb{R}$ and let $f:Q\to \mathbb{R}$ be continuous and bounded. Then $f$ is strong operator continuous on the set $S(E)$ of self adjoint operators with spectrum in $Q$.
\end{lemma}
The consequece of the above is the following
\begin{lemma}
The functions
\begin{enumerate}
	\item $x^{-1}$,
	\item $\log_I(x)$ which is monotone,
	\item $f_t(x)=\exp_I(t\log_I(x))$ for $0\leq t\leq t_0$,
	\item the mean $M_t(A,B)$ for $0\leq t\leq t_0$,
\end{enumerate}
are strongly continuous on the order intervals $[e^{-m}I,e^{m}I]$ for any $m>0$.
\end{lemma}

\begin{lemma}\label{lemstrongconv}
Let $V\in S(E)$. Then
\begin{equation}
\lim_{(s,U)\to(0,V)}\frac{\exp_I(sU)-I}{s}=V,
\end{equation}
in the strong operator topology.
\end{lemma}
\begin{proof}
Since $\exp_I$ is the holomorphic inverse of $\log_I$ and $\log_I(1)=0,\log_I'(1)=1$ it follows that $\exp_I(0)=1$ and $\exp_I'(0)=1$. Thus there is a $0$ centered open disk $D$, on which $\exp_I$ is holomorphic, hence has a uniformly and absolutely convergent power series expansion on $D$ given as
\begin{equation*}
\exp_I(z)=\sum_{j=0}^{\infty}\frac{\exp_I^{(j)}}{j!}(0)z^{j}
\end{equation*}
This means that $\sum_{j=0}^{\infty}\left|\frac{\exp_I^{(j)}}{j!}(0)\right|\left|z\right|^{j}<\infty$ for all $z\in D$. Then for $U\in S(E)$ and $s\geq 0$ such that $s\left\|U\right\|\in D$ we have
\begin{equation*}
\begin{split}
\left\|\frac{\exp_I(sU)-I}{s}-U\right\|&=\left\|\frac{\sum_{j=0}^{\infty}\frac{\exp_I^{(j)}}{j!}(0)(sU)^j-I}{s}-U\right\|\\
&=\left\|\sum_{j=1}^{\infty}\frac{\exp_I^{(j)}}{j!}(0)s^{j-1}U^j-U\right\|\\
&\leq\sum_{j=2}^{\infty}\left|\frac{\exp_I^{(j)}}{j!}(0)\right|\left|s\right|^{j-1}\left\|U\right\|^j\\
&=\sum_{j=2}^{\infty}\left|\frac{\exp_I^{(j)}}{(j-1)!}(0)\right|(s\left\|U\right\|)^{j-1}\left\|U\right\|<\infty
\end{split}
\end{equation*}
since the derivative $\exp_I'(z)=\sum_{j=1}^{\infty}\frac{\exp_I^{(j)}}{(j-1)!}(0)z^{j-1}$ absolutely and uniformly converges on the same disk $D$, moreover as $s\to 0+$ the above also goes to $0$. Now let $x\in E$ and let $s\to 0$ and $U\to V$ strongly. Then by the above we have
\begin{equation*}
\begin{split}
\left\|\frac{\exp_I(sU)-I}{s}x-Vx\right\|&\leq\left\|\frac{\exp_I(sU)-I}{s}x-Ux\right\|+\left\|Ux-Vx\right\|\\
&\leq\sum_{j=2}^{\infty}\left|\frac{\exp_I^{(j)}}{j!}(0)\right|\left|s\right|^{j-1}\left\|U\right\|^j\left\|x\right\|+\left\|Ux-Vx\right\|,
\end{split}
\end{equation*}
as $s\to 0$ the first term in the last inequality goes to $0$ so as the second term, since also $U\to V$ strongly.

\end{proof}

\begin{theorem}\label{karchersatisfied}
The lambda extension $\Lambda_M(\omega;{\Bbb A})$ satisfies the generalized Karcher equation
\begin{equation*}
\sum_{i=1}^kw_i\log_X(A_i)=0,
\end{equation*}
where $\log_X(A)=X^{1/2}\log_I\left(X^{-1/2}AX^{-1/2}\right)X^{1/2}.$
\end{theorem}
\begin{proof}
For $0<t\leq t_0$ let $X_t=M_t(\omega;{\Bbb A})$ and $X_0=\Lambda_M(\omega;{\Bbb A})=\lim_{t\to 0+}M_t(\omega;{\Bbb A})$. By Theorem~\ref{inducedconv} $X_t\to X_0$ strongly monotonically as $t\to 0+$ and $X_0\leq X_t\leq X_{t_0}\leq \sum_{i=1}^kw_iA_i$.
Now choose $m$ such that $A_i,X_{t_0},X_0\in[e^{-m}I,e^{m}I]$ for all $i$. Then also $X_t\in[e^{-m}I,e^{m}I]$ for $0\leq t\leq t_0$. The order interval $[e^{-m}I,e^{m}I]$ is closed under inversion, also $1\leq x^{1/2}\leq x$ for $x\in[1,\infty)$ and $1\geq x^{1/2}\geq x$ for $x\in(0,1)$, so $X_t^{-1/2}A_iX_t^{-1/2}\in[e^{-m}I,e^{m}I]$. By the previous lemmas therefore $X_t^{-1/2}A_iX_t^{-1/2}\to X_0^{-1/2}A_iX_0^{-1/2}$ strongly. By the strong continuity of $\log_I$
\begin{equation*}
U_i:=\log_I(X_t^{-1/2}A_iX_t^{-1/2})\to V_i:=\log_I(X_0^{-1/2}A_iX_0^{-1/2}).
\end{equation*}
By Lemma~\ref{lemstrongconv} in the strong topology we have
\begin{equation}\label{eq1}
\lim_{t\to 0+}\frac{\exp_I(tU_i)-I}{t}=V_i=\log_I(X_0^{-1/2}A_iX_0^{-1/2})
\end{equation}
for all $1\leq i\leq k$.

By definition $X_t=\sum_{i=1}^kw_iM_t(X_t,A_i)$ which is equivalent to
\begin{equation*}
I=\sum_{i=1}^kw_if_t(X_t^{-1/2}A_iX_t^{-1/2})=\sum_{i=1}^kw_i\exp_I\left(t\log_I(X_t^{-1/2}A_iX_t^{-1/2})\right),
\end{equation*}
that is $0=\sum_{i=1}^kw_i\frac{f_t(X_t^{-1/2}A_iX_t^{-1/2})-I}{t}$. By \eqref{eq1} we have
\begin{equation}
\begin{split}
0&=\lim_{t\to 0+}\sum_{i=1}^kw_i\frac{f_t(X_t^{-1/2}A_iX_t^{-1/2})-I}{t}=\sum_{i=1}^kw_i\lim_{t\to 0+}\frac{f_t(X_t^{-1/2}A_iX_t^{-1/2})-I}{t}\\
&=\sum_{i=1}^kw_i\log_I(X_0^{-1/2}A_iX_0^{-1/2}).
\end{split}
\end{equation}
By this we have also that $\sum_{i=1}^kw_i\log_X(A_i)=0$.

\end{proof}

\begin{lemma}\label{karchercongruenceinv}
The set $Ks(\omega,\mathbb{A})$ is invariant under congruencies, i.e. for any $C\in \mathrm{GL}(E)$
\begin{equation*}
CKs(\omega,\mathbb{A})C^{*}=Ks(\omega,C\mathbb{A}C^{*}).
\end{equation*}
\end{lemma}
\begin{proof}
For any $X\in Ks(\omega,\mathbb{A})$ we have $\sum_{i=1}^kw_i\log_X(A_i)=0$. Equivalently
\begin{equation}\label{karchereqy}
0=\sum_{i=1}^kw_i\log_X(A_i)=\sum_{i=1}^kw_i\log_I(X^{-1/2}A_iX^{-1/2})=\sum_{i=1}^kw_i\log_I(X^{-1}A_i).
\end{equation}
Let $C=UP$ the polar decomposition of $C$, i.e. $U^{-1}=U^{*}$ and $P\in\mathbb{P}$.
Then by \eqref{karchereqy} it follows directly that
\begin{equation*}
UKs(\omega,\mathbb{A})U^{*}=Ks(\omega,U\mathbb{A}U^{*}).
\end{equation*}
Similarly we have
\begin{equation*}
\begin{split}
0&=P^{-1}\left(\sum_{i=1}^kw_i\log_I(X^{-1}A_i)\right)P=\sum_{i=1}^kw_i\log_I(P^{-1}X^{-1}A_iP)\\
&=\sum_{i=1}^kw_i\log_I(P^{-1}X^{-1}P^{-1}PA_iP),
\end{split}
\end{equation*}
so $PXP\in Ks(\omega,P\mathbb{A}P)$, i.e. $PKs(\omega,\mathbb{A})P\subseteq Ks(\omega,P\mathbb{A}P)$. Also then $Ks(\omega,\mathbb{A})\subseteq P^{-1}Ks(\omega,P\mathbb{A}P)P^{-1}\subseteq Ks(\omega,\mathbb{A})$ which means
\begin{equation*}
PKs(\omega,\mathbb{A})P=Ks(\omega,P\mathbb{A}P).
\end{equation*}
From this and $UKs(\omega,\mathbb{A})U^{*}=Ks(\omega,U\mathbb{A}U^{*})$ we get that
\begin{equation*}
CKs(\omega,\mathbb{A})C^{*}=Ks(\omega,C\mathbb{A}C^{*}).
\end{equation*}

\end{proof}

\begin{proposition}\label{propunique}
There exists $\epsilon_{\omega}>1$ such that the generalized Karcher equation \eqref{karcherequx} has a unique solution on the order interval $[1/\epsilon_{\omega}A,\epsilon_{\omega}A]$ for any $A\in\mathbb{P}$.
\end{proposition}
\begin{proof}
Let $F_{\omega,\mathbb{A}}(X)=F(A_1,\ldots,A_k,X)=\sum_{i=1}^kw_i\log_I(X^{-1/2}A_iX^{-1/2})$. Then the map $F_{\omega,\mathbb{A}}:\mathbb{P}\mapsto S(E)$ is $C^{\infty}$ and $F_{\omega,\mathbb{A}}(X)=0$ if and only if $X\in Ks(\omega,\mathbb{A})$. The Fr\'echet derivative of $F_{\omega,\mathbb{A}}$ is a linear map on $S(E)$. Let $\mathbb{I}=(I,\ldots,I)\in\mathbb{P}^k$. Then $F_{\omega,\mathbb{I}}(X)=\log_I(X^{-1})$ so $F_{\omega,\mathbb{I}}(I)=0$ and by the property $\log_I'(1)=1$ we have that the Fr\'echet derivative $DF_{\omega,\mathbb{I}}[I]=-id_{E}$. Thus by the Implicit Function Theorem (Theorem 5.9 \cite{lang}) there exists and open neighborhood $U$ of $\mathbb{I}$ and a $C^{\infty}$ mapping $g:U\mapsto \mathbb{P}$ such that $F_{\omega,\mathbb{A}}(X)=0$ if and only if $x=g(\mathbb{A})$ for $\mathbb{A}\in U$. Pick $\epsilon_1>\epsilon>1$ such that $[1/\epsilon I,\epsilon I]^k\subseteq (1/\epsilon_1 I,\epsilon_1 I)^k\subseteq U$. Then for any $\mathbb{A}\in [1/\epsilon I,\epsilon I]^k$, $\Lambda_M(\omega;\mathbb{A})\in Ks(\omega,\mathbb{A})=\{g(\mathbb{A})\}$, i.e. the generalized Karcher equation has a unique solution $\Lambda_M(\omega;\mathbb{A})$, the $\omega$-weighted lambda extension of $M(\cdot,\cdot)$. It also follows that the unique solution is $C^{\infty}$ on the order interval $(1/\epsilon_1 I,\epsilon_1 I)^k$ so as $\Lambda_M(\omega;\mathbb{A})$.

Now let $A_i\in[1/\epsilon A,\epsilon A]$. Then $B_i=A^{-1/2}A_iA^{-1/2}\in[1/\epsilon I,\epsilon I]$. Then by the above we have $Ks(\omega,\mathbb{B})=\{\Lambda_M(\omega;\mathbb{B})\}$. Thus by Lemma~\ref{karchercongruenceinv} and Theorem~\ref{lambdaprop} we have
\begin{equation*}
\begin{split}
Ks(\omega,\mathbb{A})&=Ks(\omega,A^{1/2}\mathbb{B}A^{1/2})=A^{1/2}Ks(\omega,\mathbb{B})A^{1/2}=A^{1/2}\{\Lambda_M(\omega;\mathbb{B})\}A^{1/2}\\
&=\{\Lambda_M(\omega;A^{1/2}\mathbb{B}A^{1/2})\}=\{\Lambda_M(\omega;\mathbb{A})\}.
\end{split}
\end{equation*}

\end{proof}

\begin{remark}
The lambda extension $\Lambda_M(\omega;\mathbb{A})$ is $C^{\infty}$ on small enough neighborhoods of the diagonal in the product cone $\mathbb{P}^k$.
\end{remark}

\begin{theorem}\label{uniquegeneralizedkarcher}
$Ks(\omega,\mathbb{A})=\{\Lambda_M(\omega;\mathbb{A})\}$ for all $\omega\in\Delta_k$ and $A_i\in\mathbb{P}$, $1\leq i\leq k$.
\end{theorem}
\begin{proof}
We start with a
\begin{claim}
The equation
\begin{equation}\label{eq:uniquegeneralizedkarcher}
X=\Lambda_M(\omega;M_{t}(X,A_1),\ldots,M_{t}(X,A_k))
\end{equation}
has a unique solution in $\mathbb{P}$ for all $0<t\leq t_0$ and $A_i\in\mathbb{P}$.
\end{claim}
Let $S$ be a bounded subset of $\mathbb{P}$. Let $A\in\mathbb{P}$ and $r<\infty$ given such that $S\subseteq\overline{B}_A(r)$ and for all $X,Y\in S$ the functions $g_i(X)=M_t(X,A_i)$ are strict contractions for all $1\leq i\leq k$. Clearly by the boundedness of $S$ and the set $\{A_1,\ldots,A_k\}$ and Theorem~\ref{contract} such $A\in\mathbb{P}$ and $r<\infty$ exists. Suppose the largest contraction coefficient for the functions $g_i(X)=M_t(X,A_i)$ for $1\leq i\leq k$ is $\rho_t$ on $\overline{B}_A(r)$. Now let $h_t(X)=\Lambda_M(\omega;M_{t}(X,A_1),\ldots,M_{t}(X,A_k))$. Then
\begin{equation*}
\begin{split}
d_\infty(h_t(X),h_t(Y))&=d_\infty(\Lambda_M(\omega;M_{t}(X,\mathbb{A})),\Lambda_M(\omega;M_{t}(Y,\mathbb{A})))\\
&\leq \max_{1\leq i\leq k}d_\infty(M_{t}(X,A_i),M_{t}(Y,A_i))\\
&\leq \rho_t d_\infty(X,Y)
\end{split}
\end{equation*}
where the first inequality follows from Proposition~\ref{nonexpansivemean} with properties (4) and (5) of Theorem~\ref{lambdaprop}. So by Banach's fixed point theorem $h_t(X)$ has a unique fixed point on $S$, so the equation \eqref{eq:uniquegeneralizedkarcher} has a unique positive definite solution on $S$. Since $S$ was arbitrary bounded subset of $\mathbb{P}$, it follows that the same holds on all of $\mathbb{P}$. The claim is proved.

Let $X\in Ks(\omega,\mathbb{A})$. We have that $f_t(x)=\exp_I(t\log_I(x))\in\mathfrak{P}(t)$ for $0<t\leq t_0$. By Lemma~\ref{maxmininp} we have that
\begin{equation*}
\left((1-t)+tx^{-1}\right)^{-1}\leq f_t(x)\leq (1-t)+tx
\end{equation*}
which means that there exists a small enough $0<t\leq t_0$ such that
\begin{equation*}
f_t(X^{-1/2}A_iX^{-1/2})\in [1/\epsilon_{\omega}I,\epsilon_{\omega}I]
\end{equation*}
for all $1\leq i\leq k$. Also it is easy to see that
\begin{equation*}
\sum_{i=1}^kw_i\log_I(X^{-1/2}A_iX^{-1/2})=0=\sum_{i=1}^kw_i\log_I(f_t(X^{-1/2}A_iX^{-1/2})),
\end{equation*}
since $\exp_I$ is the inverse of $\log_I$. From this it follows by Proposition~\ref{propunique} and the definition of the lambda extension $\Lambda_M(\omega;\cdot)$ that
\begin{equation*}
I=\Lambda_M(\omega;f_t(X^{-1/2}A_1X^{-1/2}),\ldots,f_t(X^{-1/2}A_kX^{-1/2})).
\end{equation*}
By property (5) in Theorem~\ref{lambdaprop} we have that the above is equivalent to
\begin{equation*}
\begin{split}
X&=\Lambda_M(\omega;X^{1/2}f_t(X^{-1/2}A_1X^{-1/2})X^{1/2},\ldots,X^{1/2}f_t(X^{-1/2}A_kX^{-1/2})X^{1/2})\\
&=\Lambda_M(\omega;M_{t}(X,A_1),\ldots,M_{t}(X,A_k)).
\end{split}
\end{equation*}
Now the Claim implies that the solution of the above equation is unique, hence all solutions $X\in Ks(\omega,\mathbb{A})$ are identically $\Lambda_M(\omega;\mathbb{A})$, i.e. $Ks(\omega,\mathbb{A})=\{\Lambda_M(\omega;\mathbb{A})\}$.

\end{proof}

Let us summarize our results for $\Lambda_M(\omega;{\Bbb A})$.

\begin{theorem}\label{lambdaprop2}
Let ${\Bbb A}=(A_{1},\dots,A_{k}), {\Bbb
B}=(B_{1},\dots,B_{k})\in {\Bbb P}^{k},\omega\in \Delta_{k}$ and $M,N\in\mathfrak{M}$ and $\Lambda_M(\omega;{\Bbb A})$, $\Lambda_N(\omega;{\Bbb A})$ the corresponding lambda extensions. Then
\begin{itemize}
 \item[(1)] $\Lambda_M(\omega;{\Bbb A})=A$ if $A_i=A$ for all $1\leq i\leq k$;
 \item[(2)] $\Lambda_M(\omega_{\sigma};{\Bbb A}_{\sigma})=\Lambda_M(\omega;{\Bbb A})$ for any
 permutation $\sigma;$
 \item[(3)] $\Lambda_M(\omega;{\Bbb A})\leq \Lambda_M(\omega;{\Bbb B})$ if $A_{i}\leq
 B_{i}$ for all $i=1,2,\dots,k;$
 \item[(4)] if $M(A,B)\leq N(A,B)$ for all $A,B\in{\Bbb P}$ then $\Lambda_M(\omega;{\Bbb A})\leq \Lambda_N(\omega;{\Bbb A})$;
 \item[(5)] $\Lambda_M(\omega;X{\Bbb A}X^{*})=X\Lambda_M(\omega;{\Bbb
 A})X^{*}$ for any $X\in \mathrm{GL}(E);$
 \item[(6)] $(1-u)\Lambda_M(\omega;{\Bbb A})+u\Lambda_M(\omega;{\Bbb B})\leq
 \Lambda_M(\omega;(1-u){\Bbb A}+u{\Bbb B})$ for any $u\in [0,1];$
 \item[(7)]  $d_{\infty}(\Lambda_M(\omega;{\Bbb A}), \Lambda_M(\omega;{\Bbb
 B}))\leq \max_{1\leq i\leq k}\{d_{\infty}(A_{i},B_{i})\};$
 \item[(8)] $\Lambda_M(\omega^{(n)};{\Bbb A}^{(n)})=\Lambda_M(\omega;{\Bbb A})$
 for any $n\in {\Bbb N};$
\item[(9)]
$\Phi(\Lambda_M(\omega;{\Bbb A}))\leq
\Lambda_M(\omega;\Phi({\Bbb A}))$ for any positive unital linear map
$\Phi,$ where $\Phi({\Bbb A})=(\Phi(A_{1}),\dots,\Phi(A_{k}));$
\item[(10)]
$\left(\sum_{i=1}^kA_i^{-1}\right)^{-1}\leq \Lambda_M(\omega;{\Bbb A})\leq \sum_{i=1}^kA_i;$
 \item[(11)]
$\Lambda_M(\omega;A_{1},\dots,A_{k-1},X)=X$ if and only if
$X=\Lambda_M({\hat \omega};A_{1},\dots,A_{k-1}).$ In particular,
$\Lambda_M(A_{1},\dots,A_{k},X)=X$ if and only if
$X=\Lambda_M(A_{1},\dots,A_{k});$
 \item[(12)] $\Lambda_M(\omega;{\Bbb A})$ is the unique solution of the operator equation $\sum_{i=1}^kw_i\log_X(A_i)=0$ where $\log_X(A)=X^{1/2}\log_I\left(X^{-1/2}AX^{-1/2}\right)X^{1/2}.$
\end{itemize}
\end{theorem}
\begin{proof}
We only need to show the previously unproved property (11). By Theorem~\ref{uniquegeneralizedkarcher} we have that $\Lambda_M(\omega;A_{1},\dots,A_{k-1},X)=X$ if and only if $\sum_{i=1}^{k-1}w_i\log_X(A_i)+w_k\log_I(I)=0$. Since $\log_I(I)=0$ we get that $\sum_{i=1}^{k-1}w_i\log_X(A_i)=0$, i.e. $\sum_{i=1}^{k-1}\frac{w_i}{1-w_k}\log_X(A_i)=0$.

\end{proof}

\begin{remark}
In \cite{lawsonlim1} Lawson and Lim proved the above theorem for the Karcher mean $\Lambda(\omega;\mathbb{A})$ given in \eqref{karchermean} using the matrix power means $P_t(\omega;\mathbb{A})$. The corresponding Karcher equation is \eqref{karcherequ}. The argument given here for the general case is the extension of their argument to cover all possible cases of induced logarithm maps $\log_I\in\mathfrak{L}$.
\end{remark}

\begin{remark}
The lambda extensions $\Lambda_M(\omega;\mathbb{A})$ of the two-variable weighted arithmetic and harmonic means are multivariable versions $\sum_{i=1}^kw_iA_i$ and $\left(\sum_{i=1}^kw_iA_i^{-1}\right)^{-1}$ respectively. This is so since the one parameter families of the corresponding means $M_t(\omega;\mathbb{A})$ are already $\sum_{i=1}^kw_iA_i$ and $\left(\sum_{i=1}^kw_iA_i^{-1}\right)^{-1}$ respectively.
\end{remark}

\begin{remark}
Theorem~\ref{uniquegeneralizedkarcher} gives us a tool to solve operator equations that can be written in the form of a generalized Karcher equation
\begin{equation*}
\sum_{i=1}^kw_i\log_X(A_i)=0
\end{equation*}
where the corresponding $\log_I$ is induced by an $M\in\mathfrak{M}$. The solution can be calculated by choosing a sequence $t_l\to 0+$ as $l\to\infty$ and then taking the limit
\begin{equation*}
\lim_{l\to\infty}M_{t_l}(\omega;\mathbb{A})=\Lambda_M(\omega;\mathbb{A}).
\end{equation*}
\end{remark}

\begin{corollary}\label{cor:opmean2}
If $k=2$, $\Lambda_M(w_1,w_2;A,B)\in\mathfrak{M}$ is an operator mean.
\end{corollary}
\begin{proof}
By Theorem~\ref{inducedconv} the lambda extension is the strong limit of (induced) operator means, i.e.
\begin{equation*}
\Lambda_M(w_1,w_2;A,B)=\lim_{t\to 0+}M_{t}(w_1,w_2;A,B).
\end{equation*}
By Lemma 6.1 in \cite{kubo} the pointwise weak limit of operator means is an operator mean as well, so therefore it follows that the strong limit $\Lambda_M(w_1,w_2;A,B)$ of operator means is also an operator mean in the sense of Definition~\ref{symmean}.

\end{proof}

Similarly to the case of the induced means to find closed formulas for $\Lambda_M$ is nontrivial. Although for two variables there is an analogue for Proposition~\ref{P:2variable} in the case of $\Lambda_M$ as well.

\begin{proposition}\label{P:2variablelambda}
Let $\omega\in\Delta_2$, $A,B\in\mathbb{P}$ and $M\in\mathfrak{M}$ with representing function $f(x)=\exp_I(t\log_I(x))$. Then
\begin{eqnarray*}
\Lambda_M(w_1,w_2;A,B)=A^{1/2}g\left(A^{-1/2}BA^{-1/2}\right)A^{1/2}
\end{eqnarray*}
where
\begin{eqnarray}\label{ginverse2}
g^{-1}(x)=x\exp_I\left(-\frac{w_1}{w_2}\log_I(x^{-1})\right).
\end{eqnarray}
\end{proposition}
\begin{proof}
The proof in principle is the same as the proof of Proposition~\ref{P:2variable}.

\end{proof}

\section{Further properties of induced operator means and lambda extensions}
Int this section we consider further properties and problems related to induced operator means and lambda extensions of operator means. One of the main problems here are the characterization of the set of lambda extensions in two variables. The reason for that is, that given a lambda extension $\Lambda_M(\omega;{\Bbb A})$, it can be regarded as a multivariate extension of its two variable version $\Lambda_M(w_1,w_2;A_1,A_2)$. Therefore if a 2-variable operator mean in $\mathfrak{M}$ is a lambda extension, then it automatically admits multivariate extensions through its lambda extension. Similar situation arises in the case of induced operator means.

We start with some basic observations. Elements in $\mathfrak{P}(t)$ directly generate elements of $\mathfrak{L}$.

\begin{proposition}\label{P:PtgenerateL}
Let $f\in\mathfrak{P}(t)$. Then the function $\log_I(x)=\frac{f(x)-1}{f'(1)}$ is in $\mathfrak{L}$.
\end{proposition}
\begin{proof}
The function $\log_I$ is operator monotone and also $\log_I(1)=0$ and $\log'_I(1)=1$, hence is in $\mathfrak{L}$.

\end{proof}

It immediately follows that the generalized Karcher equations corresponding to such $\log_I$ admit unique solutions which are actually induced operator means, similarly to the case of the matrix power means in Proposition~\ref{powerkarcher}:

\begin{theorem}\label{P:lambdaisinduced}
All induced operator means $M(\omega;\mathbb{A})$ are unique solutions of generalized Karcher equations corresponding to logarithm maps given in Proposition~\ref{P:PtgenerateL}.
\end{theorem}
\begin{proof}
Let $\omega\in\Delta_k$ and $\mathbb{A}\in\mathbb{P}^k$. Let $M\in\mathfrak{M}$ with representing function $f(x)$. By Proposition~\ref{P:PtgenerateL}, the function $\log_I(x)=\frac{f(x)-1}{f'(1)}$ is in $\mathfrak{L}$. Now the induced operator mean $M(\omega;\mathbb{A})$ is the unique solution of the operator equation \eqref{meanequ3}
\begin{equation*}
X=\sum_{i=1}^kw_iM(X,A_i).
\end{equation*}
This is equivalent to
\begin{eqnarray*}
0&=&\sum_{i=1}^kw_i[M(X,A_i)-X]\\
0&=&\sum_{i=1}^kw_i\frac{M(X,A_i)-X}{f'(1)}\\
0&=&\sum_{i=1}^kw_iX^{1/2}\log_I(X^{-1/2}A_iX^{-1/2})X^{1/2}\\
0&=&\sum_{i=1}^kw_i\log_X(A_i),
\end{eqnarray*}
a generalized Karcher equation.

\end{proof}

The above property of induced operator means is an notable structural result which makes induced means quite similar to lambda extensions. It is important to note however, that such $\log_I$ might not give rise to operator means in the form $f_t(x)=\exp_I(t\log_I(x))$ for any $t\in(0,1)$, since $f_t(x)$ might have ramification points (where it is obviuosly not holomorphic) corresponding to the ramification points of $\log_I$, hence the $\exp_I$ lack holomorphic inverses at those points.

How about the converse of the above result? Given a $\log_I\in\mathfrak{L}$ we might provide the unique solutions of the operator equations $\sum_{i=1}^kw_i\log_X(A_i)=0$ directly using the induced mean $M(\omega;\mathbb{A})$.
\begin{proposition}
Let $\log_I\in\mathfrak{L}$ and suppose that $\lim_{x\to 0+}\log_I(x)>-\infty$. Then the function
\begin{equation*}
f_t(x)=t\log_I(x)+1
\end{equation*}
is in $\mathfrak{P}(t)$ for $0<t\leq\frac{1}{|\lim_{x\to 0+}\log_I(x)|}$.
\end{proposition}
\begin{proof}
Since $f_t$ is operator monotone hence monotonically increasing, therefore $f_t(x)>0$ for all $x\in(0,\infty)$ by operator monotonicity and the fact that $\lim_{x\to 0+}f_t(x)\geq 0$. Also $f_t(1)=1$, so $f_t\in\mathfrak{P}(t)$.

\end{proof}

\begin{proposition}
Let $\log_I\in\mathfrak{L}$ with $\lim_{x\to 0+}\log_I(x)>-\infty$ so by the previous result $M_t(I,x)=f_t(x)=t\log_I(x)+1$ is in $\mathfrak{P}(t)$ for $0<t\leq\frac{1}{|\lim_{x\to 0+}\log_I(x)|}$. Then the induced operator mean $M_t(\omega,\mathbb{A})$ is the unique solution of the generalized Karcher equation
\begin{equation*}
\sum_{i=1}^kw_i\log_X(A_i)=0
\end{equation*}
in $\mathbb{P}$ where $\log_X(A)=X^{1/2}\log_I\left(X^{-1/2}AX^{-1/2}\right)X^{1/2}$.
\end{proposition}
\begin{proof}
By definition the induced operator mean $X_t=M_t(\omega,\mathbb{A})$ is the unique solution of
\begin{equation*}
X_t=\sum_{i=1}^kw_iM_t(X_t,A_i)
\end{equation*}
in $\mathbb{P}$. The above is equivalent to
\begin{equation*}
0=\sum_{i=1}^kw_iX_t^{1/2}\frac{M_t(I,X_t^{-1/2}A_iX_t^{-1/2})-I}{t}X_t^{1/2}=\sum_{i=1}^kw_i\log_{X_t}(A_i).
\end{equation*}

\end{proof}

\begin{remark}
In the above case it is clear that every $M_t(\omega;\mathbb{A})$ is just the same for all $0<t\leq\frac{1}{|\lim_{x\to 0+}\log_I(x)|}$, also $M_t(\omega;\mathbb{A})$ is a one parameter family of induced operator means, although they generally do not form those one parameter families of induced operator means, in the sense of the previous section, that lead to a lambda extension by letting $t\to 0+$.
\end{remark}

One might wonder whether all lambda extensions occur as induced matrix means. If that would be the case we could directly calculate lambda extensions using induced operator means. This is not the case however:

\begin{proposition}\label{P:2variablelambdainduced}
Let $M\in\mathfrak{M}$ with representing function $f(x)=\exp_I(t\log_I(x))$. Suppose that the corresponding lambda extension $\Lambda_M(w_1,w_2;A,B)=N(w_1,w_2;A,B)$ for all $\omega\in\Delta_2$, $A,B\in\mathbb{P}$ where $N\in\mathfrak{M}$ with representing function $g(x)$. Then
\begin{eqnarray*}
\log_I(x)=\frac{g(x)-1}{g'(1)}.
\end{eqnarray*}
\end{proposition}
\begin{proof}
By Proposition~\ref{P:2variable} and~\ref{P:2variablelambda} we have that
\begin{eqnarray*}
x\exp_I\left(-\frac{w_1}{w_2}\log_I(x^{-1})\right)&=&xg^{-1}\left(\frac{1-w_1g(x^{-1})}{w_2}\right)\\
\exp_I\left(-\frac{w_1}{1-w_1}\log_I(x^{-1})\right)&=&g^{-1}\left(\frac{1-w_1g(x^{-1})}{1-w_1}\right)\\
\frac{\partial}{\partial w_1}\left.\exp_I\left(-\frac{w_1}{1-w_1}\log_I(x^{-1})\right)\right|_{w_1=0}&=&\left.\frac{\partial}{\partial w_1}g^{-1}\left(\frac{1-w_1g(x^{-1})}{1-w_1}\right)\right|_{w_1=0}\\
-\log_I(x^{-1})&=&\frac{1}{g'(1)}(1-g(x^{-1}))
\end{eqnarray*}
from which the assertion follows.

\end{proof}

\begin{remark}
Not all $\log_I\in\mathfrak{L}$ can be given in the above form. The most convenient counterexample is the principal branch of the logarithm $\log_I(x)=\log(x)$. It is easy to see that if $\log(x)=\frac{g(x)-1}{g'(1)}$ then $g(x)=g'(1)\log(x)+1$. There exists no such $g(x)$, since the function $s\log(x)+1$ is although operator monotone, but is not positive on $(0,\infty)$  close to zero for any finite $s>0$, hence it cannot be a representing function of any member of $\mathfrak{M}$.
\end{remark}

How big is the set of lambda extensions $\Lambda_M(w_1,w_2;A,B)$ in $\mathfrak{M}$? Is any element of $\mathfrak{M}$ a lambda extension of some $M\in\mathfrak{M}$? If that is the case then every operator mean in the sense of Kubo-Ando occurs as a lambda extension, hence has a multivariable extension $\Lambda_M(\omega;{\Bbb A})$ with the same $M\in\mathfrak{M}$. We investigate this question now.

\begin{proposition}\label{P:2variablelambdainPt}
Let $M\in\mathfrak{M}$ with representing function $f(x)=\exp_I(t\log_I(x))$. Then the representing function $g(x)$ of the corresponding lambda extension $\Lambda_M(w_1,w_2;A,B)$ for all $\omega\in\Delta_2$, $A,B\in\mathbb{P}$ is in $\mathfrak{P}(w_2)$.
\end{proposition}
\begin{proof}
By Proposition~\ref{P:2variablelambda} we have that
\begin{eqnarray*}
g^{-1}(x)&=&x\exp_I\left(-\frac{w_1}{w_2}\log_I(x^{-1})\right)\\
\frac{\partial}{\partial x}\left.g^{-1}(x)\right|_{x=1}&=&\frac{\partial}{\partial x}\left.\exp_I\left(-\frac{w_1}{w_2}\log_I(x^{-1})\right)\right|_{x=1}\\
\frac{1}{g'(g^{-1}(1))}&=&\exp_I\left(-\frac{w_1}{w_2}\log_I(x^{-1})\right)\\
&&+\left.x\exp'_I\left(-\frac{w_1}{w_2}\log_I(x^{-1})\right)\frac{w_1}{w_2}\log'_I(x^{-1})x^{-2}\right|_{x=1}\\
\frac{1}{g'(1)}&=&1+\frac{w_1}{w_2}=\frac{1}{w_2}
\end{eqnarray*}
where we used that $\exp_I(0)=1$, $\exp'_I(0)=1$, $\log_I(1)=0$, $\log'_I(1)=1$. Now by Theorem~\ref{lambdaprop2} we know that $\Lambda_M(w_1,w_2;A,B)$ is a matrix or operator mean since the properties in Definition~\ref{symmean} are fulfilled hence the representing function $g(x)$ is positive operator monotone on $(0,\infty)$, moreover $g'(1)=w_2$, so $g\in\mathfrak{P}(w_2)$.

\end{proof}

By the previous Proposition~\ref{P:2variablelambdainPt} it is clear that if an operator mean is a lambda extension, then the derivative $g'(1)=w_2$ of its representing function, where $\omega=(w_1,w_2)$ is the weight of the lambda extension. In the next result we will use the following fact. A real function $g(x)$ is a representing function of a mean in $\mathfrak{M}$ if and only if the function $g^*(x)=\frac{x}{g(x)}$ is a representing function of a mean in $\mathfrak{M}$, i.e. it is positive operator monotone on $(0,\infty)$ (cf. Proposition 7.1 \cite{schilling}). In this setting we say that $g^*(x)$ is the conjugate pair of $g(x)$ and vice versa.

\begin{theorem}\label{2varlambdacharacterization}
Let $M\in\mathfrak{M}$ be an operator mean with representing function $g(x)$ such that $g'(1)\neq 0,1/2,1$. Let $g^*(x)=\frac{x}{g(x)}$ denote the conjugate pair. Define the function $h(x)$ s.t.
\begin{enumerate}
\item if $g'(1)<1/2$ then $h(x):=\frac{x}{g^{*-1}(x)}$
\item if $g'(1)>1/2$ then $h(x):=xg^{-1}(x^{-1}).$
\end{enumerate}
Then $M$ is a lambda extension if and only if there exists a positive integer $n$, such that the function $h^{\circ 2n}(x)$ is in $\mathfrak{m}$, i.e. it is a representing function of an operator mean in $\mathfrak{M}$. Moreover in this case the function $\log_I\in\mathfrak{L}$ in \eqref{ginverse2} in Proposition~\ref{P:2variablelambda} is unique.
\end{theorem}
\begin{proof}
Suppose that $M$ is a lambda extension. Then by Proposition~\ref{P:2variablelambda} and some simple calculation with $g^*(x)$ we have
\begin{eqnarray*}
xg^{-1}(x^{-1})&=&\exp_I\left(-\frac{w_1}{w_2}\log_I(x)\right)\\
\frac{x}{g^{*-1}(x)}&=&\exp_I\left(-\frac{w_2}{w_1}\log_I(x)\right)
\end{eqnarray*}
where $\log_I\in\mathfrak{L}$ and there exists $t_0\in(0,1]$ s.t. $f_{t_0}(x)=\exp_I(t_0\log_I(x))$ is in $\mathfrak{P}(t_0)$. Then by Proposition~\ref{ftinPt} we have that for all $t\in[0,t_0)$ the function $f_{t}(x)=\exp_I(t\log_I(x))$ is in $\mathfrak{P}(t)$. By the previous Proposition~\ref{P:2variablelambdainPt} we have that $w_2=g'(1)$, moreover it is not hard to see that for the conjugate pair $w_1={g^*}'(1)$. So this means that either $g'(1)<1/2$ or $g'(1)>1/2$ the derivative $|h'(1)|<1$, since then $h'(1)$ is either $-w_2/w_1$ or $-w_1/w_2$. Also we have that $h^{\circ 2}(x)=\exp_I\left(h'(1)^2\log_I(x)\right)$ where $0\leq h'(1)^2<1$. So the iterates
\begin{equation*}
h^{\circ 2n}(x)=\exp_I\left(h'(1)^{2n}\log_I(x)\right),
\end{equation*}
i.e. there exists a positive integer $n$, s.t. $h'(1)^{2n}\leq t_0$. This means that $h^{\circ 2n}\in\mathfrak{P}(h'(1)^{2n})$ by Proposition~\ref{ftinPt}.

Conversely suppose that there exists a positive integer $n$, such that the function $h^{\circ 2n}(x)\in\mathfrak{m}$ (it is a representing function of an operator mean), or equivalently $h^{\circ 2n}\in\mathfrak{P}(h'(1)^{2n})$. This also means that by denoting $c=\left.\frac{d}{dx}h^{\circ 2n}(x)\right|_{x=1}$ we have according to Theorem~\ref{uniformconv} and Proposition~\ref{meanswithexp} that
\begin{equation*}
h^{\circ 2n}(x)=\exp_I\left(c\log_I(x)\right),
\end{equation*}
where $\log_I$ is unique among functions in $\mathfrak{L}$. Also note that by the definition of $h(x)$ it follows that $c=h'(1)^{2n}$ and also $-1<h'(1)<0$. Now by Theorem 11.6.1 in \cite{kuczma} concerning locally analytic iterative roots of holomorphic functions having a fixed point with corresponding derivative at the fixed point which has modulus strictly less than 1, it follows that all $2n$-th iterative holomorphic roots of $h^{\circ 2n}(z)$ in the neighborhood of the fixed point $1$ are of the form 
\begin{equation*}
h(z)=\exp_I\left(c^{1/(2n)}\log_I(z)\right)
\end{equation*}
in an open neighborhood of the fixed point $1$, where $c^{1/(2n)}$ can be any complex $2n$-th root of $c$. But in our case $-1<h'(1)<0$ by definition of $h(x)$, so it follows that
\begin{equation}\label{eq:2varlambdacharacterization1}
h(z)=\exp_I\left(h'(1)\log_I(z)\right)
\end{equation}
in an open neighborhood of the fixed point $1$, since the derivatives of the iterative root and $h(z)$ at the fixed point $1$ must be the same. Moreover Theorem 4.6.1 in \cite{kuczma} also ensures us that this $\log_I$ in \eqref{eq:2varlambdacharacterization1} is unique among holomorphic functions $\sigma$ with $\sigma(1)=0$, $\sigma'(1)=1$ in the neighborhood of the fixed point $1$, hence among functions in $\mathfrak{L}$. Since $h'(1)$ is either $-w_2/w_1$ or $-w_1/w_2$ given that either $g'(1)<1/2$ or $g'(1)>1/2$ with $w_1={g^*}'(1)$ and $w_2=1-w_1=g'(1)$ we arrive at the two equations
\begin{eqnarray*}
zg^{-1}(z^{-1})&=&\exp_I\left(-\frac{w_1}{w_2}\log_I(z)\right)\\
\frac{z}{g^{*-1}(z)}&=&\exp_I\left(-\frac{w_2}{w_1}\log_I(z)\right),
\end{eqnarray*}
where the first equation holds if $g'(1)>1/2$ and if $g'(1)<1/2$ then the second one is fulfilled in a neighborhood of the fixed point $1$. Now since the functions $g$ and $\log_I$ are holomorphic everywhere on $\mathbb{C}\setminus [0,-\infty)$ we obtain from this by analytic continuation that
\begin{equation*}
w_1\log_I(1/g(z))+w_2\log_I(z/g(z))=0
\end{equation*}
is satisfied, which is the generalized Karcher equation for $\log_I$ and since by assumption $h^{\circ 2n}\in\mathfrak{P}(h'(1)^{2n})$, therefore $\log_I$ generates a one parameter family of induced operator means converging to the lambda extension. We already derived that $\log_I$ is uniquely determined by $h$ through \eqref{eq:2varlambdacharacterization1}. Also it is easy to see that $h$ is uniquely determined by $g$, so actually $g$ uniquely determines $\log_I$ and vice versa.

\end{proof}

We were unable to derive similar characterizations in the case when $g'(1)=1/2$. There is one clue however.
\begin{proposition}\label{2varlambdacharacterizatio2}
Let $M\in\mathfrak{M}$ be an operator mean with representing function $g(x)$ such that $g'(1)=1/2$. Define the function as $h(x):=xg^{-1}(x^{-1})$. Then $M$ is a lambda extension if and only if $g(x)=xg(1/x)$ and there exists a holomorphic function $k(z)$ with $k(1)=0$, $k'(1)=1/2$ s.t.
\begin{equation}\label{eq:2varlambdacharacterizatio2}
\log_I(z)=-k(h(z))+k(z)
\end{equation}
where $\log_I\in\mathfrak{L}$ and there exists a $t\in(0,1]$ such that the function $f_t(x)=\exp_I(t\log_I(x))$ is in $\mathfrak{P}(t)$ (where $\exp_I$ denotes the inverse of $\log_I$ as usual).
\end{proposition}
\begin{proof}
This result is based on Theorem 4.6.2 in \cite{kuczma} and the functional equation
\begin{equation}\label{eq:2varlambdacharacterizatio2:2}
\log_I(h(z))=-\log_I(z)
\end{equation}
which follows from Proposition~\ref{P:2variablelambda} and Proposition~\ref{P:2variablelambdainPt} with $w_2=g'(1)$. Theorem 4.6.2 in \cite{kuczma} says (with transforming the fixed point $1$ to $0$) that \eqref{eq:2varlambdacharacterizatio2:2} has a solution $\log_I(z)$ with $\log_I(1)=0$ and $\log_I'(1)=1$ if and only if $h(h(z))=z$ and all solutions of \eqref{eq:2varlambdacharacterizatio2:2} are given in the form \eqref{eq:2varlambdacharacterizatio2} where $k(z)$ ranges over all holomorphic functions with $k(1)=0$, $k'(1)=1/2$. Now some calculation reveals that $g(x)=xg(1/x)$ is equivalent to $h(h(x))=x$. Also there must exist a $t\in(0,1]$ such that the function $f_t(x)=\exp_I(t\log_I(x))$ is in $\mathfrak{P}(t)$, otherwise $\log_I$ does not generate a one parameter family of induced operator means that converge to the lambda extension as $t\to 0+$.

\end{proof}


\begin{remark}
It is easy to see that there exist representing functions $g$ in $\mathfrak{m}$ such that $g'(1)=1/2$, but $g(x)\neq xg(1/x)$. By Proposition~\ref{harmonicreprprop} and Proposition~\ref{inPt}
\begin{equation*}
g(x)=\int_{[0,1]}(1-s+sx^{-1})^{-1}d\nu(s)
\end{equation*}
where $\nu$ is a probability measure on $[0,1]$ and $g'(1)=\int_{[0,1]}sd\nu(s)$. We also have that property $g(x)=xg(1/x)$ is equivalent to the symmetricity of the represented operator mean by $g$, which itself is equivalent to $d\nu(s)=d\nu(1-s)$ for all $s\in[0,1]$ according to Corollary~\ref{chracterizesymmetric}. It is easy to construct a probability measure $\nu$ on $[0,1]$ such that $\int_{[0,1]}sd\nu(s)=1/2$, but $d\nu(s)\neq d\nu(1-s)$. So it follows that there exist operator means $M\in\mathfrak{M}$ such that they are not lambda extensions, so the set of 2-variable lambda extensions form a proper subset of $\mathfrak{M}$.
\end{remark}

\begin{remark}
We mention one more problem in this section. This is related to the integral representation by Proposition~\ref{harmonicreprprop} of positive operator monotone $f(x)$ on $(0,\infty)$ with $f(1)=1$:
\begin{equation}\label{harmonicrepr2}
f(x)=\int_{[0,1]}[(1-s)+sx^{-1}]^{-1}d\nu(s)
\end{equation}
where $\nu$ is a unique probability measure over the closed interval $[0,1]$. This representation has a natural analogue in multiple variables for $A_i\in\mathbb{P}$, $1\leq i\leq k$:
\begin{equation}\label{harmonicreprk}
M(A_1,\ldots,A_k)=\int_{\Delta_k}\left(\sum_{i=1}^kw_iA_i^{-1}\right)^{-1}d\nu(w_1,\ldots,w_k)
\end{equation}
where $\nu$ is a probability measure over the simplex $\Delta_k$. It is clear that $M(A_1,\ldots,A_k)$ is operator monotone in each of its entries since it is the convex combination (with respect to the probability measure $\nu$ supported over the compact $\Delta_k$) of weighted harmonic means, moreover fulfills the property $CM(A_1,\ldots,A_k)C^*=M(CA_1C^*,\ldots,CA_kC^*)$ and also $M(A,\ldots,A)=A$. Moreover for positive real numbers it turns into a positive real function. The question is how large is the set of $k$-variable functions on $\mathbb{P}$ that can be represented in the form \eqref{harmonicreprk}? Are all functions $M(A_1,\ldots,A_k):\mathbb{P}^k\mapsto\mathbb{P}$ with properties
\begin{enumerate}
	\item $M(A,\ldots,A)=A$,
	\item $CM(A_1,\ldots,A_k)C^*=M(CA_1C^*,\ldots,CA_kC^*)$ for all $C\in \mathrm{GL}(E)$,
	\item $M(A_1,\ldots,A_k)$ is operator monotone,
	\item $M(a_1,\ldots,a_k)$ is real for all $0<a_i\in\mathbb{R}$,
\end{enumerate}
representable in the form \eqref{harmonicreprk}?
\end{remark}

\section{The classification of affine matrix means}
In this section we turn back to one of the first mentioned problems for matrix means and characterize all affine matrix means. In order to do this we stick again to the finite dimensional case of $\textit{P}(n,\mathbb{C})$. Due to Proposition~\ref{mappingsprop} we have the exponential and logarithm map of affine matrix means in the form
\begin{equation}\label{mappings3}
\begin{split}
\exp_p(X)=p^{1/2}\exp_I\left(p^{-1/2}Xp^{-1/2}\right)p^{1/2}\\
\log_p(X)=p^{1/2}\log_I\left(p^{-1/2}Xp^{-1/2}\right)p^{1/2}
\end{split}
\end{equation}
for $p\in \textit{P}(n,\mathbb{C})$, where $\exp_I(X)$ and $\log_I(X)$ are analytic functions. The function $\exp_I:\textit{H}(n,\mathbb{C})\mapsto \textit{P}(n,\mathbb{C})$ and $\log_I(X)$ is its inverse, $\log'_I(I)=I, \exp'_I(0)=I, \log_I(I)=0, \exp_I(0)=I$. Suppose that \eqref{mappings3} represent the exponential and logarithm map of an affinely connected manifold. Then the analytic function $\exp_I(t)$ is the solution of some geodesic equations 
\begin{equation*}
\exp''_I(t)+\Gamma\left(\exp'_I(t),\exp'_I(t),\exp_I(t)\right)=0\text{,}
\end{equation*}
where $\Gamma(\cdot,\cdot,\cdot):\textit{H}(n,\mathbb{C})\times \textit{H}(n,\mathbb{C})\times \textit{P}(n,\mathbb{C})\mapsto \textit{H}(n,\mathbb{C})$ is a smooth function in all variables and linear in the first two, representing the Christoffel symbols of an affine connection. By Propostion~15 and Corollary~16 of Chapter~6 in \cite{spivak} we know that connections which have the same torsion and geodesics are identical and for an arbitrary connection there is a unique connection with vanishing torsion and with the same geodesics. If we have an affine connection with non-symmetric Christoffel symbols $\Gamma^{i}_{jk}$, it has the same geodesics as its symmetric part $\frac{\Gamma^{i}_{jk}+\Gamma^{i}_{kj}}{2}$, so without loss of generality we can assume in our case that all connections are symmetric, so we will be considering mappings $\Gamma(\cdot,\cdot,\cdot)$ which are symmetric in their first two arguments.

\begin{proposition}\label{transform1}
Suppose that $\Gamma(\cdot,\cdot,\cdot), \exp_I(\cdot), \exp_p(\cdot)$ are functions given with the above properties. Then
\begin{equation}
\Gamma(X,X,p)=p^{1/2}\Gamma\left(p^{-1/2}Xp^{-1/2},p^{-1/2}Xp^{-1/2},I\right)p^{1/2}
\end{equation}
for $p\in \textit{P}(n,\mathbb{C})$ and $X\in \textit{H}(n,\mathbb{C})$.
\end{proposition}
\begin{proof}
Let $\gamma(t)=\exp_I\left(p^{-1/2}Xp^{-1/2}t\right)$. Since $\exp_I$ is an analytic function we have
\begin{equation*}
\begin{split}
\dot{\gamma}(t)&=p^{-1/2}Xp^{-1/2}\exp'_I\left(p^{-1/2}Xp^{-1/2}t\right)\\
\ddot{\gamma}(t)&=p^{-1/2}Xp^{-1/2}\exp''_I\left(p^{-1/2}Xp^{-1/2}t\right)p^{-1/2}Xp^{-1/2}
\end{split}
\end{equation*}
and other formulas hold for $\dot{\gamma}(t)$ and $\ddot{\gamma}(t)$ similarly to the second part of the proof of Theorem~\ref{geodesics}. By the geodesic equations we have
\begin{equation*}
\begin{split}
&\ddot{\gamma}(t)=-\Gamma\left(\dot{\gamma}(t),\dot{\gamma}(t),\gamma(t)\right)\\
Xp^{-1/2}\exp''_I\left(p^{-1/2}Xp^{-1/2}t\right)&p^{-1/2}X=-p^{1/2}\Gamma\left(p^{-1/2}Xp^{-1/2}\exp'_I\left(p^{-1/2}\times\right.\right.\\
\left.Xp^{-1/2}t\right),p^{-1/2}Xp^{-1/2}\exp'_I&\left(p^{-1/2}Xp^{-1/2}t\right),\left.\exp_I\left(p^{-1/2}Xp^{-1/2}t\right)\right)p^{1/2}\text{.}
\end{split}
\end{equation*}
If we consider the geodesic equations for $\gamma(t)=\exp_p(Xt)$ we get
\begin{equation*}
\begin{split}
Xp^{-1/2}\exp''_I\left(p^{-1/2}Xp^{-1/2}t\right)p^{-1/2}X=-\Gamma\left(Xp^{-1/2}\exp'_I\left(p^{-1/2}Xp^{-1/2}t\right)p^{1/2},\right.\\
p^{1/2}\exp'_I\left.\left(p^{-1/2}Xp^{-1/2}t\right)p^{-1/2}X,p^{1/2}\exp_I\left(p^{-1/2}Xp^{-1/2}t\right)p^{1/2}\right)\text{.}
\end{split}
\end{equation*}
The left hand sides of the two equations above are the same so as the right hand sides. Taking $t=0$ and that $\exp'_I(0)=I, \exp_I(0)=I$ we get for all $p\in \textit{P}(n,\mathbb{C}), X\in \textit{H}(n,\mathbb{C})$ that
\begin{equation*}
\begin{split}
p^{1/2}\Gamma\left(p^{-1/2}Xp^{-1/2},p^{-1/2}Xp^{-1/2},I\right)p^{1/2}=\\
=\Gamma\left(X,X,p\right)\text{,}
\end{split}
\end{equation*}
which proves the assertion.

\end{proof}

By the above result we have just reduced the problem of characterizing $\Gamma\left(X,X,p\right)$ to the characterzation of $\Gamma\left(X,X,I\right)$. Now we will show that $\Gamma\left(X,X,p\right)$ is invariant under similarity transformations.

\begin{proposition}\label{transform2}
For all $p\in \textit{P}(n,\mathbb{C})$ and $X\in \textit{H}(n,\mathbb{C})$ and invertible $S$ we have
\begin{equation}
\Gamma\left(SXS^{-1},SXS^{-1},SpS^{-1}\right)=S\Gamma\left(X,X,p\right)S^{-1}\text{.}
\end{equation}
\end{proposition}
\begin{proof}
We have by the geodesic equations
\begin{equation*}
\begin{split}
X^2\exp_I''(Xt)&=-\Gamma\left(X\exp_I'(Xt),X\exp_I'(Xt),\exp_I(Xt)\right)\\
SX^2\exp_I''(Xt)S^{-1}&=-S\Gamma\left(X\exp_I'(Xt),X\exp_I'(Xt),\exp_I(Xt)\right)S^{-1}\text{.}
\end{split}
\end{equation*}
Similarly if we consider the geodesic equations for the curve $\gamma(t)=\exp_I\left(SXS^{-1}t\right)$ we get
\begin{equation*}
\begin{split}
SX^2S^{-1}\exp_I''(SXS^{-1}t)&=-\Gamma\left(SXS^{-1}\exp_I'(SXS^{-1}t),SXS^{-1}\exp_I'(SXS^{-1}t),\right.\\
&\left.\exp_I(SXS^{-1}t)\right)\\
SX^2\exp_I''(Xt)S^{-1}&=-\Gamma\left(SX\exp_I'(Xt)S^{-1},SX\exp_I'(Xt)S^{-1},\right.\\
&\left.S\exp_I(Xt)S^{-1}\right)\text{.}
\end{split}
\end{equation*}
Again since the above two equations are identical we get the assertion.

\end{proof}

By the above proposition we have for Hermitian $X$ that
\begin{equation}
\Gamma\left(X,X,I\right)=U\Gamma\left(D,D,I\right)U^{*}\text{,}
\end{equation}
for some diagonal $D$ and unitary $U$, so it is enough to characterize $\Gamma\left(X,X,I\right)$ for diagonal $X$.

\begin{theorem}\label{transform3}
Let $D$ be diagonal with real coefficients. Then
\begin{equation}
\Gamma\left(D,D,I\right)=-cD^2\text{,}
\end{equation}
for some real valued constant $c$.
\end{theorem}
\begin{proof}
First we will show that $\Gamma\left(I,I,I\right)=cI$ for some real constant $c$. Consider the case when $\gamma(t)=\exp_I(\lambda It)$ for some real $\lambda$. Then by the geodesic equations for $\gamma(t)$ we have
\begin{equation*}
\lambda^2\exp_I''(\lambda It)=-\Gamma\left(\lambda\exp_I'(\lambda It),\lambda\exp_I'(\lambda It),\exp_I(\lambda It)\right)\text{.}
\end{equation*}
By linearity of $\Gamma(\cdot,\cdot,\cdot)$ in the first two variables, this is equivalent to
\begin{equation*}
\lambda^2\exp_I''(\lambda It)=-\lambda^2\Gamma\left(\exp_I'(\lambda It),\exp_I'(\lambda It),\exp_I(\lambda It)\right)\text{.}
\end{equation*}
Letting $t=0$ we get
\begin{equation*}
cI=-\Gamma\left(I,I,I\right)\text{,}
\end{equation*}
where $c=\exp_I''(0)$ is a real number, since $\exp_I:\textit{H}(n,\mathbb{C})\mapsto \textit{P}(n,\mathbb{C})$ is an analytic function with real coefficients in its Taylor series.

The next step is to show that for a projection $P=P^2=P^*$ we have $\Gamma\left(P,P,I\right)=-cP$. Consider again $\gamma(t)=\exp_I(Pt)$. Then the geodesic equations read
\begin{equation*}
P^2\exp_I''(Pt)=-\Gamma\left(P\exp_I'(Pt),P\exp_I'(Pt),\exp_I(Pt)\right)\text{.}
\end{equation*}
Since $P^2=P$ and again letting $t=0$ we get
\begin{equation*}
cP=-\Gamma\left(P,P,I\right)\text{,}
\end{equation*}
where $c$ is trivially the same constant as determined above for $\Gamma\left(I,I,I\right)$. Now suppose that we have two mutually orthogonal projections $P_1, P_2$ such that $P_1P_2=0$. Then we have for the projection $P_1+P_2$ using linearity of $\Gamma(\cdot,\cdot,\cdot)$ in the first two variables that
\begin{equation*}
\begin{split}
&\Gamma\left(P_1,P_1,I\right)+\Gamma\left(P_2,P_2,I\right)=-c(P_1+P_2)=\Gamma\left(P_1+P_2,P_1+P_2,I\right)=\\
&=\Gamma\left(P_1,P_1,I\right)+\Gamma\left(P_1,P_2,I\right)+\Gamma\left(P_2,P_1,I\right)+\Gamma\left(P_2,P_2,I\right)\text{,}
\end{split}
\end{equation*}
which yields that for mutually orthogonal projections $P_1, P_2$ we get the orthogonality relation
\begin{equation*}
\Gamma\left(P_1,P_2,I\right)=0\text{.}
\end{equation*}

Finally since a diagonal $D$ can be written as $D=\sum_i\lambda_iP_i$ for mutually orthogonal projections $P_i$, we have
\begin{equation*}
\begin{split}
&\Gamma\left(D,D,I\right)=\Gamma\left(\sum_i\lambda_iP_i,\sum_i\lambda_iP_i,I\right)=\\
&=\sum_i\lambda_i^2\Gamma\left(P_i,P_i,I\right)=-\sum_i\lambda_i^2cP_i=\\
&=-cD^2\text{,}
\end{split}
\end{equation*}
which is what needed to be shown.

\end{proof}

The above three theorems with the other preceeding results presented here, lead us to the concluding

\begin{theorem}\label{classify}
All affine matrix means $M_t(X,Y)$ are of the form
\begin{equation}\label{means2}
M(X,Y)=\begin{cases}
    X^{1/2}\left[(1-t)I+t\left(X^{-1/2}YX^{-1/2}\right)^{1-\kappa}\right]^{\frac{1}{1-\kappa}}X^{1/2}&\text{if $\kappa\neq1$,}\\
    X^{1/2}\left(X^{-1/2}YX^{-1/2}\right)^{t}X^{1/2}&\text{if $\kappa=1$,}
    \end{cases}
\end{equation}
where $0\leq\kappa\leq 2$. The symmetric affine connections corresponding to these means are
\begin{equation}\label{affineclassify}
\nabla_{X_p}Y_p=DY[p][X_p]-\frac{\kappa}{2}\left(X_pp^{-1}Y_p+Y_pp^{-1}X_p\right)\text{.}
\end{equation}
\end{theorem}
\begin{proof}
By Proposition~\ref{transform1},~\ref{transform2} and Theorem~\ref{transform3} we have that the functions $\Gamma(\cdot,\cdot,\cdot):\textit{H}(n,\mathbb{C})\times \textit{H}(n,\mathbb{C})\times \textit{P}(n,\mathbb{C})\mapsto \textit{H}(n,\mathbb{C})$ representing the Christoffel symbols are of the form
\begin{equation}\label{}
\Gamma(X,X,p)=-cXp^{-1}X\text{.}
\end{equation}
This formula determines the functions that are the symmetric parts of the possible connections, and these connections have geodesics determined by Theorem~\ref{geodesics} in the form \eqref{means2}. Again by Propostion~15 and Corollary~16 of Chapter~6 in \cite{spivak} we know that connections which have the same torsion and geodesics are identical and for an arbitrary connection there is a unique connection with vanishing torsion and with the same geodesics. So in other words since the connections \eqref{affineclassify} are symmetric, affine and have the same geodesics, therefore they give \textit{the} sought symmetric connections for each $\kappa$ if we choose $c=\kappa$.

The corresponding geodesics are given in \eqref{means}, and these are matrix means if and only if $\kappa\in [0,2]$, since the representing functions $f(t)$ in \eqref{mean} turn out to be operator monotone only in these cases due to Example~\ref{ex1}.

\end{proof}

The above result gives us the complete classification of affine matrix means. So now we can concetrate only on the connections \eqref{affineclassify}. In the next section we solve the metrization problem of these connections.

\section{The holonomy groups and metrizability of the affine family}
Let $W$ be a smooth connected manifold with an affine connection $\nabla$. The holonomy group $\mathcal{H}_p(\nabla)$ of the connection $\nabla$ at point $p\in W$ is defined to be the set of all linear automorphisms of the tangent space $T_pW$ at $p$ induced by parallel transports along $p$ based closed rectificable curves. If $W$ is simply connected then $\mathcal{H}_p(\nabla)$ is known to be a Lie subgroup of $\End(T_pW)$ \cite{merkulovschwachhofer}. In case of non-simply connectedness the restricted holonomy group $\hat{\mathcal{H}}_p(\nabla)$ is defined as the normal subgroup of $\mathcal{H}_p(\nabla)$ which is induced by closed rectificable curves homotopic to zero, see Chapter II Section 4 in \cite{kobayashi1} for more detailed information. Let $\mathfrak{h}_p(\nabla)$ and $\hat{\mathfrak{h}}_p(\nabla)$ denote the Lie algebra of $\mathcal{H}_p(\nabla)$ and $\hat{\mathcal{H}}_p(\nabla)$ respectively. The holonomy group $\mathcal{H}_p(\nabla)$ is known to be an invariant of the connected manifold $W$, since $\mathcal{H}_p(\nabla)$ is conjugate to every other $\mathcal{H}_q(\nabla)$ by parallel transports.

Now suppose that the connection $\nabla$ is real analytic. Then by Theorem~10.8 of Chapter II and Theorem~9.2 of Chapter III in \cite{kobayashi1}, $\hat{\mathfrak{h}}_p(\nabla)$ is generated by the successive covariant differentials $\nabla^rR$, $r=0,1,2,\ldots$ at the point $p$ where $R(X,Y)$ denotes the curvature endomorphism of the connection $\nabla$. This is a version of Ambrose-Singer's theorem of Kobayashi-Nomizu. The curvature tensor $R$ is defined as
\begin{equation*}
R(X,Y)Z=\nabla_X\nabla_YZ-\nabla_Y\nabla_XZ-\nabla_{[X,Y]}Z
\end{equation*}
or expressed in local coordinate system with the Christoffel symbols $\Gamma^{i}_{jk}$ as
\begin{equation}\label{curvature}
R^{i}_{jkl}=\frac{\partial \Gamma^{i}_{lj}}{\partial x^k}-\frac{\partial \Gamma^{i}_{kj}}{\partial x^l}+\Gamma^{i}_{km}\Gamma^{m}_{lj}-\Gamma^{i}_{lm}\Gamma^{m}_{kj}.
\end{equation}

Suppose now that the connection $\nabla$ is torsion-free, i.e.
\begin{equation}
\nabla_XY-\nabla_YX-[X,Y]=0
\end{equation}
for all vector fields $X,Y$ or equivalently $\Gamma^{i}_{jk}=\Gamma^{i}_{kj}$ everywhere. Then $\nabla$ is the Levi-Civita connection of a Riemannian metric if and only if the corresponding holonomy group $\hat{\mathcal{H}}_p(\nabla)$ is a compact Lie group. More generally there exists a non-degenerate $\nabla$ invariant bilinear form $\left\langle \cdot,\cdot\right\rangle_p$ if and only if $\hat{\mathcal{H}}_p(\nabla)$ leaves $\left\langle \cdot,\cdot\right\rangle_p$ invariant.

In \cite{merkulovschwachhofer} all possible irreducible holonomy groups of torsion-free affine connections are classified, so in principle we know what kind of groups can occur, at least in the reducible case. Again we are interested in the connections
\begin{equation}\label{affineclassify2}
\nabla_{X_p}Y_p=DY[p][X_p]-\frac{\kappa}{2}\left(X_pp^{-1}Y_p+Y_pp^{-1}X_p\right).
\end{equation}
These connections are real analytic, torsion-free and the corresponding manifold $\textit{P}(n,\mathbb{C})$ is analytic simply connected. So to answer the question of metrizability we have to determine the holonomy groups $\hat{\mathcal{H}}_p(\nabla)$.

In our case it turns out that
\begin{equation}\label{covdiffcurv}
\begin{split}
\Gamma^{i}_{jk}E_i=&-\frac{\kappa}{2}(E_jp^{-1}E_k+E_kp^{-1}E_j)\\
R^{i}_{jkl}E_i=&\left(\frac{\kappa}{2}-\frac{\kappa^2}{4}\right)p\left[p^{-1}E_j,\left[p^{-1}E_k,p^{-1}E_l\right]\right],
\end{split}
\end{equation}
where the $E_i$ form the standard basis of the vector space of $\textit{H}(n,\mathbb{C})$ and $[\cdot,\cdot]$ is the commutator. Note that the tangent space is $\textit{H}(n,\mathbb{C})$, so the left hand sides are in $\textit{H}(n,\mathbb{C})$. In order to determine which of these manifolds are symmetric spaces it is sufficient to calculate the covariant differential $R^{s}_{jkl;m}$, since it vanishes everywhere if and only if the underlying manifold is a symmetric space \cite{helgason}. Given the basis $E_i$ for $\textit{H}(n,\mathbb{C})$ we have the identities
\begin{equation*}
\begin{split}
R(X,Y;A_1,\ldots,A_r)Z=&\sum_{i,j,k,l_1,\ldots,l_r,m}R^{i}_{jkm;l_1,\ldots,l_r}E_iX^{j}Y^{k}Z^{m}A_1^{l_1}\cdots A_1^{l_r}\\
R^{i}_{jkm;l_1,\ldots,l_{r+1}}=&\frac{\partial}{\partial x^{l_{r+1}}}R^{i}_{jkm;l_1,\ldots,l_r}+\Gamma^{i}_{sl_{r+1}}R^{s}_{jkm;l_1,\ldots,l_r}\\
&-\sum_{\mu}\Gamma^{s}_{l_{\mu}l_{r+1}}R^{i}_{jkm;l_1,\ldots,s,\ldots,l_r}
\end{split}
\end{equation*}
where indices after $;$ denote covariant differentiation. Now we prove an analogue of Lemma 1 given in the proof of Theorem 9.2 of Chapter III \cite{kobayashi1}.

\begin{theorem}\label{thmsubsequentdiff}
Let the smooth connected manifold $\textit{P}(n,\mathbb{C})$ be equipped with real analytic connection $\nabla$ and curvature tensor given by \eqref{covdiffcurv} with $\kappa\in\mathbb{R}$. Then
\begin{equation}\label{subsequentdiff}
R(X,Y;A_1,\ldots,A_r)Z=(1-\kappa)^{r}D(R(X,Y)Z)[p][A_1,\ldots,A_r]
\end{equation}
where $D(R(X,Y)Z)[p][A_1,\ldots,A_r]$ denotes the r-th Fr\'echet differential of the map $R(X,Y)Z$ at the point $p\in\textit{P}(n,\mathbb{C})$ in the directions $A_i\in\textit{H}(n,\mathbb{C})$.
\end{theorem}
\begin{proof}
The proof is based on writing $R(X,Y)Z$ and its subsequent covariant differentials in essentially two equivalent ways. First of all note that
\begin{equation}\label{invdiff}
\frac{\partial}{\partial x^{i}}p^{-1}=D(x^{-1})[p][E_i]=-p^{-1}E_ip^{-1},
\end{equation}
so the differential operator $\frac{\partial}{\partial x^{i}}$ is equivalent to Fr\'echet differentiation at $p$ in the direction of $E_i$, also
\begin{equation}\label{curvendo}
\begin{split}
R(X,Y)Z&=\left(\frac{\kappa}{2}-\frac{\kappa^2}{4}\right)p\left[p^{-1}Z,\left[p^{-1}X,p^{-1}Y\right]\right]\\
&=\left(\frac{\kappa}{2}-\frac{\kappa^2}{4}\right)\left\{Z\left[p^{-1}X,p^{-1}Y\right]+\left[Yp^{-1},Xp^{-1}\right]Z\right\}\\
&=\left(\frac{\kappa}{2}-\frac{\kappa^2}{4}\right)\left[Zp^{-1},\left[Xp^{-1},Yp^{-1}\right]\right]p.
\end{split}
\end{equation}
Using index-less notation and the linearity of $R(X,Y;A_1,\ldots,A_{r})Z$ we have
\begin{equation}\label{subsequentcodiff}
\begin{split}
&R(X,Y;A_1,\ldots,A_{r+1})Z=\nabla_{A_{r+1}}(R(X,Y;A_1,\ldots,A_r)Z)\\
&-\frac{\kappa}{2}\left\{A_{r+1}p^{-1}R(X,Y;A_1,\ldots,A_r)Z+R(X,Y;A_1,\ldots,A_r)Zp^{-1}A_{r+1}\right.\\
&-R(A_{r+1}p^{-1}X+Xp^{-1}A_{r+1},Y;A_1,\ldots,A_r)Z\\
&-R(X,A_{r+1}p^{-1}Y+Yp^{-1}A_{r+1};A_1,\ldots,A_r)Z\\
&-R(X,Y;A_1,\ldots,A_r)(A_{r+1}p^{-1}Z+Zp^{-1}A_{r+1})\\
&\left.-\sum_{i=1}^rR(X,Y;A_1,\ldots,A_{r+1}p^{-1}A_i+A_ip^{-1}A_{r+1},\ldots,A_r)Z\right\}.
\end{split}
\end{equation}
Again the first term in the above equation is equivalent to
\begin{equation}\label{partial1}
(A_{r+1})^{s}\frac{\partial}{\partial x^{s}}R(X,Y;A_1,\ldots,A_r)Z=D(R(X,Y;A_1,\ldots,A_r)Z)[p][A_{r+1}].
\end{equation}
\begin{claim}
$R(X,Y;A_1,\ldots,A_r)Z$ is the linear combination of terms $pS$, where $S$ is some word which is a product of the terms $p^{-1}X,p^{-1}Y,p^{-1}A_1,\ldots,p^{-1}A_r$ of the first order.
\end{claim}
We prove by induction. For $r=0$ it clearly holds by the first equality in \eqref{curvendo}. Suppose that it holds for some $r$. Then by \eqref{subsequentcodiff} it is easy to see that it holds for $r+1$, due to \eqref{invdiff}, the linearity of $R(X,Y;A_1,\ldots,A_r)Z$ and the product rule of Fr\'echet differentiation. The claim is proved.

By the claim $R(X,Y;A_1,\ldots,A_r)Z$ is the linear combination of terms $pS$, therefore by linearity, \eqref{partial1} and \eqref{invdiff} we have
\begin{equation}
\begin{split}
&\nabla_{A_{r+1}}(R(X,Y;A_1,\ldots,A_r)Z)=A_{r+1}p^{-1}R(X,Y;A_1,\ldots,A_r)Z\\
&-R(A_{r+1}p^{-1}X,Y;A_1,\ldots,A_r)Z-R(X,A_{r+1}p^{-1}Y;A_1,\ldots,A_r)Z\\
&-R(X,Y;A_1,\ldots,A_r)(A_{r+1}p^{-1}Z)-\sum_{i=1}^rR(X,Y;A_1,\ldots,A_{r+1}p^{-1}A_i,\ldots,A_r)Z.
\end{split}
\end{equation}
Combining the above we arrive at a version of \eqref{subsequentcodiff}:
\begin{equation*}
\begin{split}
&R(X,Y;A_1,\ldots,A_{r+1})Z=\\
&=\left(1-\frac{\kappa}{2}\right)\left\{A_{r+1}p^{-1}R(X,Y;A_1,\ldots,A_r)Z-R(A_{r+1}p^{-1}X,Y;A_1,\ldots,A_r)Z\right.\\
&-R(X,A_{r+1}p^{-1}Y;A_1,\ldots,A_r)Z-R(X,Y;A_1,\ldots,A_r)(A_{r+1}p^{-1}Z)\\
&\left.-\sum_{i=1}^rR(X,Y;A_1,\ldots,A_{r+1}p^{-1}A_i,\ldots,A_r)Z\right\}\\
&-\frac{\kappa}{2}\left\{R(X,Y;A_1,\ldots,A_r)Zp^{-1}A_{r+1}-R(Xp^{-1}A_{r+1},Y;A_1,\ldots,A_r)Z\right.\\
&-R(X,Yp^{-1}A_{r+1};A_1,\ldots,A_r)Z-R(X,Y;A_1,\ldots,A_r)(Zp^{-1}A_{r+1})\\
&\left.-\sum_{i=1}^rR(X,Y;A_1,\ldots,A_ip^{-1}A_{r+1},\ldots,A_r)Z\right\},
\end{split}
\end{equation*}
which is equivalent to
\begin{equation}\label{subsequentcodiff2}
\begin{split}
&R(X,Y;A_1,\ldots,A_{r+1})Z=\\
&=(1-\kappa)\left\{A_{r+1}p^{-1}R(X,Y;A_1,\ldots,A_r)Z-R(A_{r+1}p^{-1}X,Y;A_1,\ldots,A_r)Z\right.\\
&-R(X,A_{r+1}p^{-1}Y;A_1,\ldots,A_r)Z-R(X,Y;A_1,\ldots,A_r)(A_{r+1}p^{-1}Z)\\
&\left.-\sum_{i=1}^rR(X,Y;A_1,\ldots,A_{r+1}p^{-1}A_i,\ldots,A_r)Z\right\}\\
&+\frac{\kappa}{2}\left\{p\left[p^{-1}A_{r+1},p^{-1}R(X,Y;A_1,\ldots,A_r)Z\right]\right.\\
&-R(p[p^{-1}A_{r+1},p^{-1}X],Y;A_1,\ldots,A_r)Z\\
&-R(X,p[p^{-1}A_{r+1},p^{-1}Y];A_1,\ldots,A_r)Z-R(X,Y;A_1,\ldots,A_r)(p[p^{-1}A_{r+1},p^{-1}Z])\\
&\left.-\sum_{i=1}^rR(X,Y;A_1,\ldots,p[p^{-1}A_{r+1},p^{-1}A_i],\ldots,A_r)Z\right\}.
\end{split}
\end{equation}
Now we can reverse the claim and using the exactly the same argument starting with the third equality in \eqref{curvendo} we can prove that $R(X,Y;A_1,\ldots,A_r)Z$ is the linear combination of terms $Sp$, where $S$ is some word which is a product of the terms $Xp^{-1},Yp^{-1},A_1p^{-1},\ldots,A_p^{-1}r$ of the first order. Similarly we end up with
\begin{equation}\label{subsequentcodiff22}
\begin{split}
&R(X,Y;A_1,\ldots,A_{r+1})Z=\\
&=(1-\kappa)\left\{R(X,Y;A_1,\ldots,A_r)Zp^{-1}A_{r+1}-R(Xp^{-1}A_{r+1},Y;A_1,\ldots,A_r)Z\right.\\
&-R(X,Yp^{-1}A_{r+1};A_1,\ldots,A_r)Z-R(X,Y;A_1,\ldots,A_r)(Zp^{-1}A_{r+1})\\
&\left.-\sum_{i=1}^rR(X,Y;A_1,\ldots,A_ip^{-1}A_{r+1},\ldots,A_r)Z\right\}\\
&+\frac{\kappa}{2}\left\{\left[R(X,Y;A_1,\ldots,A_r)Zp^{-1},A_{r+1}p^{-1}\right]p\right.\\
&-R([Xp^{-1},A_{r+1}p^{-1}]p,Y;A_1,\ldots,A_r)Z\\
&-R(X,[Yp^{-1},A_{r+1}p^{-1}]p;A_1,\ldots,A_r)Z-R(X,Y;A_1,\ldots,A_r)([Zp^{-1},A_{r+1}p^{-1}]p)\\
&\left.-\sum_{i=1}^rR(X,Y;A_1,\ldots,[A_ip^{-1},A_{r+1}p^{-1}]p,\ldots,A_r)Z\right\}.
\end{split}
\end{equation}
Now subtracting \eqref{subsequentcodiff22} from \eqref{subsequentcodiff2} and using the fact that
\begin{equation*}
[Ap^{-1},Bp^{-1}]p=p[p^{-1}A,p^{-1}B]=-p[p^{-1}B,p^{-1}A]
\end{equation*}
for any $A,B\in\textit{H}(n,\mathbb{C})$ and linearity of $R(X,Y;A_1,\ldots,A_{r})Z$, we get that
\begin{equation}
\begin{split}
&R(X,Y;A_1,\ldots,A_{r+1})Z-R(X,Y;A_1,\ldots,A_{r+1})Z=0=\\
&=p\left[p^{-1}A_{r+1},p^{-1}R(X,Y;A_1,\ldots,A_r)Z\right]-R(p[p^{-1}A_{r+1},p^{-1}X],Y;A_1,\ldots,A_r)Z\\
&-R(X,p[p^{-1}A_{r+1},p^{-1}Y];A_1,\ldots,A_r)Z-R(X,Y;A_1,\ldots,A_r)(p[p^{-1}A_{r+1},p^{-1}Z])\\
&-\sum_{i=1}^rR(X,Y;A_1,\ldots,p[p^{-1}A_{r+1},p^{-1}A_i],\ldots,A_r)Z.
\end{split}
\end{equation}
So in particular \eqref{subsequentcodiff2} is just
\begin{equation*}
\begin{split}
&R(X,Y;A_1,\ldots,A_{r+1})Z=\\
&=(1-\kappa)\left\{A_{r+1}p^{-1}R(X,Y;A_1,\ldots,A_r)Z-R(A_{r+1}p^{-1}X,Y;A_1,\ldots,A_r)Z\right.\\
&-R(X,A_{r+1}p^{-1}Y;A_1,\ldots,A_r)Z-R(X,Y;A_1,\ldots,A_r)(A_{r+1}p^{-1}Z)\\
&\left.-\sum_{i=1}^rR(X,Y;A_1,\ldots,A_{r+1}p^{-1}A_i,\ldots,A_r)Z\right\}\\
&=(1-\kappa)p\left\{p^{-1}A_{r+1}p^{-1}R(X,Y;A_1,\ldots,A_r)Z-p^{-1}R(A_{r+1}p^{-1}X,Y;A_1,\ldots,A_r)Z\right.\\
&-p^{-1}R(X,A_{r+1}p^{-1}Y;A_1,\ldots,A_r)Z-p^{-1}R(X,Y;A_1,\ldots,A_r)(A_{r+1}p^{-1}Z)\\
&\left.-\sum_{i=1}^rp^{-1}R(X,Y;A_1,\ldots,A_{r+1}p^{-1}A_i,\ldots,A_r)Z\right\}.
\end{split}
\end{equation*}
Considering again \eqref{invdiff} and the first claim we get that the above is equivalent to
\begin{equation*}
R(X,Y;A_1,\ldots,A_{r+1})Z=(1-\kappa)D(R(X,Y;A_1,\ldots,A_{r}))[p][A_{r+1}],
\end{equation*}
which is just \eqref{subsequentdiff}.

\end{proof}

Now again the Lie algebra $\hat{\mathfrak{h}}_p(\nabla)$ is generated by the endomosphisms $\nabla^rR$. This means that the generated algebra grows as $r$ increases and after some finitely many steps it stabilizes and taking higher covariant derivatives of $R$ is unnecessary. Since the manifold $\textit{P}(n,\mathbb{C})$ is simply connected the holonomy group and the restricted holonomy group coincide, so $\hat{\mathfrak{h}}_p(\nabla)=\mathfrak{h}_p(\nabla)$. By the second formula in \eqref{curvendo} and \eqref{subsequentdiff} we have the following

\begin{corollary}\label{representation}
The Lie algebra $\mathfrak{h}_p(\nabla)$ is faithfully represented over the vector space $V=\textit{H}(n,\mathbb{C})$ (or $V=\textit{H}(n,\mathbb{R})$) with $\rho:\mathfrak{h}_p(\nabla)\mapsto\End(V)$ given as
\begin{equation}\label{action}
\rho(W)Z=ZW+W^{*}Z
\end{equation}
for $W\in\mathfrak{h}_p(\nabla)$ and $Z\in\textit{H}(n,\mathbb{C})$ (or $Z\in\textit{H}(n,\mathbb{R})$).
\end{corollary}

We are in position to do a case by case analysis for different values of $\kappa$. $\mathfrak{so}(n,\mathbb{R})$ denotes the Lie algebra of skew-symmetric n-by-n matrices over the real field $\mathbb{R}$, $\mathfrak{su}(n,\mathbb{C})$ denotes the Lie algebra of skew-Hermitian matrices with vanishing trace over $\mathbb{C}$, $\mathfrak{sl}(n,\mathbb{F})$ denotes the Lie algebra of traceless matrices over the field $\mathbb{F}$.

\begin{theorem}\label{holonomies}
Let the smooth connected manifold $\textit{P}(n,\mathbb{C})$ with tangent space $\textit{H}(n,\mathbb{C})$ be equipped with real analytic connection
\begin{equation}\label{affineclassify3}
\nabla_{X_p}Y_p=DY[p][X_p]-\frac{\kappa}{2}\left(X_pp^{-1}Y_p+Y_pp^{-1}X_p\right)
\end{equation}
with $\kappa\in\mathbb{R}$. Then the holonomy algebra $\mathfrak{h}_p(\nabla)$ is as follows:
\begin{equation}
\mathfrak{h}_p(\nabla)=
    \begin{cases}
    \text{the trivial algebra}&\text{if $\kappa=0,2$,}\\
    \mathfrak{su}(n,\mathbb{C})&\text{if $\kappa=1$,}\\
    \mathfrak{sl}(n,\mathbb{C})&\text{else.}
    \end{cases}
\end{equation}
In the case of the submanifold $\textit{P}(n,\mathbb{R})$ with tangent space $\textit{H}(n,\mathbb{R})$ we have
\begin{equation}
\mathfrak{h}_p(\nabla)=
    \begin{cases}
    \text{the trivial algebra}&\text{if $\kappa=0,2$,}\\
    \mathfrak{so}(n,\mathbb{R})&\text{if $\kappa=1$,}\\
    \mathfrak{sl}(n,\mathbb{R})&\text{else.}
    \end{cases}
\end{equation}
\end{theorem}
\begin{proof}
By the conjugate invariancy of $\mathcal{H}_p(\nabla)$ it is enough to consider the case when $p=I$.

Suppose $\kappa=0,2$. Then the curvature \eqref{covdiffcurv} of the connection vanishes, so $\mathfrak{h}_p(\nabla)$ is the trivial algebra.

Suppose $\kappa=1$. Then the curvature \eqref{covdiffcurv} is nonzero, but is covariantly constant, all first and higher order covariant derivatives vanish due to Theorem~\ref{thmsubsequentdiff}. Therefore the manifold is a symmetric space that is very well known and the algebra $\mathfrak{h}_p(\nabla)$ by \eqref{curvendo} is generated by elements of the form $[X,Y]$ where $X,Y\in\textit{H}(n,\mathbb{F})$. We have for all $[X,Y]=W\in\mathfrak{h}_p(\nabla)$ that
\begin{equation*}
W^*=[X,Y]^*=-[X^*,Y^*]=-[X,Y]
\end{equation*}
where $^*$ can be replaced by the transpose $^T$ over $\mathbb{F}=\mathbb{R}$. Also since $TrW=Tr[X,Y]=0$ we have $\mathfrak{h}_p(\nabla)=\mathfrak{so}(n,\mathbb{R})$ if $\mathbb{F}=\mathbb{R}$ and $\mathfrak{h}_p(\nabla)=\mathfrak{su}(n,\mathbb{C})$ if $\mathbb{F}=\mathbb{C}$.

Suppose $\kappa\neq 0,1,2$. Then by Theorem~\ref{thmsubsequentdiff} the higher order covariant derivatives $\nabla^rR$ (as we will see immediately) no longer vanish. Let $W\in\mathfrak{h}_p(\nabla)$. Then by \eqref{curvendo}, \eqref{subsequentdiff} and Corollary~\ref{representation}
\begin{equation*}
W=\left(\frac{\kappa}{2}-\frac{\kappa^2}{4}\right)(1-\kappa)^{r}D([p^{-1}X,p^{-1}Y])[p][A_1,\ldots,A_r],
\end{equation*}
where $X,Y\in\textit{H}(n,\mathbb{F})$. I.e. $W$ is given by the linear combination of commutators of some n-by-n matrices over the field $\mathbb{F}$, so $TrW=0$. This tells us that
\begin{equation}\label{slsubalgebra}
\mathfrak{h}_p(\nabla)\subseteq\mathfrak{sl}(n,\mathbb{F}).
\end{equation}
Now we will show that the generated algebra already stabilizes for $r=1$. Without loss of generality we can assume that $p=I$. Then we have to consider the generators of the form
\begin{equation}\label{generators}
G=D([p^{-1}X,p^{-1}Y])[I][A_1]=-[A_1X,Y]-[X,A_1Y].
\end{equation}
Let
\begin{equation*}
\begin{split}
E^+_{ik}=&
    \begin{cases}
    E_{ik}+E_{ki}&\text{if $i\neq k$,}\\
    E_ii&\text{if $i=k$,}
    \end{cases}\\
E^-_{ik}=&E_{ik}-E_{ki}
\end{split}
\end{equation*}
where $E_{ik}$ is the matrix with zero entries excluding the $(ik)$ entry which is $1$. Then $E^+_{ik}$ form a basis of $\textit{H}(n,\mathbb{R})$ and $E^-_{ik}$ form a basis of the vector space of skew-Hermitian matrices $\textit{SH}(n,\mathbb{R})$ over the real field $\mathbb{R}$. The vector space $\textit{SH}(n,\mathbb{C})$ is defined similarly over $\mathbb{C}$. Note also that $\textit{H}(n,\mathbb{C})\cong\textit{H}(n,\mathbb{R})\oplus\textit{SH}(n,\mathbb{R})$ and that $E^+_{ik}E^-_{lm}=0$ in general. Suppose that $A_1=E^+_{iz}$, $X=E^+_{ky}$ and $Y=E^+_{ik}$. Then by \eqref{generators}
\begin{equation*}
G=-E^+_{iz}E^+_{ky}E^+_{ik}+E^+_{ik}E^+_{iz}E^+_{ky}-E^+_{ky}E^+_{iz}E^+_{ik}+E^+_{iz}E^+_{ik}E^+_{ky}.
\end{equation*}
Using that $E^+_{ik}=E^+_{ki}$ and imposing restrictions $z\neq k$ and $y\neq i$ we get that
\begin{equation*}
G=
    \begin{cases}
    E_{zy}&\text{if $z\neq y$,}\\
    E_{zz}-E_{ii}&\text{if $y=z$.}
    \end{cases}
\end{equation*}
So the matrices $G$ of this form span the whole $\mathfrak{sl}(n,\mathbb{R})$, i.e. considering \eqref{slsubalgebra} we have $\mathfrak{h}_p(\nabla)=\mathfrak{sl}(n,\mathbb{R})$ if $\mathbb{F}=\mathbb{R}$. Similarly if $A_1=E^-_{iz}$, $X=E^-_{ky}$ and $Y=E^-_{ik}$, then we get the same generator $G$, so $\mathfrak{h}_p(\nabla)=\mathfrak{sl}(n,\mathbb{R})\oplus\mathfrak{sl}(n,\mathbb{R})$ due to $E^+_{ik}E^-_{lm}=0$ if $\mathbb{F}=\mathbb{C}$; that is $\mathfrak{h}_p(\nabla)=\mathfrak{sl}(n,\mathbb{C})$ in the complex case.

\end{proof}

By the proof of the previous Theorem~\ref{holonomies} we see that $R^{s}_{jkl;m}=0$ everywhere if and only if $\kappa=0,1,2$. This proves the following
\begin{corollary}
The only matrix means which are affine means corresponding to symmetric spaces are the arithmetic, harmonic and geometric means.
\end{corollary}
Since we know the holonomy groups we can decide their metrizability.

\begin{corollary}\label{metrizability}
The affine connections \eqref{affineclassify3} are metric in the following cases:
\begin{enumerate}
	\item $n=1,2$, $\kappa$ arbitrary, $\mathbb{F}=\mathbb{R}$ or $\mathbb{C}$,
	\item $n\geq 3$, $\kappa=0,1,2$, $\mathbb{F}=\mathbb{R}$ or $\mathbb{C}$.
\end{enumerate}
\end{corollary}
\begin{proof}
The case $n=1$ is trivial. In \cite{merkulovschwachhofer} all irreducible holonomies of affine connections are classified and metrizability is also dicussed. The metric connections were classified by Berger long ago. The holonomy $\mathfrak{sl}(2,\mathbb{R})$ is isomorphic to $\mathfrak{sp}(2,\mathbb{R})$ which is metric, there exists an invariant symplectic form. Also $\mathfrak{sl}(2,\mathbb{C})$ is isomorphic to $\mathfrak{sp}(2,\mathbb{C})$ there exists an invariant symplectic form. This isomorphic correspondence fails in higher dimensions $n\geq 3$, where the holonomies $\mathfrak{sl}(n,\mathbb{F})$ ($\mathbb{F}=\mathbb{R}$ or $\mathbb{C}$) with representation over $\textit{H}(n,\mathbb{F})$ is non-metric.

\end{proof}

\begin{remark}
In the second case in Corollary~\ref{metrizability} although there exists no metric structure, however by inspecting the holonomy group $\mathcal{H}_p(\nabla)$ we get that there exist totally geodesic flat submanifolds. That is if we consider the subset $\textit{D}(n,\mathbb{F})$ of diagonal matrices of $\textit{H}(n,\mathbb{F})$ in both cases $\mathbb{F}=\mathbb{R}$ or $\mathbb{C}$, we get a totally geodesic Euclidean submanifold and a Riemannian metric on $\textit{D}(n,\mathbb{F})$ is given in the form
\begin{equation*}
Tr\left\{p^{-2\kappa}\log^2_p(q)\right\}
\end{equation*}
where $\log_p(q)$ is the logarithm map given in \eqref{mappingsaff}.
\end{remark}

So there exist no previously unknown affine matrix mean which correspond to a Riemannian manifold. Although we have found a previously unknown, generally non-metrizable, one parameter family of affinely connected manifolds where the points of the geodesics are matrix means, in particular matrix power means.



\end{document}